\documentclass[reqno]{amsart}

\usepackage{amscd,amssymb, mathrsfs}
\usepackage[cmtip,arrow,curve,matrix]{xy}

\usepackage{url}
\usepackage[headings]{fullpage}

\numberwithin{equation}{section}
\DeclareFontFamily{OMS}{rsfs}{\skewchar\font'60}
\DeclareFontShape{OMS}{rsfs}{m}{n}{<-5>rsfs5 <5-7>rsfs7 <7->rsfs10 }{}
\DeclareSymbolFont{rsfs}{OMS}{rsfs}{m}{n}
\DeclareSymbolFontAlphabet{\scr}{rsfs}

\newtheorem{Theorem}[equation]{Theorem}

\newtheorem{Lemma}[equation]{Lemma}
\newtheorem{Proposition}[equation]{Proposition}
\newtheorem{Corollary}[equation]{Corollary}

\theoremstyle{definition}

\newtheorem{Definition}[equation]{Definition}
\newtheorem{Remark}[equation]{Remark}

\newcommand{\alg}[2]{#2[#1]} \newcommand{\ainfty}{A_{\infty}}
\DeclareMathOperator{\Alg}{Alg}
\newcommand{\ainftyspaces}{\CatOf{$\ainfty$ spaces}}
\newcommand{\ainftyspectra}{\spectra^{\ainfty}}

\newcommand{\A}{\mathbb{A}}
\newcommand{\E}{\mathbb{E}}
\DeclareMathOperator{\Aut}{Aut}
\DeclareMathOperator{\End}{End}

\newcommand{\BL}[1]{B_{\mathcal{L}}#1}

\newcommand{\CatOf}[1]{(\mbox{#1})}

\newcommand{\rsps}[1]{\uln{#1}}
\newcommand{\cohps}[2]{#1_{#2}}
\DeclareMathOperator*{\colim}{colim}
\newcommand{\cover}[1]{\mathfrak{#1}}

\newcommand{\eqdef}{\overset{\text{def}}{=}}
\newcommand{\einfty}{E_{\infty}}
\newcommand{\einftyspectra}{\spectra[\einfty]}

\newcommand{\EL}[1]{E_{\mathcal{L}}#1}
\newcommand{\EKMM}{\scr{M}}

\newcommand{\GLsym}{GL_{1}}
\newcommand{\GL}[1]{\GLsym #1}

\newcommand{\glsym}{gl_{1}}
\newcommand{\gl}[1]{\glsym #1}

\newcommand{\gplike}[2]{\alg{#1}{#2}^{\times}}
\newcommand{\gplikeainftyspaces}{\ainftyspaces^{\times}}

\DeclareMathOperator{\ho}{ho}

\newcommand{\heq}{\simeq}
\DeclareMathOperator{\Ho}{ho}
\renewcommand{\i}{\infty}

\renewcommand{\smash}{\wedge}
\renewcommand{\SS}{\mathcal{S}}

\DeclareMathOperator{\id}{id}
\newcommand{\iso}{\cong}

\newcommand{\linf}{\Omega^{\infty}}

\newcommand{\Lin}{\mathcal{L}}
\renewcommand{\L}{\mathbb{L}}

\newcommand{\monspectra}{\mathcal{M}}
\newcommand{\mr}{\rho}

\DeclareMathOperator{\Open}{Open}
\newcommand{\plus}{+}
\newcommand{\pt}[1]{#1_{\plus}}
\newcommand{\ptC}{C_{*}}
\newcommand{\ptspace}{\ast}
\newcommand{\ptdspaces}{\spaces_{*}}

\newcommand{\Qeq}{\approx}

\newcommand{\restr}[1]{|_{#1}}
\newcommand{\Rcf}{R^{\circ}}
\newcommand{\Rmod}{\Mod{R}}
\newcommand{\RmodT}{\Rmod_{\monspectra}}

\newcommand{\lRmodules}{\CatOf{right $\linf R$-modules}}
\newcommand{\Rmodules}{\CatOf{right $R$-modules}}

\newcommand{\slot}{\,-\,}
\newcommand{\spaces}{\TT}
\newcommand{\spectra}{\mathscr{S}} \newcommand{\sinf}{\Sigma^{\infty}}
\newcommand{\symmspectra}{\mathcal{S}_{\Sigma}}
\newcommand{\splus}{\sinf_{+}} 
\newcommand{\Smash}{\wedge} \newcommand{\SmashL}{\wedge_{\mathcal{L}}}
\newcommand{\Svee}{\mathbb{S}}

\newcommand{\twistedcohps}[2]{#1_{#2}}

\newcommand{\timesL}{\times_{\mathcal{L}}}
\newcommand{\T}{\mathbb{T}}
\newcommand{\PP}{\mathbb{P}}

\newcommand{\uln}[1]{\underline{#1}}
\newcommand{\universe}[1]{\mathcal{#1}}

\newcommand{\unitscohps}[1]{\uln{#1}^{\times}}

\newcommand{\xra}[1]{\xrightarrow{#1}}

\newcommand{\Z}{\mathbb{Z}}

\let\catsymbfont\mathcal

\newcommand{\C}{\mathscr{C}}

\newcommand{\D}{\mathscr{D}}
\newcommand{\N}{\mathrm{N}}
\newcommand{\TT}{\mathscr{T}}

\newcommand{\aC}{{\catsymbfont{C}}}
\newcommand{\aM}{\EKMM}
\newcommand{\aU}{{\catsymbfont{U}}}

\newcommand{\sL}{{\mathcal{L}}}
\newcommand{\sma}{\wedge}    
\newcommand{\htp}{\simeq}

\newcommand{\sI}{\scr{I}}

\newcommand{\HTT}[1]{[HTT, #1]}

\DeclareMathOperator{\Fun}{Fun}

\DeclareMathOperator{\map}{map}
\newcommand{\MtoI}[1]{#1^{\circ}}
\DeclareMathOperator{\op}{op}
\DeclareMathOperator{\orient}{or}

\DeclareMathOperator{\Cat}{Cat}

\DeclareMathOperator{\Gpd}{Gpd}

\DeclareMathOperator{\Map}{Map}
\DeclareMathOperator{\Pre}{Pre}
\DeclareMathOperator{\Set}{Set}
\DeclareMathOperator{\Stab}{Stab}

\newcommand{\R}{\mathbb{R}}

\newcommand{\Rorient}{\orient_{R}}

\newcommand{\Mod}[1]{{#1}\text{-}\mathrm{mod}}

\newcommand{\Sing}[1]{\SingOp #1}
\newcommand{\SingOp}{\Pi_{\infty}}
\DeclareMathOperator{\SingTxt}{Sing}
\newcommand{\Line}[1]{{#1}\text{-}\mathrm{line}}

\newcommand{\Triv}[1]{{#1}\text{-}\mathrm{triv}}
\newcommand{\lmap}{F_{\times_{\mathcal{L}}}}

\newcommand{\Rwe}{\Line{R}}
\newcommand{\RweT}{\Rwe_{\monspectra}}
\newcommand{\Rtriv}{\Triv{R}}

\begin{document}

\title{Units of ring spectra and Thom spectra}

\author[Ando]{Matthew Ando}
\address{Department of Mathematics \\
The University of Illinois at Urbana-Champaign \\
Urbana IL 61801 \\
USA} \email{mando@math.uiuc.edu}

\author[Blumberg]{Andrew J. Blumberg}
\address{Department of Mathematics \\Stanford University \\
Stanford CA  94305}
\email{blumberg@math.Stanford.edu}

\author[Gepner]{David Gepner}
\address{Department of Mathematics \\
The University of Illinois at Chicago \\ Chicago IL 60607 \\ USA}
\email{gepner@math.uic.edu}

\author[Hopkins]{Michael J.~Hopkins}
\address{Department of Mathematics \\
Harvard University \\
Cambridge MA 02138}
\email{mjh@math.harvard.edu}

\author[Rezk]{Charles Rezk}

\address{The University of Illinois at Urbana-Champaign \\
Urbana IL 61801 \\
USA}
\email{rezk@math.uiuc.edu}

\date{Version 3.56 October 2009}
\thanks{Ando was supported by NSF grant DMS-0705233.  Blumberg was
supported by an NSF Postdoctoral Research Fellowship.  Gepner was
supported by EPSRC grant EP/C52084X/1.  Hopkins was
supported by the NSF.  Rezk was
supported by NSF grant DMS-0505056.}

\begin{abstract}
We review and extend the theory of Thom spectra and the associated
obstruction theory for orientations.
Specifically, we show that for
an $\einfty$ ring spectrum $A$, the classical construction of
$\gl{A}$, the spectrum of units, is the right adjoint of the functor
%
%
\[
\splus\linf: \Ho \CatOf{connective spectra} \rightarrow \Ho
\CatOf{$\einfty$ ring spectra}.
\]
To a map of spectra
\[
   f: b \to b\gl{A},
\]
we associate an $\einfty$ $A$-algebra Thom spectrum $Mf$, which admits
an $\einfty$
$A$-algebra map to $R$ if and only if the composition
\[
    b \to b\gl{A} \to b\gl{R}
\]
is null; the classical case developed by \cite{MQRT:ersers} arises
when $A$ is the sphere spectrum.  We develop the analogous theory
for $\ainfty$ ring spectra.  If $A$ is an
$\ainfty$ ring spectrum, then to a map of spaces
\[
     f: B \to B\GL{A}
\]
we associate an $A$-module Thom spectrum $Mf,$ which admits an
$R$-orientation if and only if
\[
B \to B\GL{A} \to B\GL{R}
\]
is null.  We note that $B\GL{A}$ classifies the twists of $A$-theory.
We take two different approaches to the $\ainfty$ theory which are of
independent interest. The first involves a rigidified model of
$\ainfty$ (and $\einfty$) spaces, as developed in
\cite{Blumberg-thesis,Blumberg-Cohen-Schlichtkrull}.  The second uses
the theory of $\infty$-categories as described in
\cite{math.CT/0608040} and involves an $\infty$-categorical account of
parametrized spectra.  In order to compare these approaches to one
another and to the classical theory, we characterize the Thom spectrum
functor from the perspective of Morita theory.
\end{abstract}

\maketitle

\tableofcontents

\section{Introduction}
\label{sec:prol-thom-isom}

The study of orientations of vector bundles and spherical fibrations
is a classical topic in algebraic topology, deeply intertwined with
the modern development of the subject.  In his 1970 MIT notes
\cite{MR2162361}, Sullivan observed that to describe the
obstructions to orientability and to classify orientations, one should
take seriously the idea of the ``units'' of a multiplicative cohomology
theory.  In order to make such algebraic notions about multiplicative
cohomology theories precise, May, Quinn, Ray, and Tornehave
\cite{MQRT:ersers} developed the notion of an $\einfty$ ring spectrum
and used it to describe the obstruction theory for orientations of
spherical fibrations.  Indeed the Thom spectra of infinite
Grassmannians such as $BO$ and $BU$ provided their starting examples
of $\einfty$ ring spectra.  It is interesting to recall that the
theory of structured ring spectra and ``brave new algebras'' in part
had roots in this extremely concrete geometric theory.

In a forthcoming paper, three of us (Ando, Hopkins, Rezk) construct an
$\einfty$ orientation of $tmf$, the spectrum of topological modular
forms \cite{Hopkins:icm2002,AHR:orientation}: more precisely, we construct a map
of $\einfty$ ring spectra from the Thom spectrum $MO\langle 8
\rangle$, also known as $MString$, to the spectrum $tmf$.  That paper
requires a formulation of the obstruction theory of \cite{MQRT:ersers}
in terms of the adjoint relationship between units and $\splus \linf$
described in Theorem \ref{t-th-units-intro} below.  It also requires a
formulation of the obstruction theory in terms of modern topological
or simplicial model categories of $\einfty$ ring spectra.

This paper began as an effort to connect the results in
\cite{MQRT:ersers} to the results required by \cite{AHR:orientation}.  In the
course of doing this, we realized that modern technology makes it
possible to develop an analogous theory of orientations for $\ainfty$
ring spectra.  We also discovered that recent work on
$\infty$-categories provides an excellent framework for the study of
Thom spectra and orientations, allowing constructions which are very
close to the original sketch in Sullivan's notes.

In this paper we review and extend the theory of orientations of
\cite{MQRT:ersers}, taking advantage of intervening technical
developments, particularly in the theory of multiplicative spectra and
of $\infty$-categories.  We recover the obstruction theory of
\cite{MQRT:ersers}, expressed in terms of the adjunction of Theorem
\ref{t-th-units-intro}, and extend it to $\ainfty$ ring spectra.  We
also extend the theory to more general kinds of fibrations, such as
those which appear in the study of twisted generalized cohomology.

We begin by reviewing the analogy between the theory of orientations
and the theory of locally free sheaves of rank one.  This analogy,
which appears in Sullivan's notes, is an excellent guide to the basic
results of the subject.  Modern technology enables us to hew more
closely to this intuitive picture than was possible thirty years ago,
and so it is remarkable how much of the theory was worked out in the
early seventies.

\subsection{The Thom isomorphism and invertible sheaves}

Let $A$ be a commutative ring spectrum; at this point, we remain vague
about exactly what we mean by the term.  An $A$-module $M$ is free of
rank one if there is a weak equivalence of $A$-modules
\[
   A \to M.
\]
If $X$ is a space, and $V$ is a virtual vector bundle of
rank zero over $X$, then let $X^{V}$ be the associated Thom spectrum.
A Thom isomorphism for $V$ is a weak equivalence of $A^{\pt{X}}$-modules
\[
          A^{\pt{X}} \heq A^{X^{V}}.
\]
The local triviality of $V$ and the suspension isomorphism together
imply that, locally on $X,$ $A^{X^{V}}$ is free of rank one, and
suggest that there is a cohomological obstruction to the existence of
a Thom isomorphism.

Thus let
\[
 \rsps{A}: \CatOf{spaces}^{\op} \rightarrow \CatOf{commutative ring
spectra}
\]
be the presheaf of ring spectra defined by
\[
      \rsps{A} (Y) = A^{\pt{Y}},
\]
and let
\[
   \cohps{A}{X} = \rsps{A} \restr{\Open (X)^{\op}}
\]
be the associated presheaf of ring spectra on $X.$ Similarly, for $U$
open in in $X$, let $U^{V}$ be the Thom space of $V$ restricted to
$U$, and let $\twistedcohps{A}{V}$ be the presheaf of
$\cohps{A}{X}$-modules whose value on the open set $U$ is given by the
formula
\[
       \twistedcohps{A}{V} (U) = A^{U^{V}}.
\]

An open cover $\cover{U} = \{U_{\alpha} \}_{\alpha \in I}$ of $X$ and
trivializations
\[
      h_{\alpha}: U_{\alpha}^{V} \cong \pt{(U_{\alpha})}
\]
of $V$ over $\cover{U}$ give equivalences
\[
     \cohps{A}{X} (U_{\alpha}) = A^{\pt{( U_{\alpha})}}
\xrightarrow{A^{h_{\alpha}}} A^{(U_{\alpha}^{V})} =
\twistedcohps{A}{V} (U_{\alpha}).
\]
Over an intersection $U_{\alpha\beta} = U_{\alpha}\cap U_{\beta}$, we
have an equivalence of $\cohps{A}{X} (U_{\alpha\beta})$-modules
\[
g_{\alpha\beta} : \cohps{A}{X} (U_{\alpha\beta}) \xra{h_{\alpha}}
\twistedcohps{A}{V} (U_{\alpha\beta}) \xra{h_{\beta}^{-1}}
\cohps{A}{X} (U_{\alpha\beta}).
\]

This analysis suggests that we let $\unitscohps{A}$ be the presheaf on
spaces which on $Y$ consists of $\rsps{A} (Y)$-module equivalences
\[
   \rsps{A} (Y) \rightarrow \rsps{A} (Y)
\]
Then the $g_{\alpha\beta}$ are a Cech $1$-cocycle for the cover
$\cover{U}$, with values in the presheaf $\unitscohps{A}$, whose class
\[
   [V] \in H^{1} (X;\unitscohps{A})
\]
depends only on (the isomorphism class of) the vector bundle $V$.
However $[V]$ is defined, it should be both the Cech cohomology class
classifying the invertible $\cohps{A}{X}$-module $\twistedcohps{A}{V}$
and the obstruction to giving a Thom isomorphism for $V$ in
$A$-theory.

The difficulty in making this program precise is at once subtle and
familiar, and arises from the suppleness of the presheaf $\uln{A}.$
For example, in trivializing the cocycle $g_{\alpha\beta},$ one
expects to encounter a family of sections
\[
     f^{0}_{\alpha} \in \unitscohps{A} ( U_{\alpha})
\]
\emph{and homotopies}
\begin{equation} \label{eq:37}
    f^{1}_{\alpha\beta}: f^{0}_{\beta}\heq f^{0}_{\alpha}
g_{\alpha\beta}
\end{equation}
in $\unitscohps{A} (U_{\alpha\beta}\times I).$ The
$f^{1}_{\alpha\beta}$ will satisfy a coherence condition of their own,
and so forth.  Thus, we are brought quickly to the need for
homotopical coherence machinery. 

\subsection{Units} \label{sec:units}

The classical approach to addressing these issues goes as follows.
Let $A$ be an associative ring spectrum, that is, a monoid in the
homotopy category of spectra.  Following \cite{MQRT:ersers}, let
$\GL{A}$ be the pull-back in the diagram of (unpointed)
spaces\footnote{$\GL{A}\to \linf A$ is the inclusion of a set of path
components, so this is a homotopy pull-back.  To make the functor
$A\mapsto \GL{A}$ homotopically well-behaved, we should require $A$
to be a fibrant object in a model category of algebras.  For the
model category of $\ainfty$ ring spectra in Lewis-May-Steinberger
spectra, all objects are fibrant, and the simple definition
\eqref{eq:59} suffices.  See \S\ref{sec:space-units-derived} for
further discussion on this point.}
\begin{equation}\label{eq:59}
\begin{CD}
\GL{A} @>>> \linf{A}
 \\
@VVV @VVV \\
(\pi_{0}A)^{\times} @>>> \pi_{0}A.
\end{CD}
\end{equation}
If $X$ is a
space, then
\[
    [X,\GL{A}] = \{f\in A^{0} (\pt{X})| \pi_{0}f (X) \subset
(\pi_{0}A)^{\times}\} = A^{0} (\pt{X})^{\times},
\]
so $\GL{A}$ is called the \emph{space of units} of $A,$ and so a more
refined definition of $\unitscohps{A}$ is
\[
     \unitscohps{A} (X) = \map (X,\GL{A}).
\]
If $A$ is $\ainfty$ ring spectrum in the sense of
Lewis-May-Steinberger, then $\GL{A}$ is a group-like $\ainfty$ space,
so it has a delooping $B\GL{A}$, and the appropriate cohomology group
is
\[
    H^{1} (X,\unitscohps{A}) = [X,B\GL{A}].
\]
Moreover if $A$ is an $\einfty$ ring spectrum, then there is a spectrum
$\gl{A}$ such that
\[
     \GL{A} \heq \linf \gl{A}.
\]

One could also work in a modern symmetric monoidal category of
spectra, such as the $S$-modules of \cite{EKMM} or one of the
categories of diagram spectra \cite{MR1806878}.  As we discuss in
\S\ref{sec:space-units-derived}, in that case we should require $A$ to be a
cofibrant-fibrant algebra, and then (Proposition
\ref{inf-t-pr-Rwe-and-GL}) the homotopy type of $\GL{A}$
is that of the subspace of the \emph{derived} mapping space of
$A$-module endomorphisms $\hom_{A} (A,A)$ consisting of weak
equivalences.  This approach is related to the
approach studied in \cite{MR2271789}.


For now we continue by imitating how one would proceed if $\GL{A}$
were a topological group, acting on the spectrum $A.$ Associated to
the group $\GL{A}$ we would have the principal fibration
\[
    \GL{A} \to E\GL{A} \to B\GL{A},
\]
and given a map $f: X\to B\GL{A}$, we would form the pull-back
\[
\begin{CD}
   P @>>> E\GL{A} \\
   @VVV @VVV \\
    X @> f >> B\GL{A}.
\end{CD}
\]
One expects that the Thom spectrum associated to $f$ is a sort of
Borel construction
\begin{equation} \label{eq:borelthom}
   Mf = P \times_{\GL{A}} A,
\end{equation}
and that the space of $A$-orientations of $Mf$ is the space of
sections of $P/X$.  To motivate this picture, consider the case that
$G=\GL{R}$ is the group of units of a discrete ring $R$, so that $BG$,
the moduli space of $G$-torsors, is equivalently the moduli space of
free $R$-modules of rank $1$.  To a principal $G$-bundle
\[
    G \to P \to X
\]
we can attach the family of free rank-one $R$-modules
\[
    \pi: \xi = P \times_{G} R \to X
\]
parametrized by X; conversely to such a family $\xi$ we can attach its
$G$-torsor of trivializations
\[
    P = \xi^{\times} = \{ z \in \xi | z \text{ is an $R$-module
    generator of }\xi_{\pi (z)} \}
\]
An obvious $R$-module associated to this situation, at least if $X$ is
discrete, is
\begin{equation} \label{eq:31}
   \bigoplus_{x\in X} \xi_{x} \iso \colim_{x\in X} \xi_{x}
\end{equation}
If $X$ is discrete then $P$ is the $G$-space
\[
    P = \coprod_{x\in X} P_{x},
\]
and we can also form the $R$-module
\begin{equation} \label{eq:32}
    \Z[P]\otimes_{\Z[G]} R.
\end{equation}
Here $\Z[P]$ is the free abelian group on the points of $P,$ $\Z[G]$
is the group ring of $G,$ and the natural map
\[
   \Z[G] \to R
\]
is the counit of the adjunction
\begin{equation} \label{eq:29}
\xymatrix{ {\Z: \CatOf{groups}} \ar@<.3ex>[r] & {\CatOf{rings}:
\GLsym.}  \ar@<.3ex>[l] }
\end{equation}
In fact the two $R$-modules \eqref{eq:31} and \eqref{eq:32} are
isomorphic, as they represent the same functor:
\begin{align} \label{eq:56}
   \CatOf{$R$--modules} \bigl( \bigoplus_{x\in X} \xi_{x},T\bigr) &
\iso
   \prod_{x\in X} \CatOf{$R$--modules} (\xi_{x},T) \\
   & \iso \notag
   \prod_{x\in X} \CatOf{$G$--sets} (\xi_{x}^{\times},T)\\
   & \iso \notag
   \CatOf{$G$--sets} (P, T) \\
   & \iso \notag
      \CatOf{$\Z[G]$--modules} (\Z[P],T) \\
   & \iso \CatOf{$R$--modules}(\Z[P]\otimes_{\Z[G]}R,T). \notag
\end{align}

These constructions generalize to give equivalent notions of
$A$-module Thom spectrum when $A$ is an $\ainfty$ ring spectrum.  The
idea that this should be so is due to the senior author (Hopkins), and
was the starting point for this paper.  For now we set
\begin{equation} \label{eq:30}
   Mf = \Z[P]\otimes_{\Z[G]} R,
\end{equation}
and observe that with $T=R$ above we have
\[
  \CatOf{$R$--modules} (Mf,R) \iso \CatOf{$G$--sets} (P, R).
\]
With respect to this isomorphism, the set of \emph{orientations} of
$Mf$ is to be the subset
\[
\xymatrix{ {\CatOf{$R$--modules} (Mf,R)} \ar@{-}[r]^-{\iso} &
{\CatOf{$G$--sets} (P, R)}
\\
{\CatOf{orientations} (Mf,R)} \ar@{-}[r]^-{\iso} \ar@{>->}[u] &
{\CatOf{$G$--sets} (P,G),} \ar@{>->}[u] }
\]
which in turn is isomorphic to the set of sections of the principal
$G$-bundle $P/X$, as expected.

Returning to the provisional definition of
equation~\eqref{eq:borelthom}, one approach to generalizing to the
input of a ring spectrum $A$ and a space $X$ is to develop the
machinery to mimic that definition, in the form \eqref{eq:32}, in the
setting of $\ainfty$ spaces. We carry this out in
\S\ref{sec:ainfty-thom-spectrum}.

Another possible approach to this problem involves coming to grips
with homotopy sheaves of spectra, and construct $\xi$ as a homotopy
sheaf of $A$-modules.  If $X$ is a paracompact Hausdorff space, then
as in \cite{MR0232393} or \HTT{7.1}, it is equivalent to consider
homotopy local systems of $A$-modules parametrized by $X.$

The parametrized homotopy theory of May and Sigurdsson provides one
context for doing so, and they have discussed twisted generalized
cohomology from this point of view in \cite{MR2271789}.
The apparatus of $\infty$-categories provides another setting for such
questions.  In \S\ref{sec:units-via-infty} we show how to develop a
framework for parametrized spectra in the context of the theory of
quasicategories of Joyal and Lurie that makes it possible to develop
the theory of Thom spectra and orientations essentially as we have
described it above, with a construction of Thom spectra generalizing
\eqref{eq:31}.

As mentioned above, it is remarkable the extent to which these ideas
can and have been implemented using classical methods.  To be more
specific, the space $B\GL{S}$ associated to the sphere spectrum is the
classifying space for stable spherical fibrations, and if $A$ is an
$\ainfty$ ring spectrum, then the unit $S\to A$ gives rise to a map
$B\GL{S} \to B\GL{A}.$

Given a spherical fibration classified by a map $g: X\to B\GL{S}$, we
can form the solid diagram
\begin{equation}\label{eq:28}
\xymatrix{ {P} \ar[r] \ar[d] & {B (S,A)} \ar[r] \ar[d] & {E\GL{A}}
 \ar[d]\\
{X} \ar[r]_-{g} \ar[ur] \ar@{<} `d [r] ` [rr]_-{f} [rr] & {B\GL{S}}
\ar[r] & {B\GL{A},} }
\end{equation}
in which the rectangles are homotopy pull-backs.  In that case,
\[
     Mf = (Mg)\Smash A,
\]
the space $B (S,A)$ is the space of $A$-oriented spherical fibrations,
and to give an orientation
\[
   Mg \to A
\]
is to give a lift as indicated in the diagram.

In his 1970 MIT notes \cite{MR2162361} (in the version available at
\url{http://www.maths.ed.ac.uk/~aar/books/gtop.pdf}, see the note on page
236), Sullivan introduced this picture, and suggested that Dold's
theory of homotopy functors \cite{MR0198464} could be used to construct
the space $B (S,A)$ of $A$-oriented spherical fibrations.  He also
mentioned that the technology to construct the delooping $B\GL{A}$ was
on its way.  Soon thereafter, May, Quinn, Ray, and Tornehave in
\cite{MQRT:ersers} constructed the space $B\GL{A}$ in the case that
$A$ is an $\einfty$ ring spectrum, and described the associated
obstruction theory for orientations of spherical fibrations.
Thus, at least for spherical fibrations and $\einfty$ ring spectra,
the obstruction theory for Thom spectra and orientations has been
available for over thirty years.

Various aspects of the theory of units and Thom spectra have been
revisited by a number of authors as the foundations of stable homotopy
theory have advanced. For example, Schlichtkrull \cite{MR2057776}
studied the units of a symmetric ring spectrum, and May and Sigurdsson
\cite{MR2271789} have studied units and orientations in light of their
categories of parametrized spectra.  Very recently May has prepared an
authoritative paper revisiting operad (ring) spaces and operad (ring)
spectra from a modern perspective, which has substantial overlap with
some of the discussion in this paper \cite{May:rant} (notably
Section~\ref{sec:constr-may-quinn}).

\subsection{Acknowledgments}

We thank Peter May  for his many contributions to this subject and for
useful conversations and correspondence.  We are also very grateful to
Mike Mandell for invaluable help with many parts of this project.
We thank Jacob Lurie for helpful conversations and encouragement.  We
thank John Lind for pointing out an error in a previous draft.  Some
of the results in Section~\ref{sec:ainfty-thom-spectrum} are based on
work in the 2005 University of Chicago Ph.D. thesis of the second
author: he would like to particularly thank his advisors, May and
Mandell, for all of their assistance.

\section{Overview} \label{sec:introduction}

In this section, we give a detailed summary of the contents of the
paper and state the main results.

\subsection{The spectrum of units and $\einfty$ orientations}
\label{sec:obstr-theory-einfty}

We begin in \S\ref{sec:constr-may-quinn} by revisiting the
construction, due to May et al. \cite{MQRT:ersers}, of the spectrum of
units $\gl{R}$ associated to an $\einfty$ ring spectrum $R.$ Motivated
by the adjunction \eqref{eq:29}, or more precisely its restriction to
abelian groups
\[
\xymatrix{ {\Z: \CatOf{abelian groups}} \ar@<.3ex>[r] &
{\CatOf{commutative rings}: \GLsym,} \ar@<.3ex>[l] }
\]
we prove the following (see also \cite{May:rant}).

\begin{Theorem}[Theorem \ref{t-th-units}]\label{t-th-units-intro}
The functor $\glsym$ participates as the right adjoint in an adjunction
\begin{equation}  \label{eq:33}
\xymatrix{ {\splus \linf: \Ho \CatOf{$(-1)$-connected spectra}}
\ar@<.3ex>[r] & {\Ho \CatOf{$\einfty$ ring spectra}: \glsym} \ar@<.3ex>[l] }
\end{equation}
which preserves the homotopy types of derived mapping spaces.
\end{Theorem}

With this in place, in \S\ref{sec:thom-spectra} we recall and extend
the obstruction theory of \cite{MQRT:ersers} for $\einfty$
orientations.  Let $R$ be an $\einfty$ ring spectrum, and suppose that
$b$ is a spectrum over $b\gl{R} = \Sigma\gl{R}.$ Let $p$ be the
homotopy pull-back in the solid diagram
\begin{equation} \label{eq:13}
\xymatrix{ {\gl{R}} \ar@{=}[r] \ar[d] & {\gl{R}} \ar[d]
 \\
{p} \ar[r] \ar[d] \ar@{-->}[ur] & {e\gl{R}\heq \ptspace } \ar[d]
 \\
{b} \ar[r]^-{f} \ar@{-->}[ur] & {b\gl{R}.}  }
\end{equation}
Note that if $B=\linf b$ and $P=\linf p$, then after looping down we
have a fibration sequence
\begin{equation} \label{eq:5}
  \GL{R} \to P \to B\to B\GL{R}.
\end{equation}
Note also that an $E_\i$ map $\varphi:R\to A$ gives a diagram
\begin{equation} \label{eq:13A}
\xymatrix{ {\gl{R}} \ar[r] \ar[d] & {\gl{A}} \ar[d]
 \\
{p} \ar[r] \ar[d] \ar@{-->}[ur] & {e\gl{A}\heq \ptspace } \ar[d]
 \\
{b} \ar[r]^-{\tilde{\varphi}\circ f} \ar@{-->}[ur] & {b\gl{A},} }
\end{equation}
where we write $\tilde{\varphi}:bgl_1 R\to bgl_1 A$ for the induced
map.

\begin{Definition}\label{def-5}
The Thom spectrum $M=Mf$ of $f: b\to b\gl{R}$ is the homotopy push-out
$M$ of the diagram of $\einfty$ spectra
\begin{equation}   \label{eq:6}
\begin{CD}
\splus \linf  \gl{R} @>>> R \\
@VVV @VVV \\
\splus \linf p @>>> M,
\end{CD}
\end{equation}
where the top map is the counit of the adjunction \eqref{eq:33}.  Note
that the spectrum underlying $M$ is the derived smash product
\begin{equation} \label{eq:60}
    M = \splus P \smash^{L}_{\splus \GL{R}} R,
\end{equation}
generalizing the construction \eqref{eq:30}.
\end{Definition}


The adjunction between $\splus \linf$ and $\glsym$ shows that, for any
$E_\infty$ map $\varphi:R\to A$, we have a homotopy pull-back diagram
of derived mapping spaces
\begin{equation}
\begin{CD}
\CatOf{$\einfty$ ring spectra}  (M,A) @>>> \spectra (p,\gl{A}) \\
@VVV @VVV \\
\CatOf{$\einfty$ ring spectra}  (R,A) @>>> \spectra (\gl{R},\gl{A}),
\end{CD}
\end{equation}
and comparing fibers over $\varphi$ gives the following.

\begin{Theorem} \label{t-th-Thom-iso}
There is a homotopy pull-back diagram
\[
\begin{CD}
\CatOf{$\einfty$ $R$-$\mathrm{algebras}$} (M,A) @>>> \spectra (p,\gl{A})\\
@VVV @VVV \\
\{\varphi \} @>>> \spectra (\gl{R},\gl{A}).
\end{CD}
\]
That is, the space of $R$-algebra maps $M\to A$ is weakly equivalent
to the space of lifts in the diagram \eqref{eq:13A}.
\end{Theorem}

An important feature of our construction is that it produces a Thom
spectrum from any map $b \to b\gl{R},$ not just from spherical
fibrations.  To make
contact with the classical situation, let $S$ be the sphere
spectrum, and suppose we are given a map
\[
  g: b\to b\gl{S},
\]
so that
\[
   G = \linf g: B \to B\GL{S}
\]
classifies a stable spherical fibration.  Using Definition
\ref{def-5}, we can form the $\einfty$ Thom spectrum $Mg.$
In \S\ref{sec:comp-thom-spectra}, we show that the spectrum underlying
$Mg$ is the usual Thom spectrum of the spherical fibration classified
by $G,$ as constructed for example in \cite{LMS:esht}.


Now suppose that $R$ is an $\einfty$ spectrum with unit $\iota: S\to
R$, and let
\[
  f = b\gl{\iota}\circ  g: b \to b\gl{S} \to b\gl{R}.
\]
Then
\[
    Mf \heq Mg \smash^{L} R,
\]
and so we have an equivalence of derived mapping spaces
\[
    \einftyspectra  (Mg,R) \heq \CatOf{$\einfty$ $R$-algebras} (Mf,R).
\]
If we let $b (S,R)$ be the pull-back in the solid diagram
\begin{equation} \label{eq:34}
\xymatrix{ {p} \ar[r] \ar[d] & {b (S,R)} \ar[r] \ar[d] & {\ptspace}
\ar[d]
\\
{b} \ar[r] \ar@{-->}[ur] & {b\gl{S}} \ar[r] & {b\gl{R},} }
\end{equation}
then Theorem \ref{t-th-Thom-iso} specializes to a result of May,
Quinn, Ray, and Tornehave \cite{MQRT:ersers}.

\begin{Corollary}\label{t-thom-iso-mqr}
The derived space of $\einfty$ maps $Mg\to R$ is weakly equivalent to
the derived space of lifts in the diagram \eqref{eq:34}.  
\end{Corollary}

\subsection{The space of units and orientations}
\label{sec:space-units-orient}

The authors of \cite{MQRT:ersers} describe $\einfty$ orientations for
$\einfty$ ring spectra, and they also discuss not-necessarily $\einfty$
orientations of $\einfty$ spectra.  As far as we know, there is no
treatment in the literature of the analogous orientation theory for
$\ainfty$ ring spectra.

In \S\ref{sec:ainfty-thom-spectrum} and in
\S\ref{sec:units-via-infty}, we give two independent approaches to the
theory of Thom spectra and orientations for $\ainfty$ ring spectra.
Our first approach begins by adapting the ideas of the $\einfty$
construction in \S\ref{sec:constr-may-quinn}.  In the
associative case the analogue of the adjunction \eqref{eq:29} is
\[
\xymatrix{ {\CatOf{group-like $A_{\infty}$ spaces}} \ar@<.5ex>[r] &
{\ainftyspaces} \ar@<.5ex>[r]^-{\splus} \ar@<.5ex>[l]^-{\GLsym} &
{\CatOf{$\ainfty$ ring spectra}: \GLsym,} \ar@<.5ex>[l]^-{\linf} }
\]
where the right-hand adjunction is a special case of
\cite[p. 366]{LMS:esht}.  If $R$ is an $\ainfty$ spectrum, then we
have the related adjunction
\[
\xymatrix{ {\splus: \lRmodules} \ar@<.5ex>[r] & {\Rmodules: \linf.}
\ar@<.5ex>[l] }
\]
The main difficulty in using this classical operadic approach is that
$\GL{R}$ is a not a topological group but rather only a group-like
$\ainfty$ space, and so it is not immediately apparent how to form the
(quasi)fibration
\[
    \GL{R} \to E\GL{R} \to B\GL{R},
\]
and then make sense of the constructions suggested in
\S\ref{sec:units}.

In \S\ref{sec:ainfty-thom-spectrum} we present technology to realize
the picture sketched in \S\ref{sec:units}.   The essential strategy is
to adapt the operadic smash product of \cite{MR1361938,EKMM} to the
category $\spaces$ of spaces.  Specifically, we produce a symmetric
monoidal product on a subcategory of $\spaces$ such that monoids
for this product are precisely $\ainfty$-spaces; this allows us to
work with models of $\GL{R}$ which are strict monoids for the new
product.  The observation that one could carry out the program of
\cite{EKMM} in the setting of spaces is due to Mike Mandell, and was
worked out in the thesis of the second author \cite{Blumberg-thesis}
(see also the forthcoming paper \cite{Blumberg-Cohen-Schlichtkrull}).

The relevant category of spaces for this
product is $*$-modules, the space-level analogue of EKMM's
$S$-modules.  Precisely, let $\L$-spaces be the category of spaces
with an action of the $1$-space $\sL (1)$ of the linear isometries
operad.  On this category we can define an operadic product $\timesL$
which is associative and commutative but not unital.  The category of
$*$-modules is the subcategory of $\L$-spaces for which the unit
$*\timesL X \to X$ is an isomorphism.  The category of $*$-modules is
Quillen equivalent to the category of spaces and admits an operadic
symmetric monoidal product.

In this setting, we can form a model of $\GL{R}$ which is a group-like
monoid, and then model $E\GL{R} \to B\GL{R}$ as a quasi-fibration with
an action of $\GL{R}.$  Given a fibration of $*$-modules
\[
f: B\to B\GL{R},
\]
$\GL{R}$ acts on the pull-back $P$ in the diagram
\[
\xymatrix{
P \ar[r] \ar[d] & E\GL{R} \ar[d] \\
B \ar[r]^f & B\GL{R},
}
\]
and the $S$-module $\splus P$ is a right $\splus
\GL{R}$-module.  We can then imitate \eqref{eq:30} to form an
$R$-module Thom spectrum.

\begin{Definition}  \label{def-6}
Given a map of spaces $f \colon B \to B\GL{R}$, form $P$ as the
pullback in the diagram above associated to a fibrant replacement in
$*$-modules of the free map $* \timesL \L B \to B\GL{R}$.  The Thom
spectrum of $f$ is defined to be the derived smash product
\begin{equation} \label{eq:62}
   Mf = \splus P \smash^{L}_{\splus \GL{R}} R.
\end{equation}
\end{Definition}
With this definition, we have
\begin{equation} \label{eq:35}
   \CatOf{right $R$--modules} (M, R) \heq \CatOf{right
$\GL{R}$--spaces} (P,\linf R),
\end{equation}
where here (and in the remainder of this subsection) we are referring
to derived mapping spaces.

\begin{Definition}
The space of \emph{orientations} of $M$ is the subspace of $R$-module
maps $M \to R$ which correspond to
\[
   \CatOf{right $\GL{R}$--modules} (P,\GL{R}) \subset \CatOf{right
$\GL{R}$--modules} (P,\linf R).
\]
under the weak equivalence \eqref{eq:35}.  That is, we have a homotopy
pull-back diagram
\[
\xymatrix{ {\CatOf{orientations} (M,R)} \ar[r]^-{\heq} \ar[d] &
{\CatOf{right $\GL{R}$--spaces} (P,\GL{R})} \ar[d]
\\
{\CatOf{right $R$--modules} (M, R)} \ar[r]^-{\heq} & {\CatOf{right
$\GL{R}$--modules} (P,\linf R).}  }
\]
\end{Definition}

To make contact with classical notions of orientation, for $x\in B$ let
$M_{x}$ be the Thom spectrum associated to the map
\[
\{x \} \to B \xra{f} B\GL{R}.
\]
The map $x \to B$ on passage to Thom spectra gives rise to a map of
$R$-modules $M_{x} \to M$.  Then we have the following (Theorem
\ref{t-th-orientation-ainfty} and Proposition
\ref{t-pr-characterize-orientations}).

\begin{Theorem}  \label{t-th-or-unstable-and-stable}
\hspace{5 pt}
\begin{enumerate}
\item A map of right $R$-modules $u: M\to R$ is an orientation if and
only if for each $x\in B$, the map of $R$-modules
\[
    M_{x} \to M \to R
\]
is a weak equivalence.
\item If $f: B\to B\GL{R}$ is a fibration, then the space of
orientations $M\to R$ is weakly equivalent to the derived space of
lifts
\[
\xymatrix{ {P} \ar[r] \ar[d] & {E\GL{R}} \ar[d]
\\
{B} \ar[r]^-{f} \ar@{-->}[ur] & {B\GL{R}.}  }
\]
\end{enumerate}
\end{Theorem}

As for the Thom isomorphism, just as in the classical situation we
have an $R$-module Thom diagonal
\[
   M \to \splus B \Smash M.
\]
and given a map of right $R$-modules $u: M\to R$, we obtain the
composite map of right $R$-modules
\[
  \mr (u): M \to \splus B \Smash M \xra{1 \Smash f} \splus B \Smash R,
\]
about which we have the following (See \cite{MR609673} and \cite[\S
IX]{LMS:esht}).

\begin{Proposition}[Proposition \ref{t-pr-thom-iso}; see also
Corollary \ref{t-thom-iso-infty-cat-I}]
If $u: M \to R$ is an orientation, then $\mr (u)$ is a weak
equivalence.
\end{Proposition}

As in the $\einfty$ case, we emphasize that our construction
associates an $R$-module Thom spectrum to a map $B \to B\GL{R},$ which
need not arise from a spherical fibration.  To compare to the
classical situation, we suppose that $F$ arises from a stable
spherical fibration via
\[
     F: B \xra{G} B\GL{S} \xra{B\GL{\iota}} B\GL{R}.
\]
Then we can form the Thom spectrum $MG$ using Definition~\ref{def-6}
(in \S\ref{sec:comp-thom-spectra} we show that this coincides with the
Thom spectrum associated to $G$ as in for example \cite{LMS:esht}),
and it follows directly from the definition that
\begin{equation}\label{eq:7}
    MF \heq MG \smash^{L} R.
\end{equation}
We define an $R$-orientation of $MG$ to be a map of spectra
\[
   MG \to R
\]
such that the induced map of $R$-modules
\[
    MF \to R
\]
is an orientation as above.  We then recover the result of Sullivan
and of May, Quinn, Ray, and Tornehave \cite{MR2162361,MQRT:ersers}.

\begin{Corollary}\label{t-th-classical-orientation-sullivan-MQR}
Let $B (S,R)$ to be the pull-back in the solid diagram
\[
\xymatrix{ {P} \ar[r] \ar[d] & {B (S,R)} \ar[r] \ar[d] & {E\GL{R}}
\ar[d]
\\
{B} \ar[r] \ar@{-->}[ur] & {B\GL{S}} \ar[r] & {B\GL{R}.}  }
\]
Then the space of $R$-orientations of $MG$ is the space of indicated
lifts.
\end{Corollary}

\subsection{Quasicategories and units} \label{sec:quasicatgories-units}

As we observed in \S\ref{sec:prol-thom-isom}, the useful notion of
$\GL{R}$-bundle is that of a homotopy sheaf.  Joyal's theory of
quasicategories, as developed in Lurie's book \cite{math.CT/0608040},
allows us to be precise about this. Specifically, we use this theory
in \S\ref{sec:units-via-infty} and \S\ref{sec:i-cat-ii} to give an
account of parametrized spectra (homotopy sheaves), Thom spectra and
orientations which is very close to the intuitive picture discussed in
\S\ref{sec:prol-thom-isom}.

In the new theory, the analogue of $B\GL{R}$ is $\Rwe$, the
subcategory of the $\infty$-category $\Rmod$ of $R$-modules consisting
of equivalences of free rank-one cofibrant and fibrant $R$-modules.  To see the
virtues of $\Rwe$, we note suggestively that it contains all
two-simplices of the form
\[
\xymatrix{ {L} \ar[r]^{f} \ar[dr]_{h} & {M} \ar[d]^{g}
\ar@{}[dl]|(.32){\sigma}
\\
 & {N,} }
\]
where $f,$ $g,$ and $h$ are $R$-module weak equivalences, and $\sigma$
is a homotopy of weak equivalences from $gf$ to $h.$ This is
the sort of data mentioned at \eqref{eq:37}, and it exhibits $\Rwe$
as the classifying space for ``homotopy local systems of free
$R$-modules of rank $1$'' (we call them bundles of $R$-lines).

Specifically, for a fibrant simplicial set $X$, there is an
equivalence of $\infty$-categories between maps $X \to \Rmod$ and
bundles of $R$-modules over $X$.  The identity map $\Rmod \to \Rmod$
classifies the ``universal bundle of $R$-modules'', and pulling back
along the inclusion $\Rwe \to \Rmod$ gives the the universal bundle of
$R$-lines.  We discuss this general approach to parametrized spectra
in detail in Section~\ref{sec:infinity-category-r}.

Now suppose that $R$ is a cofibrant and fibrant algebra in spectra.
Let $\Rcf$ be an object of $R$-line, and let
\[
    \Aut (\Rcf) \eqdef \Rwe (\Rcf,\Rcf) \subset \Rmod (\Rcf,\Rcf)
\]
be the subspace of weak equivalences.  It is a group in the
$\i$-categorical sense: as the automorphisms of $\Rcf$
in the $\i$-category $\Rwe$, it is a group-like monoidal
$\i$-groupoid, or equivalently a group-like $A_{\i}$ monoid.  Note
that $\Rmod (\Rcf,\Rcf)$ is the derived space of endomorphisms of $R$,
and so $\Aut (\Rcf)$ is the derived space of self weak equivalences of
$R.$  We show (Proposition \ref{inf-t-pr-Rwe-and-GL} and
\S\ref{sec:i-category-r}) that
\[
   \GL{(\Rcf)} \heq \Aut (\Rcf).
\]

Now the full $\i$-subcategory
of $\Rwe$ on the single object $\Rcf$ is just $B\Aut (\Rcf).$
By definition, $\Rwe$ is a connected $\i$-groupoid (connected Kan
complex), and so the inclusion
\[
     B\Aut (\Rcf) \to \Rwe
\]
is an equivalence, and it follows that
\[
   B\GL{(\Rcf)} \heq \Rwe.
\]

The analogue of $E\GL{R}$ is $\Rtriv$, the
$\infty$-category of trivialized $R$-lines: $R$-modules equipped with
a specific equivalence to $R$.  It is a contractible Kan complex, and
the natural map
\[
     \Rtriv \rightarrow \Rwe
\]
is a Kan fibration, our model for the fibration $E\GL{R} \to B\GL{R}.$

In this setting, a map $X\to B\GL{R}$ corresponds to a map of simplicial sets
\[
  f: \Sing{X} \to \Rwe,
\]
where $\Sing{X}$ is the $\i$-groupoid (the singular complex) of the
space $X$.  By construction, $\Rwe$ comes equipped with  a universal
bundle $\mathscr{L}$ of $R$-lines, and the map $f$ classifies
the bundle of
$R$-lines $f^{*}\mathscr{L}$ over $X.$  A lift in the diagram
\begin{equation} \label{eq:53}
\xymatrix{ & {\Rtriv}
 \ar[d] \\
{\Sing{X}} \ar@{-->}[ur] \ar[r]_-{f} & {\Rwe;}  }
\end{equation}
corresponds to an equivalence of bundles of $R$-lines
\[
     f^{*}\mathscr{L} \xra{\heq} R_{X},
\]
where $R_{X}$ denotes the trivial bundle of $R$-lines over $X.$

In analogy to \eqref{eq:31}, we have the following.

\begin{Definition}\label{def-7}
The Thom spectrum $Mf$ of $f$ is the colimit of the map of
$\infty$-categories
\[
    \Sing{X} \xra{f} \Rwe \xra{} \Rmod.
\]
\end{Definition}

Note that the colimit is the same as the left Kan extension along the
map to a point, so this definition is an analogue of the May-Sigurdsson
description of the Thom spectrum as the composite of the pullback of a
universal parametrized spectrum followed by the base change along the
map to a point \cite[23.7.1,23.7.4]{MR2271789}.   Using this
definition, it is straightforward to prove the following.

\begin{Theorem} \label{t-i-cat-obstr-intro}
The space of orientations of $Mf$ is weakly equivalent to the space of
lifts in the diagram \eqref{eq:53},
equivalently, to the space of trivializations
\[
    f^{*}\mathscr{L} \to R_{X}.
\]
\end{Theorem}

This definition also implies the following characterization, which
plays a role in \S\ref{sec:comp-thom-spectra} when we compare our
approaches to Thom spectra.  Recall that $\Aut(\Rcf) = \Rwe(\Rcf,\Rcf)$ is
a group in the $\i$-categorical sense: it is a group-like monoidal
$\i$-groupoid, or equivalently a group-like $A_{\i}$ monoid.   In general
if $G$ is a group-like monoidal $\i$-groupoid, then it has a classifying
$\i$-groupoid $BG.$

\begin{Proposition} \label{t-thom-sp-is-R-mod-G-intro}
Let $G$ be a group-like monoidal $\i$-groupoid.
Then the Thom spectrum of a map $BG \xra{} \Rmod$ is equivalent to the
homotopy quotient $\Rcf / G$.
\end{Proposition}

%
The Proposition follows immediately from the construction of the Thom
spectrum, since by definition the quotient in the statement is the
colimit of the map of $\i$-categories
\[
   BG \xra{} \Rwe \to \Rmod.
\]

In these introductory remarks, we have taken some care to relate the
$\i$-categorical treatment to other approaches.  In fact, as explained
in \cite{math.CT/0608040,DAGI,DAGII}, it is possible to develop the $\i$-categories of
of spectra, algebras in spectra, and $R$-modules entirely in the
setting of quasicategories, without appeal to external models of
spectra or even of spaces.  We take this approach to Thom spectra and
orientations in \S\ref{sec:units-via-infty}.  It leads
to a clean discussion, free from distractions of comparison to other
models for spectra.  We also hope that it will serve as a useful
introduction to the results of \cite{DAGI}.

For the reader who prefers to begin with a classical
model for spectra,  in \S\ref{sec:i-cat-ii} we recapitulate some of
the discussion in \S\ref{sec:units-via-infty}, building the
$\i$-category $\Rmod$ from a monoidal simplicial model category of
spectra.  We make use of some basic results from \cite{math.CT/0608040}, but do
not require anything from \cite{DAGI,DAGII}.  We hope that this
section offers a useful introductory example of doing
homotopy theory with quasicategories.

\subsection{Comparison of Thom spectra}

So far we have given three constructions of Thom spectra, Definitions
\ref{def-5}, \ref{def-6}, and \ref{def-7}.   In section
\S\ref{sec:comp-thom-spectra}, we compare these notions to each other
and to the Thom spectra of Lewis-May-Steinberger and May-Sigurdsson
\cite{LMS:esht,MR2271789}.

An inspection of Definitions \ref{def-5} and \ref{def-6} shows that
the following result is a consequence of the familiar fact that, if
$Mf$ is the homotopy pushout in a diagram of $\einfty$ ring spectra
\[
\begin{CD}
   \splus\linf \gl{R} @>>> R \\
   @VVV @VVV \\
   \splus \linf p @>>> Mf,
\end{CD}
\]
then the spectrum underlying $Mf$ is the derived smash product
\[
      Mf \heq \splus \linf p \smash^{L}_{\splus \linf \gl{R}} R.
\]

\begin{Proposition}
Let $R$ be an $\einfty$ ring spectrum, and let $f: b \to b\gl{R}$  be
a map.  The spectrum underlying the $\einfty$ Thom spectrum $Mf$ of
Definition \ref{def-5} is equivalent to the spectrum $M\linf f$ of
Definition \ref{def-6}. 
\end{Proposition}

Given this simple observation, our main focus in
\S\ref{sec:comp-thom-spectra} is on comparing the $\ainfty$
construction of Definition \ref{def-6} and
\S\ref{sec:ainfty-thom-spectrum}, the quasicategorical
construction of Definition \ref{def-7} and
\S\ref{sec:units-via-infty}, and the constructions of
\cite{LMS:esht,MR2271789}.

In \S\ref{sec:i-category-r} (see Proposition
\ref{inf-ii-t-pr-egl-bgl-fib}) we show that $B\GL{R} \heq \Rwe,$
so in \S\ref{sec:comp-thom-spectra} we focus on more conceptual
matters.   Essentially, we are confronted with topological analogues
of the two definitions \eqref{eq:31} and \eqref{eq:32}.    Definition
\ref{def-6} associates to a fibration of $*$-modules $f: B\to B\GL{R}$
the pull-back in the diagram
\[
\begin{CD}
P @>>> E\GL{R}\\
@VVV @VVV \\
B @> f >> B\GL{R},
\end{CD}
\]
The Thom spectrum is then defined to be
\[
    \splus P \Smash^{L}_{\splus \GL{R}}  R.
\]
Definition \ref{def-7} associates to $f$ the ($\i$-category) colimit
of the map
\begin{equation} \label{eq:57}
    \Sing{B} \xra{f} \Sing{B\GL{R}} \heq B\Aut (\Rcf)\heq \Rwe  \to \Rmod.
\end{equation}

The $\i$-categorical setting makes it possible for the comparison of
these two constructions to proceed much as in the discrete case, discussed in
\S\ref{sec:units} (see \eqref{eq:56}).  Just as a set $X$ is the
colimit of the constant map
\[
     X \iso \colim ( X \to \CatOf{sets}),
\]
so if $X$ is a space then it is weakly equivalent to the
$\i$-categorical colimit of the constant map
\[
     X \heq \colim (\Sing{X} \to \CatOf{spaces}),
\]
where $\CatOf{spaces}$ denotes the $\i$-category of spaces.
More generally, if $P\to B$ is principal $G$-bundle and $\CatOf{free
$G$-spaces}$ is the $\i$-category of free $G$-spaces, then
\[
     P \heq \colim (\Sing{B}\to \CatOf{free $G$-spaces} \heq BG);
\]
in our setting this becomes
\[
    P  \heq \colim (\Sing{B} \to \CatOf{$\GL{R}$-spaces} \heq B\GL{R}),
\]
and so $\splus P \Smash_{\splus\GL{R}} R $ is the colimit of
\begin{equation} \label{eq:58}
 \Sing{B} \to \Sing{B\GL{R}} \heq \CatOf{$\GL{R}$-modules} \xra{\splus (\slot)}
    \CatOf{$\splus \GL{R}$-spaces}
       \xra{(\slot)\Smash_{\splus \GL{R}} R} \Rmod.
\end{equation}
From this point of view our job is to show that the two functors
\eqref{eq:57} and \eqref{eq:58}
\[
     \Sing{B\GL{R}} \heq \Rwe \rightarrow \Rmod
\]
are equivalent, which amounts to showing that in each case the functor
is equivalent to the standard inclusion of $\Rwe$ in $\Rmod.$

We develop an efficient proof along these lines in
\S\ref{sec:algebr-thom-spectr}.  However, much of
\S\ref{sec:comp-thom-spectra} is devoted to a more general
characterization of the Thom
spectrum functor from the point of view of Morita theory.  This
viewpoint is implicit in the definition of the Thom spectrum in
Definition~\ref{def-6} as the derived smash product with $R$ regarded
as an $\splus \GL{R}$-$R$ bimodule specified by the canonical action of
$\splus \GL{R}$ on $R$.  Recalling that the target category of
$R$-modules is stable, we can regard this Thom spectrum as essentially
given by a functor from (right) $\splus \GL{R}$-modules to $R$-modules.

Now, roughly speaking, Morita theory (more precisely, the
Eilenberg-Watts theorem) implies that any continuous functor from
(right) $\splus \GL{R}$-modules to $R$-modules which preserves homotopy
colimits and takes $\GL{R}$ to $R$ can be realized as tensoring with
an appropriate $(\splus \GL{R})$-$R$ bimodule.  In particular, this
tells us that the Thom spectrum functor is characterized amongst such
functors by the additional data of the action of $\GL{R}$ on $R$.

In \S\ref{sec:comp-thom-spectra} we develop these ideas in the setting of
$\i$-categories.  Given a functor
\[
      F: \spaces_{/B\Aut (\Rcf)} \to \Rmod
\]
which takes $\ptspace/B\Aut (\Rcf)$ to $\Rcf,$
we can restrict along the Yoneda embedding \eqref{eq:63}\footnote{If $X$ is a space, viewed as an $\i$-groupoid, then the equivalence of $\i$-categories $\spaces/X\simeq\Fun(X^{op},\spaces)$ allows us to regard the Yoneda embedding as a functor $X\to\spaces/X$ which, roughly, sends the point $p$ of $X$ to the ``path fibration'' $X_{/p}\to X$.}
\[
      B\Aut (\Rcf) \xra{} \spaces_{/B\Aut (\Rcf)} \xra{F} \Rmod;
\]
since it takes the object of $B\Aut (\Rcf)$ to $\Rcf$, we may view
this as a functor (map of simplicial sets)
\[
   k: B\Aut (\Rcf) \to B\Aut (\Rcf).
\]
Conversely, given a map $k: B\Aut
(\Rcf) \to B\Aut (\Rcf),$ we get a colimit-preserving functor
\[
    F: \spaces_{/B\Aut (\Rcf)} \to \Rmod
\]
whose value on $B \to B\Aut (\Rcf)$ is
\[
   F (B/B\Aut (\Rcf))  =  \colim (B \xra{} B\Aut (\Rcf) \xra{k} B\Aut
   (\Rcf)\hookrightarrow  \Rmod).
\]
About this correspondence we prove the following.
\begin{Proposition}[Corollary \ref{t-co-axiom-geom-thom-functor}]
A functor $F$ from the $\i$-category $\spaces_{/B\Aut (R)}$ to the
$\i$-category of $R$-modules is equivalent to the Thom spectrum
functor if and only if it
preserves colimits and its restriction along the Yoneda embedding
\[
B\Aut (\Rcf) \to \spaces_{/B\Aut (\Rcf)} \xra{F} \Rmod
\]
is
equivalent to the canonical inclusion
\[
  B\Aut (\Rcf) \xra{\heq} \Rwe \xra{} \Rmod.
\]
\end{Proposition}

It follows easily (Corollary \ref{t-co-thom-spectra-equivalent}) that
the Thom spectrum functors of
Definitions \ref{def-6} and \ref{def-7} are equivalent.
It also follows that, as in Proposition
\ref{t-thom-sp-is-R-mod-G-intro}, the Thom spectrum of a group-like $\ainfty$ map
\[
    \phi: G \to \GL{S}
\]
is the homotopy quotient
\[
    \colim (BG \to \Rmod) \heq R_{h G}.
\]
This observation is the basis for our comparison with the Thom
spectrum of Lewis and May.
In \S\ref{sec:neo-classical-thom} we show that the Lewis-May Thom
spectrum associated to the map
\[
   B\phi: BG \to B\GL{S}
\]
is a model for the homotopy quotient $S_{h G},$ and it follows easily
that we have the following.

\begin{Proposition}[Corollary \ref{t-co-LM-thom-comp}]
The Lewis-May Thom spectrum associated to a map
\[
    f: B \to B\GL{S}
\]
is equivalent to the Thom spectrum associated by Definition
\ref{def-7} to the map of $\infty$-categories
\[
     \Sing{B} \xra{\Sing{f}}  \Sing{B\GL{S}} \heq \Line{S}. 
\]
\end{Proposition}

\subsection{Twisted generalized cohomology}

Our construction of Thom
spectra begins with an $\ainfty$ or $\einfty$
ring spectrum $R$, and attaches to a map
\[
    f: X \to B\GL{R}
\]
an $R$-module Thom spectrum $Mf.$  As we have explained, $B\GL{R}$
can be thought of as the classifying space for bundles of free
$R$-modules of rank $1.$  As such, it is the classifying space for
``twists'' of $R$-theory.  Let $F_{R}$ denote the $R$-module function
spectrum.  Given the map $f$, the \emph{$f$-twisted $R$-homology} of
$X$ is by definition
\[
     R^{f}_{k}X \eqdef \pi_{0}\Rmod (\Sigma^{k}R,Mf) \iso \pi_{k}Mf,
\]
while the \emph{$f$-twisted $R$-cohomology} of $X$ is
\[
    R_{f}^{k}X \eqdef \pi_{0}\Rmod (Mf, \Sigma^{k} R).
\]
If $f$ factors as
\begin{equation}\label{eq:17}
   f: X \xra{g} B\GL{S} \xra{i} B\GL{R},
\end{equation}
then as in \eqref{eq:7} we have
\[
     Mf \heq (Mg)\smash^{L} R,
\]
and so
\begin{align*}
  R_{k}^{f} (X) & = \pi_{k}Mf  \iso \pi_{k}Mg\smash^{L} R  = R_{k} Mg \\
  R^{k}_{f} (X) & = \pi_{0}\Rmod (Mf,\Sigma^{k}R) \iso \pi_{0} \Mod{S} (Mg,\Sigma^{k}R) \iso R^{k} Mg,
\end{align*}
so the $f$-twisted homology and cohomology coincide with untwisted
$R$-homology and cohomology of the usual Thom spectrum of the
spherical fibration classified by $g$.  Thus the constructions in this
paper exhibit twisted generalized cohomology as the cohomology of a
generalized Thom spectrum.  In general the twists correspond to maps
\[
  X \to B\GL{R};
\]
the ones which arise from Thom spectra of spherical fibrations are the
ones which factor as in \eqref{eq:17}.  We shall discuss the
relationship to other approaches to generalized twisted cohomology in
another paper in preparation.

\section{Units after May-Quinn-Ray}

\label{sec:constr-may-quinn}

Let $A$ be an $\einfty$ ring spectrum: then there is a spectrum
$\gl{A}$ such that
\begin{equation} \label{eq:38}
     \linf \gl{A}\heq \GL{A}.
\end{equation}
We recall the construction of $\gl{A},$ which is due to
\cite{MQRT:ersers}.  Since $A$ is an $\einfty$ spectrum, $\GL{A}$ is a
group-like $\einfty$ space, and group-like $\einfty$ spaces model
connective spectra.  More precisely, we prove the following result.

\begin{Theorem} \label{t-th-units}
The functors $\splus \linf$ and $\glsym$ induce adjunctions
\begin{equation} \label{eq:22}
\xymatrix{ {\splus \linf: \Ho \CatOf{$(-1)$-connected spectra}}
\ar@<.3ex>[r] & {\Ho \einftyspectra: \glsym} \ar@<.3ex>[l] }
\end{equation}
of categories enriched over the homotopy category of spaces.
\end{Theorem}

In more detail, in
\S\ref{sec:einfty-spectra}--\ref{sec:einfty-spaces-einfty-1} we recall
that if $C$ is an operad over the linear isometries operad, then there
are Quillen
model categories $\alg{C}{\spectra}$ and $\alg{C}{\spaces}$ of
$C$-algebras in spectra and spaces, and that
the adjunction
\[
       \splus: \spaces \leftrightarrows \spectra: \linf
\]
induces by restriction a continuous Quillen adjunction
\[
       \splus: \alg{\ptC}{\ptdspaces}\iso
\alg{C}{\spaces}\leftrightarrows \alg{C}{\spectra}: \linf
\]
We also recall that if $C$ and $D$ are $\einfty$ operads, then there is
a zig-zag of continuous Quillen equivalences
\[
   \alg{C}{\spectra} \heq \alg{D}{\spectra},
\]
so we have a robust notion of the homotopy category of $\einfty$ ring
spectra, which we denote $\ho \einftyspectra.$

If $X$ is a $C$-space then $\pi_{0}X$ is a monoid, and $X$ is said to
be \emph{group-like} if $\pi_{0}X$ is a group.  We write
$\gplike{C}{\spaces}$ for the full-subcategory of group-like
$C$-spaces.  If $X$ is a $C$-space, then in \S\ref{sec:einfty-spaces-group} we
define $\GL{X}$ to be the pull-back in the diagram
\[
\begin{CD}
   \GL{X} @>>> X \\
@VVV  @VVV \\
\pi_{0} (X)^{\times} @>>> \pi_{0}X.
\end{CD}
\]
The functor $X\mapsto \GL{X}$ is the \emph{right} adjoint of the
inclusion
\[
     \gplike{C}{\spaces} \rightarrow \alg{C}{\spaces}.
\]

One of the main results of \cite{May:gils,MR0339152} is that, for
suitable $\einfty$ operads $C$, $\gplike{C}{\spaces}$ is a model for connective
spectra.  In \S\ref{sec:einfty-spac-conn}, we express this result in
the language of model categories.  Let $C$ be a unital operad
(i.e. the zero space of the operad is a point),
equipped with a map of monads on pointed spaces
\[
   f:  C_{*} \to Q.
\]
Then $\linf: \spectra \to \spaces$ factors through $\alg{C}{\spaces},$
and we show that it has a left adjoint
\[
     \Sigma^{f}: \alg{C}{\spaces} \rightarrow \spectra.
\]
If the basepoint $\ptspace \to C (1)$ is non-degenerate, and if for
each $n$ the space $C (n)$ has the homotopy type of of a
$\Sigma_{n}$-CW complex, then the
adjoint pair $(\Sigma^{f},\linf)$ induces an equivalence of enriched
homotopy categories
\[
\xymatrix{ {\Sigma^{f}: \ho \gplike{C}{\spaces}} \ar@<.3ex>[r] & {\ho
\CatOf{connective spectra}: \Omega^{\infty}} \ar@<.3ex>[l] }.
\]

In \S\ref{sec:units:-proof-theorem} we put all this together.  For a
suitable $\einfty$ operad $C$, we have a sequence of
adjunctions (the left
adjoints are listed on top, and equivalence of homotopy categories is
indicated by $\Qeq$)
\[
\xymatrix{ {\splus \linf: \CatOf{$(-1)$-connected spectra}}
\ar@<-.3ex>[r]_-{\Omega^{f}} & \gplike{C}{\spaces}
\ar@<-.3ex>[l]_-{\Sigma^{f},\Qeq} \ar@<.3ex>[r] & {\alg{C}{\spaces}}
\ar@<.3ex>[l]^-{\GLsym} \ar@<.3ex>[r]^-{\splus} & {\alg{C}{\spectra}:
\glsym} \ar@<.3ex>[l]^-{\linf} }.
\]
This is our model for the adjunction of Theorem \ref{t-th-units}.

The rest of this paper depends on this section only through the
relationship \eqref{eq:38} between $\glsym$ and $\GLsym$ and Theorem
\ref{t-th-units}.  The reader will notice that our construction of
$\glsym$ and the proof of Theorem \ref{t-th-units} mostly amount to
assembling results from the literature, particularly
\cite{May:gils,MR0339152,MQRT:ersers,LMS:esht,EKMM}.  We wrote this
section in the hope that it can serve as a useful guide to
the literature.  In the meantime May has prepared a review of the
relevant multiplicative infinite loop space theory \cite{May:rant},
which also includes the results we need.

\begin{Remark}
Theorem \ref{t-th-units} can be
formulated as an adjunction of $\i$-categories
\[
\xymatrix{ {\splus \linf: \CatOf{$(-1)$-connected spectra}}
\ar@<.3ex>[r] & {\einftyspectra: \glsym} \ar@<.3ex>[l] }.
\]
\end{Remark}

\subsection{$\einfty$ spectra} \label{sec:einfty-spectra}

In this section we review the notion of a $C$-spectrum, where $C$ is
an operad (in spaces) over the linear isometries operad.  We also
recall the fact that the homotopy category of $\einfty$ spectra is
well defined, in the sense that if $C$ and $D$ are two $\einfty$
operads over the linear isometries operad, then the categories of
$C$-spectra and $D$-spectra are connected by a zig-zag of continuous
Quillen equivalences.

If $C$ is an operad, then for $k\geq 0$ we write $C (k)$ for the
$k^\text{th}$ space of the operad.   We also write $C$ for the
associated monad.  Let $\spectra = \spectra_{\universe{U}}$ denote the
category of spectra based on a universe $\universe{U}$, in the sense
of \cite{LMS:esht}.  Let $\Lin$ denote the linear isometries operad of
$\universe{U}$, and let $C\to \Lin$ be an operad over $\Lin$.  Then
\[
     CV = \bigvee_{k\geq 0} C (k)\ltimes_{\Sigma_{k}} V^{\wedge k}.
\]
is the free $C$-algebra on $V$.  We write $\alg{C}{\spectra}$
for the category of $C$-algebras in $\spectra$, and we call its
objects $C$-spectra.

In general $C (\ptspace) \iso \splus C (0)$ is the initial object of
the category of $C$-spectra.  We shall say that $C$ is \emph{unital}
if $C (0) = \ptspace,$ so that $C (0)\iso S$ is the sphere spectrum.

Lewis-May-Steinberger work with unital operads and the free
$C$-spectrum with prescribed unit.  If $S\to V$ is a spectrum under
the sphere, then we write $\ptC V$ for the free $C$ spectrum
on $V$ with unit $\iota: S\to V \to \ptC V.$ This is the pushout in
the category of $C$-spectra in the diagram
\begin{equation} \label{eq:41}
\begin{CD}
CS @>>> C (\ptspace )  = S \\
@V C \iota VV @VVV \\
CV @>>> \ptC V.
\end{CD}
\end{equation}
By construction, $\ptC$ participates in a monad on the category
$\spectra_{S/}$ of spectra under the sphere spectrum.

As explained in \cite[II, Remark 4.9]{EKMM},
\[
    \Svee (V) = S \vee V
\]
defines a monad on $\spectra,$ and we have an equivalence of
categories
\[
   \spectra_{S/} \iso \alg{\Svee}{\spectra}.
\]
It follows that there is a natural isomorphism
\begin{equation}\label{eq:44}
       C (V) \iso C_{*} \Svee (V).
\end{equation}
and (\cite[II, Lemma 6.1]{EKMM}) an equivalence of categories
\[
   \alg{C}{\spectra} \iso \alg{C_{*}}{\spectra_{S/}}.
\]

We recall the following, which can be proved easily using the argument
of \cite{EKMM,MR1806878}, in particular an adaptation of the
``Cofibration Hypothesis'' of \S{VII} of \cite{EKMM}.

\begin{Proposition} \label{t-pr-C-spectra-level-model}
The category of $C$-spectra has the structure of a cofibrantly
generated topological closed model category, in which the forgetful
functor to $\spectra$ creates fibrations and weak equivalences.  If
$\{A\to B \}$ is a set of generating (trivial) cofibrations of
$\spectra$, then $\{CA\to CB \}$ is a set of generating (trivial)
cofibrations of $\alg{C}{\spectra}$. 
\end{Proposition}


In particular, the category of $C$-spectra is cocomplete (this is
explained on pp.  46---49 of \cite{EKMM}), a fact we use in the
following construction.  Let $f: C\to D$ be a map of operads over
$\Lin,$ so there is a forgetful functor
\[
    f^{*} : \alg{D}{\spectra} \rightarrow \alg{C}{\spectra}.
\]
We construct the left adjoint $f_{!}$ of $f^{*}$ as a certain
coequalizer in $C$-algebras; see \cite[\S II.6]{EKMM} for further
discussion of this construction.

Denote by $m: DD\to D$ the multiplication for $D$, and let $A$ be a
$C$-algebra with structure map $\mu: CA\to A$.  Define $f_{!}A$ to be
the coequalizer in the diagram of $D$-algebras
\begin{equation} \label{eq:11}
\xymatrix{ {DCA} \ar@<.3ex>[rr]^{D\mu} \ar@<-.3ex>[rr] \ar[dr]_{Df} &
& {DA} \ar[r] & {f_{!}A.}
\\
& {DDA} \ar[ur]_{m} }
\end{equation}

In fact, it's enough to construct $f_{!}A$ as the coequalizer in
spectra.  Then $D$, applied to the unit $A\to CA$, makes the diagram a
reflexive coequalizer of spectra, and so $f_{!}A$ has the structure of
a $D$-algebra, and as such is the $D$-algebra coequalizer \cite[Lemma
6.6]{EKMM}.

\begin{Proposition} \label{t-pr-D-tens-C}
The functor $f_{!}$ is a continuous left adjoint to $f^{*}$; moreover, for any
spectrum $V$, the natural map
\begin{equation} \label{eq:8}
     f_{!} CV \to DV
\end{equation}
is an isomorphism.
\end{Proposition}

\begin{proof}
It is straightforward if time-consuming to check that $f_{!}$ is
continuous and the left adjoint of $f^*$; in fact, the  statement and
argument work for any map of monads $f:C\to D$ on any category
$\mathcal{C}$, provided that the coequalizer \eqref{eq:11} exists
(naturally in $A$).

For the second part, note that if $T$ is any $D$-algebra, then we have
\[
    \alg{D}{\spectra} (f_{!} CV,T) \iso \alg{C}{\spectra} (CV, f^{*}T
) \iso \spectra (V,T) \iso \alg{D}{\spectra} (DV, T).
\]
In the last two terms in this sequence of isomorphisms we have omitted
the notation for the forgetful functors, and we have used the fact
that the diagram
\begin{equation}\label{eq:20}
\xymatrix{ {\alg{D}{\spectra}} \ar[r]^{f^{*}} \ar[dr] &
{\alg{C}{\spectra}} \ar[d]
\\
& {\spectra} }
\end{equation}
commutes.
\end{proof}

\begin{Remark}\label{rem-3}
The reader may prefer to write $C\otimes V$ for the free $C$-algebra
$CV$, and then $D\otimes_{C}A$ for $f_{!}A.$ With this notation, the
coequalizer diagram defining $D\otimes_{C}A$ takes the form
\[
\xymatrix{ {D\otimes C \otimes A} \ar@<.3ex>[rr] \ar@<-.3ex>[rr] & &
{D\otimes A} \ar[r] & {D\otimes_{C} A,} }
\]
and the isomorphism of Proposition \ref{t-pr-D-tens-C} becomes
\[
    D\otimes_{C} C \otimes V \iso D \otimes V.
\]
\end{Remark}

About this adjoint pair there is the following well-known result.

\begin{Proposition}\label{t-pr-coeo-lms-spectra}
Let $f: C\to D$ be a map of operads over $\Lin$.  The pair
$(f_{!},f^{*})$ is a continuous Quillen pair.
\end{Proposition}

\begin{proof}
Since the diagram \eqref{eq:20} commutes and the forgetful functor to
spectra creates fibrations and weak equivalences in $C$-algebras,
$f^{*}$ preserves fibrations and weak equivalences.  By Proposition
\ref{t-pr-C-spectra-level-model}, $C$ carries a generating set of
cofibrations of $\spectra$ to a generating set of cofibrations of
$\alg{C}{\spectra}.$ The isomorphism $f_{!}C\iso D$ \eqref{eq:8} shows
that $f_{!}$ carries this set to a generating set of cofibrations of
$\alg{D}{\spectra}.$
\end{proof}

It is folklore that various $E_{\infty}$ operads over $\Lin$ give rise
to the same homotopy theory.   Over the years, various arguments have
been given to show this, starting with May's use of the bar
construction to model $f_!$ (see \cite[\S II.4.3]{EKMM} for the most
recent entry in this line).  We present a model-theoretic formulation
of this result in the remainder of the subsection.

\begin{Proposition} \label{t-pr-C-spectra-equiv-D-spectra}
If $f$ is a map of $\einfty$ operads, then $(f_{!},f^{*})$ is a
Quillen equivalence.  More generally, if each map
\[
    f: C (n) \to D (n)
\]
is a weak equivalence of spaces, then $(f_{!},f^{*})$ is a Quillen
equivalence.
\end{Proposition}

Before giving the proof, we make a few remarks.  Assume $f$ is a weak
equivalence of operads.  Since the pullback $f^* \colon \alg{D}{\spectra}
\to \alg{C}{\spectra}$ preserves fibrations and weak equivalences, to
show that $(f_!, f^*)$ is a Quillen equivalence it suffices to show
that for a cofibrant $C$-algebra $X$ the unit of the adjunction $X
\to f^* f_! X$ is a weak equivalence.

If $X = C Z$ is a \emph{free}
$C$-algebra, then $f_{!}X = f_{!} CZ \iso DZ$ by \eqref{eq:8}, and so
the map in question is the natural map
\[
      C Z \to DZ.
\]
It follows from Propositions X.4.7, X.4.9, and A.7.4 of \cite{EKMM} that
if the operad spaces $C (n)$ and $D (n)$ are CW-complexes, and if $Z$
is a wedge of spheres or disks, then $CZ\to DZ$ is a homotopy
equivalence.   In fact, this argument applies to the wider class of
\emph{tame} spectra, whose definition we now recall.

\begin{Definition}[\cite{EKMM}, Definition I.2.4]
A prespectrum $D$ is \emph{$\Sigma$-cofibrant} if each of the structure maps
$\Sigma^{W}D (V)\to D (V\oplus W)$ is a (Hurewicz) cofibration.  A
spectrum $Z$ is \emph{$\Sigma$-cofibrant} if it is isomorphic to one of the
form $LD$, where $D$ is a $\Sigma$-cofibrant prespectrum.  A spectrum
$Z$ is \emph{tame} if it is homotopy equivalent to a
$\Sigma$-cofibrant spectrum.  In particular, a spectrum $Z$ of the
homotopy type of a CW-spectrum is tame.
\end{Definition}

For a general cofibrant $X$, the argument proceeds by reducing to
the free case $X=CZ$. In this paper, we present an inductive argument
due to Mandell \cite{Mandell:thesis}.  A different induction of this
sort appeared in \cite{Mandell:TAQ} in the algebraic setting; that
argument can be adapted to the topological context with minimal
modifications.

Our induction will involves the geometric realization
of simplicial spectra.  As usual, we would like to ensure that a map
of simplicial spectra
\[
f_{\bullet}: K_{\bullet} \to K'_{\bullet}
\]
in which each $f_{n}: K_{n} \to K'_{n}$ is a weak equivalence yields a
weak equivalence upon geometric realization.   The required condition
is that the spectra $K_{n}$ and $K_{n}'$ are tame: Theorem X.2.4 of
\cite{EKMM} says that the realization of weak equivalences of tame
spectra is a weak equivalence if $K_{\bullet}$ and $K'_{\bullet}$ are
``proper'' \cite[\S X.2.1]{EKMM}.  Recall that a simplicial spectrum
$K_\bullet$ is proper if the natural map of coends
\[
\int^{D_{q-1}} K_p \Smash D(q,p)_+ \to \int^{D_q} K_p \Smash D(q,p)_+
\cong K_q
\]
is a Hurewicz cofibration, where $D$ is the subcategory of $\Delta$
consisting of the monotonic surjections (i.e. the degeneracies), and
$D_q$ is the full subcategory of $D$ on the objects $0 \leq i \leq
q$.  This is a precise formulation of the intuitive notion that the
inclusion of the union of the degenerate spectra $s_j K_{q-1}$ in
$K_q$ should be a Hurewicz cofibration.

Thus, to ensure that the spectra that arise in our argument are tame
and the simplicial objects proper, we make the following simplifying
assumptions on our operads.

\begin{enumerate}
\item We assume that the spaces $C(n)$ and $D(n)$ have the
  homotopy type of $\Sigma_{n}$-$CW$-complexes.
\item We assume that $C(1)$ and $D(1)$ are equipped with
  nondegenerate basepoints.
\end{enumerate}

We believe these assumptions are reasonable, insofar as they are
satisfied by many natural examples; for instance, the linear
isometries operad and the little $n$-cubes operad both satisfy the
hypotheses above (see \cite[XI.1.4, XI.1.7]{EKMM} and
\cite[4.8]{May:gils} respectively).  More generally, if $\mathcal{O}$
is an arbitrary operad over the linear isometries operad, then taking
the geometric realization of the singular complex of the spaces
$\mathcal{O}$ produces an operad $|S (\mathcal{O})|$ with the
properties we require.

Goerss and Hopkins have proved two versions of Proposition
\ref{t-pr-C-spectra-equiv-D-spectra} using resolution model structures
to resolve an arbitrary cofibrant $C$-space by a simplicial $C$-space
with free $k$-simplices for every $k$.  A first version
\cite{GH:mpsht} proves the Proposition for LMS spectra, avoiding our
simplifying assumptions on the operads via a detailed study of
``flatness'' for spectra (as an alternative to the theory of
``tameness'').  A more modern treatment \cite{GH:obstruct} works with
operads of simplicial sets and symmetric spectra in topological
spaces.  In that case, as they explain, a key point is that if $X$ is
a cofibrant spectrum, then $X^{(n)}$ is a \emph{free}
$\Sigma_{n}$-spectrum (see Lemma 15.5 of \cite{MR1806878}).  This
observation helps explain why the general form of the Proposition is
reasonable, even though the analogous statement for spaces is much too
strong.


\begin{proof}
A cofibrant $C$-spectrum is a retract of a cell $C$-spectrum, and
so we can assume without loss of generality that $X$ is a cell
$C$-spectrum.  That is, $X = \colim_n X_n$, where $X_0 = C(*)$ and
$X_{n+1}$ is obtained from $X_n$ via a pushout (in $C$-algebras) of
the form
\[
\xymatrix{
C A \ar[r] \ar[d] & X_n \ar[d] \\
C B \ar[r] & X_{n+1} \\
}
\]
where $A \to B$ is a wedge of generating cofibrations of spectra.  By
the proof of Proposition~\ref{t-pr-C-spectra-level-model}
(specifically, the Cofibration Hypothesis), the map $X_n \to X_{n+1}$
is a Hurewicz
cofibration of spectra.   The hypotheses on $C$ and the fact that $A$
and $B$ are CW-spectra imply that that $CA$ and $CB$ have the homotopy
type of CW-spectra, and thus inductively so does $X_n$.  Therefore,
since $f_!$ is a left adjoint, it suffices to show that $X_n \to f^*
f_! X_n$ is a weak equivalence for each $X_n$ --- under these
circumstances, a sequential colimit of weak equivalences is a weak
equivalence.

We proceed by induction on the number of stages required to build the
$C$-spectrum.  The base case follows from the remarks preceding the
proof.  For the induction hypothesis, assume that $f_!$ is a weak
equivalence for all cell $C$-algebras that can be built in $n$ or
fewer stages.  The spectrum $X_{n+1}$ is a pushout $C B \coprod_{C
  A} X_n$ in $C$-algebras, and this pushout is homeomorphic to a bar
construction $B(C B, C A, X_n)$.  Since $f_!$ is a continuous left
adjoint,
it commutes with geometric realization and
coproducts in $C$-algebras, and so $f_!(B(C B, C A, X_n))$ is
homeomorphic to $B(D B, D A, f_! X_n)$.

The bar constructions we are working with are proper simplicial
spectra by the hypothesis that $C(1)$ and $D(1)$ have
nondegenerate basepoints, and thus it suffices to show that at each
level in the bar construction
\[B_q (C B, C A, X_n) \to B_q (D B, D A, f_! X_n)\]
we have a weak equivalence of tame spectra.  This follows
from the inductive hypothesis: we have already shown that the spectra
are tame, and $C B \coprod^q C A \coprod X_n$ can be built in $n$
stages, since $X_n$ can be built in $n$ stages and the free algebras
can be built and added in a single stage.
\end{proof}

The idea of the following corollary goes all the way back to
\cite{May:gils}.

\begin{Corollary}\label{t-co-ho-einfty-well-defined}
If $C$ and $D$ are any two $\einfty$ operads over the linear
isometries operad, then the categories of $C$-algebras and
$D$-algebras are connected by a zig-zag of continuous Quillen
equivalences.
\end{Corollary}

\begin{proof}
Apply Proposition \ref{t-pr-coeo-lms-spectra} to the maps of
$\einfty$ operads over $\Lin$
\[
    C \leftarrow C\times D \rightarrow D.
\]
\end{proof}

Backed by this result, we adopt the following convention.

\begin{Definition}\label{def-8}
We write $\ho \einftyspectra$ for the homotopy category of $\einfty$
ring spectra.  By this we mean the homotopy category $\ho
\alg{C}{\spectra}$ for any $\einfty$ operad $C$ over the linear
isometries operad.
\end{Definition}

\subsection{$\einfty$ spaces} \label{sec:einfty-spaces}

We adopt notation for operad actions on spaces analogous to our
notation for spectra in \S\ref{sec:einfty-spectra}.  Let $C$ be an
operad in topological spaces.  The free $C$-algebra on a space $X$ is
\begin{equation} \label{eq:25}
   CX = \coprod_{k\geq 0} C (k) \times_{\Sigma_{k}} X^{k}.
\end{equation}
We set $C (\emptyset) = C (0).$ The category of $C$-algebras in
spaces, or $C$-spaces, will be denoted $\alg{C}{\spaces}.$

Note that the sequence of spaces given by
\begin{align*}
    P (0) &  = \ptspace = P (1) \\
   P (k) & = \emptyset  \text{ for }k>1
\end{align*}
has a unique structure of operad, whose associated monad is
\[
    PX = \pt{X},
\]
so
\[
     \alg{P}{\spaces}\iso \ptdspaces.
\]

If $C$ is a unital operad and if $Y$ is a pointed space, let $\ptC Y$
be the pushout in the category of $C$-algebras
\begin{equation} \label{eq:42}
\begin{CD}
    C\ptspace @>>> C (\emptyset) = \ptspace \\
@VVV @VVV \\
CY @>>> \ptC Y.
\end{CD}
\end{equation}
Then $\ptC$ participates in a monad on the category of pointed spaces.
Indeed $\ptC$ is isomorphic to the monad $C_{May}$ introduced in
\cite{May:gils}, since for a test $C$-space $T$,
\[
    \alg{C}{\spaces}(\ptC Y,T)\iso \ptdspaces (Y,T) \iso
\alg{C}{\spaces} (C_{May}Y,T).
\]
There is a natural isomorphism
\[
     C X \iso \ptC (\pt{X}),
\]
and an equivalence of categories
\begin{equation} \label{eq:43}
        \alg{C}{\spaces} \iso \alg{\ptC}{\ptdspaces}.
\end{equation}
Part of this equivalence is the observation that, if $X$ is a
$C$-algebra, then it is a $\ptC$ algebra via
\[
   \ptC X \to \ptC (\pt{X}) \iso C X \to X.
\]
We have the following analogue of Proposition~\ref{t-pr-C-spectra-level-model}.

\begin{Proposition}\label{t-pr-C-spaces-level-model}
\hspace{5 pt}
\begin{enumerate}
\item \label{item:1} The category $\alg{C}{\spaces}$ has the structure of a
cofibrantly generated topological closed model category, in which the
forgetful functor to $\spaces$ creates fibrations and weak
equivalences.  If $\{A\to B \}$ is a set of generating (trivial)
cofibrations of $\spaces$, then $\{C A\to C B \}$ is a set of
generating (trivial) cofibrations of $\alg{C}{\spaces}$.
\item \label{item:2} The analogous statements hold for $\ptC$ and
$\alg{\ptC}{\ptdspaces}.$
\item Taking $C=P$, the resulting model category structure on the
category $\alg{P}{\spaces} \iso \ptdspaces$ is the usual one.
\item The equivalence $\alg{C}{\spaces}\iso \alg{\ptC}{\ptdspaces}$
\eqref{eq:43} carries the model structure arising from part
\eqref{item:1} to the model structure arising from part \eqref{item:2}.
\end{enumerate}
\end{Proposition}

\begin{proof}
The statements about the model structure on $\alg{C}{\spaces}$ or on
$\alg{\ptC}{\ptdspaces}$ can be proved for example by adapting the
argument in \cite{EKMM,MR1806878}.  The third part is standard, and
together the first three parts imply the last.
\end{proof}

We conclude this subsection with two results
which will be useful in \S\ref{sec:einfty-spac-conn}.
For the first, note that a point of $C (0)$ determines a map of
operads
\[
   P \to C,
\]
and so we have a forgetful functor
\[
  \alg{C}{\spaces} \to \alg{P}{\spaces} \iso \ptdspaces.
\]
We say that a point of $Y$ is \emph{non-degenerate} if $(Y,\ptspace)$ is
an NDR pair, i.e. that $\ptspace \to Y$ is a Hurewicz cofibration.

\begin{Proposition}\label{t-pr-cof-c-alg-cof-ptd-space}
Suppose that $C$ is a unital operad in topological spaces (or more
generally, an operad in which the base point of $C (0)$ is
nondegenerate).  If $X$ is a cofibrant object
$\alg{\ptC}{\ptdspaces},$ then its base point is nondegenerate.
\end{Proposition}

Note that Rezk
\cite{Rezk:thesis} and Berger and Moerdijk \cite{MR2016697} have
proved a similar result, for algebras in a general model category over
an \emph{cofibrant} operad.  In our case, we need only assume that the
zero space $C (0)$ of our operad has a non-degenerate base point.

\begin{proof}
In the model structure described in Proposition
\ref{t-pr-C-spaces-level-model}, a cofibrant object is a retract of a
cell object, and so we can assume without loss of generality that $X$
is a cell $C$-space.  That is,
\begin{equation}\label{eq:50}
X = \colim_{n} X_{n}
\end{equation}
where $X_{0} = C (\emptyset)$ and $X_{n+1}$ is obtained from $X_n$ as
a pushout in $C$-spaces
\begin{equation} \label{eq:49}
\xymatrix{
C A \ar[r] \ar[d] & X_n \ar[d] \\
C B \ar[r] & X_{n+1}, \\
}
\end{equation}
where $A \to B$ is a disjoint union of generating cofibrations of
$\spaces$.

Our argument relies on a form of the Cofibration Hypothesis of \S{VII} of
\cite{EKMM}.  The key points are the following.

\begin{enumerate}
\item By assumption $X_0 = C (\emptyset) = C (0)$ is non-degenerately
based.
\item The space underlying the $C$-algebra colimit $X$ in
\eqref{eq:50} is just the space-level colimit.
\item In the pushout above,
\[
   X_{n} \to X_{n+1}
\]
is a based map and an unbased Hurewicz cofibration.
\end{enumerate}

The second point is easily checked (and is the space-level analog of
Lemma 3.10 of \cite{EKMM}).  For the last part, the argument in
Proposition 3.9 of \S{VII} of \cite{EKMM} (see also Lemma 15.9 of
\cite{MR1806878}) shows that the pushout \eqref{eq:49} is isomorphic
to a two-sided bar construction $B(C B, C A, X_n)$: this is the
geometric realization of a simplicial space where the $k$-simplices
are given as
$$C B \coprod_{C} (C A)^{\coprod k}\coprod_{C} X_n,$$ and the
simplicial structure maps are induced by the folding map and the maps
$C A \to C B$ and $C A \to X_n$.  Note that by $\coprod_{C}$ we mean
the coproduct in the category of $C$-spaces.  Recall that coproducts
(and more generally all colimits) in $C$-spaces admit a description as
certain coequalizers in $\spaces$.  Specifically, for $C$-spaces $X$
and $Y$ the coproduct $X \coprod_{C} Y$ can be described as the
coequalizer in $\spaces$
\[
\xymatrix{ C (C X \coprod C Y) \ar@<1ex>[r] \ar@<-1ex>[r] & C(X
\coprod
Y) \ar[r]&  X\coprod_{C} Y, \\
}
\]
where the unmarked coproducts are taken in $\spaces$ and the maps are
induced from the action maps and the monadic structure map,
respectively.  Following an argument along the lines of \cite[\S
VII.6]{EKMM} we obtain the following lemma.

\begin{Lemma}
Let $C$ be an operad in spaces.  Let $A$ be a $C$-space and $B$ a
space.  The map $A \to A \coprod_{C} C B$ is an inclusion of a
component in a disjoint union. 
\end{Lemma}

This implies that the simplicial degeneracy maps in the bar
construction are unbased Hurewicz cofibrations and hence that the
simplicial space is proper, that is, Reedy cofibrant in the
Hurewicz/Str\o{}m model structure.  Thus the inclusion of the zero
simplices $C B \coprod_{C} X_n$ in the realization is an unbased
Hurewicz cofibration, and hence the map $X_n \to X_{n+1}$ is itself a
unbased Hurewicz cofibration.  As a map of $C$-algebras, it's also a
based map.
\end{proof}

The second result we need is the following.

\begin{Proposition}\label{t-pr-cx-cw}
Let $C$ be an operad and suppose that each $C(n)$ has the homotopy
type of a $\Sigma_{n}$-$CW$ complex.  Let $X$ be a $C$-space with the
homotopy type of a cofibrant $C$-space.  Then $CX$ has the homotopy
type of a cofibrant $C$-space and the underlying space of $X$ has the
homotopy type of a $CW$-complex.
\end{Proposition}

\begin{proof}
The first statement is an easy consequence of the fact that $C$
preserves homotopies and cofibrant objects.  To see the second,
observe that the forgetful functor preserves homotopies, so it
suffices to suppose that $X$ is a cofibrant $C$-space.  Under our
hypotheses on $C$, if $A$ has the homotopy type of a CW-complex then
so does the underlying space of $CA$ (see for instance page 372 of
\cite{LMS:esht} for a proof).  The result now follows from an
inductive argument along the lines of the preceding proposition.
\end{proof}



\subsection{$\einfty$ spaces and $\einfty$ spectra}

\label{sec:einfty-spaces-einfty-1}

Suppose that $C\to \Lin$ is an operad over $\Lin$.  In this section we
recall the proof of the following result:

\begin{Proposition}[\cite{MQRT:ersers}, \cite{LMS:esht}
p. 366]\label{t-pr-l-spaces-l-spectra} The continuous Quillen pair
\begin{equation} \label{eq:2}
       \splus: \spaces \leftrightarrows \spectra: \linf
\end{equation}
induces by restriction a continuous Quillen adjunction
\begin{equation} \label{eq:3}
       \splus: \alg{\ptC}{\ptdspaces}\iso
\alg{C}{\spaces}\leftrightarrows \alg{C}{\spectra}: \linf
\end{equation}
between topological model categories.
\end{Proposition}

The first thing to observe is that $C$ and $\splus$ satisfy a strong
compatibility condition.

\begin{Lemma}\label{t-le-C-commutes-with-susp}
There is a natural isomorphism
\begin{equation}\label{eq:4}
    C \splus X \iso \splus C X.
\end{equation}
\end{Lemma}

\begin{proof}
It follows from \S{VI}, Proposition 1.5 of \cite{LMS:esht} that, if
$X$ is a space, then
\[
   C (k) \ltimes_{\Sigma_{k}} (\splus X)^{\wedge k} \iso \splus (C (k)
\times_{\Sigma_{k}} X^{k}),
\]
and so
\[
   C \splus X = \bigvee_{k\geq 0} C(k) \ltimes_{\Sigma_{k}} (\splus
X)^{\wedge k} \iso \bigvee_{k\geq 0} \splus (C (k)\times X^{k}) \iso
\splus\left(\coprod_{k\geq 0} C (k) \times X^{k} \right) = \splus CX.
\]
\end{proof}

Next we have the following, from \cite[p. 366]{LMS:esht}.

\begin{Lemma}  \label{t-pr-l-spaces-l-spectra-adj}
The adjoint pair
\begin{equation} \label{eq:23}
       \splus: \spaces \leftrightarrows \spectra: \linf
\end{equation}
induces an adjunction
\begin{equation} \label{eq:36}
       \splus: \alg{C}{\spaces}\leftrightarrows \alg{C}{\spectra}:
\linf
\end{equation}
and so also
\[
       \splus: \alg{\ptC}{\ptdspaces}\iso
\alg{C}{\spaces}\leftrightarrows \alg{C}{\spectra}: \linf
\]
\end{Lemma}

\begin{proof}
We show that the adjunction \eqref{eq:23} restricts to the
adjunction \eqref{eq:36}.  If $X$ is a $C$-space with structure map
$\mu: CX \to X$, then, using the isomorphism \eqref{eq:4}, $\splus X$
is a $C$-algebra via
\[
     C\splus X \iso \splus C X \xra{\splus \mu} \splus X.
\]
If $A$ is a $C$-spectrum, then $\linf A$ is a $C$-space via
\[
    C \linf A \rightarrow \linf C A \xra{} \linf A.
\]
The second map is just $\linf$ applied to the $C$-structure on $A$;
the first map is the adjoint of the map
\[
    \splus C \linf A \iso C \splus \linf A \rightarrow C A
\]
obtained using the counit of the adjunction.
\end{proof}

This adjunction allows us to prove the pointed analogue of Lemma
\ref{t-le-C-commutes-with-susp}.

\begin{Lemma}[\cite{LMS:esht}, \S VII, Prop. 3.5] \label{t-le-ptC-adjunction}
If $C$ is a unital operad over $\Lin,$ then there is a natural isomorphism
\begin{equation} \label{eq:45}
    \splus\ptC Y \iso \ptC \splus Y \iso C\sinf Y.
\end{equation}
\end{Lemma}

\begin{proof}
Let $Y$ be a pointed space.  By Lemma
\ref{t-pr-l-spaces-l-spectra-adj} and the isomorphism \eqref{eq:4},
applying the left adjoint $\splus$ to the pushout diagram
\eqref{eq:42} defining $\ptC Y$ identifies $\splus \ptC Y$ with the
pushout of the diagram \eqref{eq:41} defining $\ptC \splus Y.$ The
second isomorphism is just the isomorphism \eqref{eq:44} together with
the isomorphism (for pointed spaces) $Y$
\[
    \splus Y \iso \sinf (S \vee Y).
\]
\end{proof}

\begin{proof}[Proof of Proposition \ref{t-pr-l-spaces-l-spectra}] It
remains to show that the adjoint pair $(\splus,\linf)$ induces a
Quillen adjunction.  For this it suffices to show that the right
adjoint $\linf$ preserves fibrations and weak equivalences (see, for
example, \cite[Lemma 1.3.4]{Hovey:MC}).  Now recall that the forgetful
functor $\alg{C}{\spectra}\to\spectra$ creates fibrations and weak
equivalences, and similarly for $\spaces$ \cite{EKMM,MR1806878}.  It
follows that the functor
\[
     \linf: \alg{C}{\spectra}\to \alg{C}{\spaces}
\]
preserves fibrations and weak equivalences, since
\[
     \linf: \spectra\rightarrow \spaces
\]
does.
\end{proof}

\begin{Remark}
Note that if $A$ is an $\einfty$ ring spectrum, then $\linf A$ is an
$\einfty$ space in two ways: one is described above, and arises from
the multiplication on $A$.  The other arises from the additive
structure of $A$, i.e. the fact that $\linf A$ is an infinite loop
space.  Together these two $\einfty$ structures give an $\einfty$ ring
space in the sense of \cite{MQRT:ersers} (see also \cite{May:rant}).
\end{Remark}

\subsection{$\einfty$ spaces and group-like $\einfty$ spaces}
\label{sec:einfty-spaces-group}

Suppose that $C$ is a unital operad (pointed would be enough), and let
$X$ be a $C$-algebra in spaces.  The structure maps
\begin{align*}
  \ptspace \to C (0) &\to X \\
  C (2)\times X \times X &\to X
\end{align*}
correspond to a family of $H$-space structures on $X$ and give to
$\pi_{0}X$ the structure of a monoid. 

\begin{Definition}  \label{def-1}
$X$ is said to be \emph{group-like} if $\pi_{0}X$ is a group.  We
write $\gplike{C}{\spaces}$ for the full subcategory of
$\alg{C}{\spaces}$ consisting of group-like $C$-spaces.
\end{Definition}

Note that if $f: X\to Y$ is a weak equivalence of $C$-spaces, then $X$
is group-like if and only $Y$ is.

\begin{Definition}\label{def-2}
We write $\Ho \gplike{C}{\spaces}$ for the image of
$\gplike{C}{\spaces}$ in $\Ho \alg{C}{\spaces}$.  It is the full
subcategory of homotopy types represented by group-like spaces.
\end{Definition}

If $X$ is a $C$-space, let $\GL{X}$ be the (homotopy) pull-back in the
diagram
\begin{equation} \label{eq:12}
\begin{CD}
   \GL{X} @>>> X \\
@VVV  @VVV \\
\pi_{0} (X)^{\times} @>>> \pi_{0}X.
\end{CD}
\end{equation}
Then $\GL{X}$ is a group-like $C$-space.

\begin{Proposition}  \label{t-pr-GL}
The functor $\GLsym$ is the right adjoint of the inclusion
\[
        \gplike{C}{\spaces} \rightarrow \alg{C}{\spaces}
\]
\end{Proposition}

\begin{proof}
If $X$ is a group-like $C$-space, and $Y$ is a $C$-space, then
\[
\alg{C}{\spaces} (X,Y) \iso \gplike{C}{\spaces} (X,\GL{Y});
\]
just as, if $G$ is a group and $M$ is a monoid, then
\[
    \CatOf{monoids} (G,M) = \CatOf{groups} (G,\GL{M}).
\]
\end{proof}

\subsection{Group-like $\einfty$ spaces and connective spectra}
\label{sec:einfty-spac-conn}

A guiding result of infinite loop space theory is that group-like
$\einfty$ spaces provide a model for connective spectra.  We take a
few pages to show how the primary sources (in particular
\cite{MR0420609,May:gils,MR0339152}) may be used to prove a
formulation of this result in the language of model categories.


%

To begin, suppose that $C$ is a \emph{unital} $\einfty$ operad, and
$f$ is a map of monads (on pointed spaces)
\[
    f: \ptC \to Q \eqdef \linf \sinf.
\]
For example, we can take $C$ to be a unital $\einfty$ operad over the
infinite little cubes operad, but it is interesting to note that any
map of monads will do. If $V$ is a spectrum, then $\linf V$ is
a group-like $C$-algebra, via the map
\[
     \ptC \linf V \xra{f} \linf \sinf \linf V \to \linf V.
\]
Thus we have a factorization
\begin{equation} \label{eq:21}
\xymatrix{ {\spectra} \ar[r]^-{\Omega^{f}} \ar[dr]_-{\linf} &
{\gplike{C}{\spaces}} \ar[d]
\\
& {\ptdspaces} }
\end{equation}
We next show that the functor $\Omega^{f}$ has a left adjoint
$\Sigma^{f}.$ By regarding a $C$-space $X$ as a pointed space via
$\ptspace \to C (0) \to X$, we may form the spectrum $\sinf X$.  Let
$\Sigma^{f} X$ be the coequalizer in the diagram of spectra
\[
\xymatrix{ {\sinf \ptC X} \ar@<.3ex>[rr]^{\sinf \mu} \ar@<-.3ex>[rr]
\ar[dr]_{\sinf f} & & {\sinf X} \ar[r] & {\Sigma^{f} X.}
\\
& {\sinf \linf \sinf X} \ar[ur] }
\]
Then we have the following.

\begin{Lemma} \label{t-le-Sigma-f-C-sinf}
The pair
\begin{equation} \label{eq:27}
   \Sigma^{f}: \alg{C}{\spaces} \leftrightarrows \spectra: \Omega^{f}
\end{equation}
are a Quillen pair.  Moreover, the natural transformation
\[
     \Sigma^{f}\ptC \to \sinf
\]
is an isomorphism.
\end{Lemma}

\begin{proof}
As mentioned in the proof of Proposition \ref{t-pr-D-tens-C}, it is
essentially a formal consequence of the construction that $\Sigma^{f}$
is the left adjoint of $\Omega^{f}$.  Given the adjunction, we find
that $\Sigma^{f}\ptC\iso \sinf$, since, for any pointed space $X$ and
any spectrum $V$, we have
\begin{align*}
\spectra (\Sigma^{f}\ptC X,V) & \iso
   \alg{C}{\spaces} (\ptC X,\Omega^{f}V) \\
  & \iso  \ptdspaces (X,\linf V) \\
  & \iso \spectra (\sinf X, V).
\end{align*}
To show that we have a Quillen pair, it suffices (\cite[Lemma
  1.3.4]{Hovey:MC}) to show that $\Omega^{f}$ preserves weak
equivalences and fibrations.  This follows from the commutativity of
the diagram \eqref{eq:21}, the fact that $\linf$ preserves weak
equivalences and fibrations, and the fact that the forgetful functor
\[
\alg{C}{\spaces} \rightarrow \spaces
\]
creates fibrations and weak equivalences.
\end{proof}

Lemma \ref{t-le-Sigma-f-C-sinf} implies that the pair
$(\Sigma^{f},\Omega^{f})$ induce a continuous Quillen adjunction
\[
   \Sigma^{f}: \alg{C}{\spaces} \leftrightarrows \spectra:
\Omega^{f}.
\]
It is easy to see that this cannot be a Quillen equivalence.  Instead,
one expects that it induces an equivalence between the homotopy
categories of \emph{group-like} $C$-spaces and \emph{connective}
spectra.  In \cite{MR1806878}, this situation is called a ``connective
Quillen equivalence.''  The rest of this subsection is devoted to the
proof of the following result along these lines:

\begin{Theorem}  \label{t-th-einfty-spaces-spectra}
Suppose that $C$ is a unital operad, equipped with a map of monads
\[
      f: C \to \linf \sinf.
\]
Suppose moreover that
\begin{enumerate}
\item the base point $\ptspace \to C (1)$ is non-degenerate, and
\item for each $n$, the $n$-space $C (n)$ has the homotopy type of a $\Sigma_{n}$-$CW$-complex.
\end{enumerate}
Then the adjunction $(\Sigma^{f},\Omega^{f})$
induces an equivalence of categories
\[
\xymatrix{ {\Sigma^{f}: \ho \gplike{C}{\spaces}} \ar@<.3ex>[r] & {\ho
\CatOf{connective spectra}: \Omega^{f}} \ar@<.3ex>[l] }
\]
enriched over $\ho \spaces.$
\end{Theorem}

\begin{Remark}
As observed in \cite{May:gils}, adding a whisker to a degenerate
basepoint produces a new operad $C'$ from $C.$   Also if $C$ is a
unital $\einfty$ operad equipped with a map of monads $f: C\to \linf
\sinf$, then  taking the geometric realization of the
singular complex of the spaces $C (n)$ produces
an operad $|S (C)|$ with the properties we require.
\end{Remark}

The following Lemma, easily checked, is implicit in \cite{MR1806878}.
Let
\[
   F: \mathcal{M} \leftrightarrows \mathcal{M'}: G
\]
be a Quillen adjunction between topological closed model categories.
Let $\mathcal{C} \subseteq \mathcal{M}$ and $\mathcal{C'} \subseteq
\mathcal{M'}$ be full subcategories, stable under weak equivalence, so
we have sensible subcategories $\ho \mathcal{C} \subseteq \ho
\mathcal{M}$ and $\ho \mathcal{C'} \subseteq \ho \mathcal{M'}.$
Suppose that $F$ takes $\mathcal{C}$ to $\mathcal{C'}$, and $G$ takes
$\mathcal{C'}$ to $\mathcal{C}.$

\begin{Lemma}\label{t-le-conn-equiv}
If, for every cofibrant $X\in \mathcal{C}$ and every fibrant $Y\in
\mathcal{C'}$, a map
\[
     \phi: FX \to Y
\]
is a weak equivalence if and only if its adjoint
\[
     \psi: X \to GY
\]
is, then $F$ and $G$ induce equivalences
\[
    F: \ho \mathcal{C} \leftrightarrows \ho \mathcal{C'}: G
\]
of categories enriched over $\ho \spaces.$ 
\end{Lemma}

The key result in our setting is the following classical proposition;
we recall the argument from \cite{May:gils,MR0339152}.

\begin{Proposition}\label{t-pr-unit-we-May}
Let $C$ be a unital $\einfty$ operad, equipped with a map of monads
\[
   f: C \to \linf \sinf.
\]
Suppose that the basepoint $\ptspace \to C (1)$ is non-degenerate, and
that each $C (n)$ has the homotopy type of a
$\Sigma_{n}$-$CW$-complex.
If $X$ is a cofibrant $C$-space, then the unit
of the adjunction
\[
    X \to \Omega^{f}\Sigma^{f} X
\]
is group completion, and so a weak equivalence if $X$ is group-like.
\end{Proposition}

The proof of the proposition follows from analysis of the following
commutative diagram of simplicial $C$-spaces:

\begin{equation} \label{eq:14}
\xymatrix{
 B_{\bullet} (\ptC ,\ptC ,X) \ar[r] \ar[d] & \ar[d] \Omega^{f}
 \Sigma^{f} B_{\bullet} (\ptC ,\ptC ,X) \\
 X \ar[r] & \Omega^{f} \Sigma^{f}X.
}
\end{equation}

Specifically, we will show that under the hypotheses, on passage to
realization the vertical maps are weak equivalences and the top
horizontal map is group completion.

We begin by studying the left-hand vertical map; the usual simplicial
contraction argument shows the underlying map of spaces is a homotopy
equivalence, and so on passage to realizations we have a weak
equivalence of $C$-spaces.

\begin{Lemma}\label{t-le-BCCX-to-X}
For any operad $C$ and any $C$-space $X$, the left vertical arrow is a
map of simplicial $C$-spaces and a homotopy equivalence of simplicial
spaces, and so induces a weak
equivalence of $C$-spaces
\[
  B (\ptC,\ptC,X) \to X
\]
upon geometric realization.
\end{Lemma}

The right vertical map is more difficult to analyze, because we do not
know that $\Sigma^{f}$ preserves homotopy equivalences of spaces.  May
\cite{May:gils} shows that, for suitable simplicial pointed spaces
$Y_{\bullet}$, the natural map
\begin{equation} \label{eq:73}
         |\Omega Y_{\bullet}| \to \Omega|Y_{\bullet}|
\end{equation}
is a weak equivalence, and he explains in \cite{May:gils,May:rant} how this weak
equivalence gives rise to a weak equivalence of $C$-spaces
\[
        |\Omega^{f} \Sigma^{f} B_{\bullet} (\ptC,\ptC,X)|\to
	\Omega^{f}|\Sigma^{f}B_{\bullet} (\ptC,\ptC,X)| \iso
	\Omega^{f}\Sigma^{f}B (\ptC,\ptC,X)
\]
by passage to colimits.   We note that in \cite{May:rant}, May
describes proving that \eqref{eq:73} is a weak equivalence as the
hardest thing in \cite{May:gils}.  Therefore, to show that the map
\[
|\Omega^{f} \Sigma^{f} B_{\bullet} (\ptC,\ptC,X)| \to
 \Omega^{f}\Sigma^{f} X
\]
is a weak equivalence, it suffices to show that for cofibrant $X$, the
map $\Sigma^f B(\ptC, \ptC, X) \to X$ is a weak equivalence.
As it is straightforward to check from the definition that
$\Sigma^{f}$ does preserve weak equivalences between $C$-spaces with
the homotopy type of cofibrant $C$-spaces, the desired result will
follow once we show that $B(\ptC, \ptC, X)$ has the homotopy type of
a cofibrant $C$-space if $X$ is cofibrant.



\begin{Lemma} \label{t-le-BCC-cof}
Suppose that $C$ is a unital operad, such that the base point $* \to
C(1)$ is non-degenerate and each $C(n)$ has the homotopy type of a
$\Sigma_{n}$-$CW$-complex.  Let $X$ be a cofibrant $C$-space.  Then $B
(\ptC,\ptC,X)$ has the homotopy type of a cofibrant $C$-space.
\end{Lemma}

\begin{proof}
With our hypotheses, it follows from Proposition \ref{t-pr-cx-cw} that
the spaces $\ptC^{n}X$ have the homotopy type of cofibrant
$C$-spaces.  By Proposition \ref{t-pr-cof-c-alg-cof-ptd-space}, the
simplicial space $B_{\bullet} (\ptC,\ptC,X)$ is proper.  Finally, we
apply an argument analogous to that of Theorem X.2.7 of \cite{EKMM} to
show that if $Y_{\bullet}$ is a proper $C$-space in which each level
has the homotopy type of a cofibrant $C$-space, then $|Y_{\bullet}|$
has the the homotopy type of a cofibrant $C$-space.
\end{proof}





Finally, we consider the top horizontal map in \eqref{eq:14}.   We have
\emph{isomorphisms} of simplicial $C$-spaces
\[
      \Omega^{f}\Sigma^{f}B_{\bullet} (\ptC,\ptC,X) \iso B_{\bullet}
(\Omega^{f}\Sigma^{f}\ptC,\ptC,X) \iso B_{\bullet}
(\Omega^{f}\sinf,\ptC,X) \iso B_{\bullet} (Q,\ptC,X)
\]
(we used the isomorphism $\Sigma^{f}\ptC\iso \sinf$ of Lemma
\ref{t-le-Sigma-f-C-sinf}), and so an isomorphism of $C$-spaces
\[
      B (Q,\ptC,X) \iso |\Omega^{f}\Sigma^{f} B_{\bullet} (\ptC ,\ptC ,X)|
\]
We then apply the following result from \cite{MR0339152}.

\begin{Lemma} \label{t-le-gp-completion}
Let $C$ be a unital $\einfty$ operad, equipped with a
map of monads
\[
  f:   \ptC  \to \linf \sinf.
\]
Let $X$ be a $C$-space (and so pointed via $C (0)\to X$).   Suppose that
the base point of $C (1)$ and the base point of $X$ are
non-degenerate.  Then the map
\[
    B (\ptC,\ptC,X) \to  B (Q,\ptC,X),
\]
and so
\[
   B (\ptC,\ptC,X) \to |\Omega^{f}\Sigma^{f} B_{\bullet} (\ptC ,\ptC ,X)|,
\]
is group-completion.
\end{Lemma}

\begin{proof}
The point is that in general
\[
    \ptC Y \to \linf \sinf Y
\]
is group-completion \cite{MR0339176,MR0436146,MR0402733}, and so we
have the level-wise group completion
\[
     \ptC (\ptC )^{n}X \to \linf \sinf (\ptC)^{n}X
\]
(see \cite{MR0339152}).

The argument requires the simplicial spaces involved to be ``proper,''
that is, Reedy cofibrant with respect to the Hurewicz/Str\o{}m model
structure on topological spaces, so that the homology spectral
sequences have the expected $E_{2}$-term.  May proves that they are,
provided that $(C (1),\ptspace)$ and $(X,\ptspace)$ are NDR-pairs.
\end{proof}

%
We can now finish the proof of Theorem~\ref{t-th-einfty-spaces-spectra}.

\begin{proof}
It remains to show that if $X$ is a group-like cofibrant $C$-algebra
and $V$ is a (fibrant) $(-1)$-connected spectrum, then a map
\[
      \phi: \Sigma^{f} X \to V
\]
is a weak equivalence if and only if its adjoint
\[
       \psi: X \to \Omega^{f} V
\]
is.  These two maps are related by the factorization
\[
    \psi: X \to \Omega^{f} \Sigma^{f} X \xra{\Omega^{f}\phi}
\Omega^{f} V.
\]
The unit of adjunction is a weak equivalence by Proposition
\ref{t-pr-unit-we-May}.  It follows that $\psi$ is a weak
equivalence if and only if $\Omega^{f}\phi$ is.  Certainly if $\phi$
is a weak equivalence, then so is $\Omega^{f}\phi.$ Since both
$\Sigma^{f}X$ and $V$ are $(-1)$-connected, if $\Omega^{f}\phi$ is a
weak equivalence, then so is $\phi.$
\end{proof}


\begin{Remark}
There is another perspective on
Theorem~\ref{t-th-einfty-spaces-spectra} which elucidates the role of
the ``group-like'' condition on $C$-spaces.  Let's define a map
\[
    \alpha: X\to Y
\]
of $C$-spaces to be a \emph{stable} equivalence if the induced map
\[
    \Sigma^{f}\alpha': \Sigma^{f} X' \to \Sigma^{f}Y'
\]
is a weak equivalence ($X'$ and $Y'$ are cofibrant replacements of $X$
and $Y$).  The``stable'' model structure on $C$-spaces is the
localization of the model structure we have been considering in which
the weak equivalences are the stable equivalences, and the
cofibrations are as before.

In this stable model structure a $C$-space is fibrant if and only if
it is group-like; compare the model structure on $\Gamma$-spaces
discussed in \cite{MR1670249,MR1806878}.  The homotopy category
associated with the stable model structure is exactly $\ho
\gplike{C}{\spaces}$, and so this is a better encoding of the homotopy
theory of $\aC$-spaces. We have avoided discussing this approach in
detail in order to minimize technical complications, as we do not need
it for the applications.
\end{Remark}

\begin{Remark}
In \cite{MR1806878} it is shown that
\[
   \ho \CatOf{group-like $\Gamma$-spaces} \iso \ho \CatOf{connective
spectra}.
\]
Rekha Santhanam \cite{Santhanam:Thesis} has shown that the work of May
and Thomason \cite{MR508885} can be used to prove that the category of
$C$-spaces is Quillen equivalent to the category of $\Gamma$-spaces.
These two results give another proof of the equivalence
\[
    \ho \gplike{C}{\spaces} \heq \ho \CatOf{connective spectra}.
\]
\end{Remark}

\subsection{Units: proof of Theorem \ref{t-th-units}}
\label{sec:units:-proof-theorem}

Let $C$ be unital $\einfty$ operad, equipped with a map of operads
\[
         C \to \Lin,
\]
a map of monads on pointed spaces
\[
         f: \ptC \to \linf \sinf,
\]
and satisfying the hypotheses of Theorem
\ref{t-th-einfty-spaces-spectra}.  For example, we can take $C$ to be
\[
     C = | \SingTxt  (\mathcal{C}\times \Lin)|,
\]
the geometric realization of the singular complex on the product
operad $\mathcal{C}\times \Lin$, where $\mathcal{C}$ is infinite
little cubes operad of Boardman and Vogt \cite{MR0420609}.

Then we have a sequence of continuous adjunctions (the left adjoints
are listed on top, and connective Quillen equivalence is indicated by
$\Qeq$).
\[
\xymatrix{ {\splus \linf: \CatOf{$(-1)$-connected spectra}}
\ar@<-.3ex>[r]_-{\Omega^{f}} & {\gplike{C}{\spaces}}
\ar@<-.3ex>[l]_-{\Sigma^{f},\Qeq} \ar@<.3ex>[r] & {\alg{C}{\spaces}}
\ar@<.3ex>[l]^-{\GLsym} \ar@<.3ex>[r]^-{\splus} & {\alg{C}{\spectra}:
\glsym} \ar@<.3ex>[l]^-{\linf} }
\]
By Proposition \ref{t-pr-coeo-lms-spectra}, $\alg{C}{\spectra}$ is a
model for the category of $\einfty$ spectra.  This completes the proof
of Theorem \ref{t-th-units}.

\section{$\einfty$ Thom spectra and orientations}
\label{sec:thom-spectra}


With the adjunction of Theorem \ref{t-th-units} in hand, one can
construct and orient $\einfty$ Thom spectra as described in
\S\ref{sec:obstr-theory-einfty}, where we emphasize the more novel
case of the Thom spectrum associated to a map of spectra
\[
     b \to b\gl{R}.
\]
We add a few additional remarks here, emphasizing the classical
$\einfty$ Thom spectra associated to maps of spectra $b\to b\gl{S}.$

\subsection{Commutative $S$-algebra Thom
spectra}\label{sec:comm-s-algebra}

We write $S$ for the sphere spectrum, $b\gl{S}$ for $\Sigma \gl{S}$,
and $B\GL{S}$ for $\linf b\gl{S}$.  $B\GL{S}$ is the classifying space
for stable spherical fibrations.  Theorem \ref{t-th-units} gives a map
(in $\Ho \einftyspectra$)
\[
   \epsilon: \splus \linf \gl{S} \to S.
\]

Given a map
\[
    \zeta: b\to b\gl{S},
\]
let $j=\Sigma^{-1}\zeta: g=\Sigma^{-1}b \to \gl{S},$ and form the
diagram
\[
\xymatrix{ {g} \ar[r]^-{j} \ar[d] & {\gl{S}} \ar@{=}[r] \ar[d] &
{\gl{S}} \ar[d]
\\
{\ptspace} \ar[r] & {Cj} \ar[r] \ar[d] & {e\gl{S} \heq \ptspace}
\ar[d]
 \\
& {b} \ar[r] & {b\gl{S}} }
\]
by requiring that the upper left and bottom right squares are homotopy
Cartesian.  Note that we may also view $b$ as an infinite loop map
\[
     f: B \xra{} B\GL{S}.
\]

As in Definition \ref{def-5}, the \emph{Thom spectrum} of $f$, or of
$\zeta$, or of $j$, is the homotopy pushout $M=M\zeta$ in the diagram
of $\einfty$ spectra
\begin{equation} \label{or-eq:8}
\begin{CD}
\splus \linf g @> \splus \linf j >>   \splus \linf \gl{S} @> \epsilon >> S \\
@V \splus \linf \ptspace VV @VVV  @VVV \\
S = \splus \linf \ptspace @>>> \splus \linf Cj @>>> M,
\end{CD}
\end{equation}
which is to say that
\[
       M \cong \splus \linf Cj \Smash^{L}_{\splus \linf \gl{S}} S
\cong S \Smash^{L}_{\splus \linf g} S.
\]
(In \S\ref{sec:comp-thom-spectra} we compare this notion to classical
definitions).
Note that when writing this homotopy pushout, we are suppressing the
choice of a point-set representative of the homotopy class $\epsilon$.
Since all objects are fibrant in the model structure of
Proposition~\ref{t-pr-C-spectra-level-model}, it suffices to choose a
cofibrant model for $\linf \gl{S}$ (and subsequently of $\linf g$).

Now suppose that $R$ is an $\einfty$ spectrum with unit $\iota: S\to
R$; let $i=\gl{\iota}$, and let $k = ij: g\to \gl{R},$ so that we have
the solid arrows of the diagram
\begin{equation}  \label{eq:10}
\xymatrix{ {g} \ar[r]^-{j} \ar[dr]_{k} & {\gl{S}} \ar[r] \ar[d]_{i} &
{Cj} \ar@{-->}[dl]^{u}
\\
& {\gl{R},} }
\end{equation}
in which the row is a cofiber sequence.  The homotopy pushout diagram
\eqref{or-eq:8} and the adjunction of Theorem \ref{t-th-units} gives
the following.

\begin{Theorem} \label{t-obstruction}
Each of the squares in the commutative diagram of derived mapping
spaces
\begin{equation} \label{eq:9}
\begin{CD}
\einftyspectra  (M,R) @>>> \spectra (Cj,\gl{R}) @>>> *\\
@VVV @VVV @VVV\\
\{ i \} @>>> \spectra (\gl{S},\gl{R}) @>>> \spectra (g,\gl{R}).
\end{CD}
\end{equation}
are homotopy cartesian.  That is, the map $k$ is the obstruction to
the existence of an $\einfty$ map $M\to R$, and $\einftyspectra (M,R)$ is
weakly equivalent to the space of lifts in the diagram \eqref{eq:10}.
\end{Theorem}

If $\einftyspectra (M,R)$ is non-empty (i.e. if $i$ is homotopic to the
trivial map $g\to\gl{R}$) then we have equivalences of derived mapping spaces
$$
\einftyspectra(M,R)\heq\Omega\spectra(g,\gl{R})\heq\spectra(b,\gl{R})
\heq \einftyspectra (\splus B,R);
$$
this is an $E_\i$ analogue of the usual Thom isomorphism.

\subsection{$R$-algebra Thom spectra}

More generally, suppose that $R$ is an $\einfty$ ring spectrum.  Given
a map
\[
   \zeta: b \to b\gl{R},
\]
we obtain a map of cofiber sequences
\[
\xymatrix{
g\ar[r]\ar[d] & \gl{R}\ar[d]\\
{\gl{R}} \ar@{=}[r] \ar[d] & {\gl{R}} \ar[d]
 \\
{p} \ar[r] \ar[d] \ar@{-->}[ur] & {e\gl{R}\heq \ptspace } \ar[d]
 \\
{b} \ar[r]^-{\zeta} \ar@{-->}[ur] & {b\gl{R}.}  }
\]
in which $g=\Sigma^{-1} b$ and $p$ is the fiber of $b\to b\gl{R}$.

\begin{Definition} \label{def-4}
The \emph{$R$-algebra Thom spectrum} of $\zeta$ is the $\einfty$
$R$-algebra $M$ which is the pushout in the diagram of $\einfty$
spectra
\[
\begin{CD}
\splus \linf g @>>>  \splus \linf \gl{R} @>>> R \\
@VVV @VVV @VVV \\
\splus \linf * @>>> \splus \linf p @>>> M.
\end{CD}
\]
\end{Definition}

If $\zeta$ factors as
\[
     b \xra{\zeta'} b\gl{S} \xra{b\gl{\iota}} b\gl{R},
\]
then $M\zeta$ is the derived smash product
\[
       M\zeta' \Smash^{L}_{S} R,
\]
and so the following result is a generalization of Theorem
\ref{t-obstruction}.

\begin{Theorem} \label{t-obstruction-R}
Let $A$ be a commutative $R$-algebra, and write
$$
i:\gl{R}\to\gl{A}
$$
for the induced map on unit spectra.  Then each of the squares in the
commutative diagram
\begin{equation}
\begin{CD}
\CatOf{$\einfty$ $R$-$\mathrm{algebras}$} (M,A) @>>> \spectra (p,\gl{A}) @>>> *\\
@VVV @VVV @VVV\\
\{ i \} @>>> \spectra (\gl{R},\gl{A}) @>>> \spectra(g,\gl{A})
\end{CD}
\end{equation}
is homotopy cartesian.
\end{Theorem}

Taking $A=R$, we see that the space of $R$-algebra orientations of
$M\zeta$ is the space of lifts
\[
\xymatrix{ & {e\gl{R}} \ar[d]
\\
{b} \ar[r]_-{\zeta} \ar@{-->}[ur] & {b\gl{R}.}  }
\]
In this form the obstruction theory generalizes to the associative
case.  We discuss this generalization operadically in
\S\ref{sec:ainfty-thom-spectrum} and again using quasicategories in
\S\ref{sec:units-via-infty}.

\section{$\ainfty$ Thom spectra and orientations}

\label{sec:ainfty-thom-spectrum}

\subsection{Sketch of the construction}
\label{sec:sketch-construction}

In \S\ref{sec:space-units-orient}, we outlined how one might
generalize the study of orientations of $\einfty$ ring spectra in
\S\ref{sec:constr-may-quinn} and \S\ref{sec:thom-spectra} to the
associative case.  We briefly review what was proposed there.  We have
adjunctions
\begin{equation} \label{eq:ainfadj}
\xymatrix{ {\gplikeainftyspaces} \ar@<.5ex>[r] & {\ainftyspaces}
\ar@<.5ex>[r]^-{\splus} \ar@<.5ex>[l]^-{\GLsym} & {\ainftyspectra:
\GLsym.}  \ar@<.5ex>[l]^-{\linf} }
\end{equation}
Moreover, if $R$ is an $\ainfty$ spectrum, then $\splus$ and $\linf$
induce continuous Quillen adjunctions
\[
\xymatrix{ {\splus: \lRmodules} \ar@<.5ex>[r] & {\Rmodules: \linf.}
\ar@<.5ex>[l] }
\]

Using the fact that $\GL{R}$ is a group-like $\ainfty$ space, one
would like to form a ``principal $\GL{R}$-bundle''
\begin{equation} \label{eq:39}
  \GL{R} \rightarrow E\GL{R} \rightarrow B\GL{R},
\end{equation}
so that $E\GL{R}$ is a right $\ainfty$ $\GL{R}$-module.  Given a map
of spaces
\begin{equation} \label{or-eq:11}
  f: X\to B\GL{R},
\end{equation}
we would pull back this bundle as in the diagram
\begin{equation} \label{eq:40}
\begin{CD}
\GL{R} @= \GL{R} \\
@VVV @VVV \\
P @>>> E\GL{R} \\
@VVV @VVV \\
X @> f >> B\GL{R}.
\end{CD}
\end{equation}
Then $P$ should be a right $\ainfty$ $\GL{R}$-module, and so $\splus
P$ should be a right $\splus \GL{R}$-module.  The Thom spectrum of $f$
is to be the derived smash product
\[
        Mf = \splus P \Smash^{L}_{\splus \GL{R}} R.
\]
We would then have a weak equivalence
\[
 \Rmodules (Mf,R) \heq \CatOf{$\GL{R}$--modules} (P, \linf R).
\]
With respect to this isomorphism, the space of \emph{orientations} of
$Mf$ should be the homotopy pull-back in the diagram
\[
\begin{CD}
\CatOf{orientations} (M,R) @> \heq >>
\CatOf{right $\GL{R}$--spaces} (P,\GL{R})  \\
@VVV @VVV \\
\CatOf{right $R$--modules} (M, R) @> \heq >> \CatOf{right
$\GL{R}$--spaces} (P,\linf R),
\end{CD}
\]
and we should have weak equivalences
\begin{align} \label{eq:24}
\CatOf{orientations} (M,R) & \heq \CatOf{right $\GL{R}$--spaces}
(P,\GL{R}) \\
& \heq \spaces_{/B\GL{R}} (B,E\GL{R}). \label{eq:48}
\end{align}

The difficulties in making this sketch precise arise from the fact
that $\GL{R}$ is not a topological group but rather only a group-like
$\ainfty$ space.  This means for example it is more delicate to form
the space $P$ on which $\GL{R}$ will act, in such a way that we can
prove the homotopy equivalence \eqref{eq:48}.

In this section we present an approach to carrying out the program
described above.  The essential strategy is to adapt the operadic smash
product of \cite{MR1361938,EKMM} to the category of spaces.
Specifically, we produce a symmetric monoidal product on a subcategory
of $\spaces$ such that monoids for this product are precisely
$\ainfty$-spaces; this allows us to work with models of $\GL{R}$
which are strict monoids for the new product.  The observation that
one could carry out the program of \cite{EKMM} in the setting of
spaces is due to Mike Mandell, and was worked out in the thesis of the
second author \cite{Blumberg-thesis}.  In this paper we present a
streamlined exposition covering the part of the theory we need for our
applications.  The interested reader should consult the forthcoming
paper \cite{Blumberg-Cohen-Schlichtkrull} for further discussion and
in particular proofs of the foundational theorems stated below.

\subsection{$\L$-spectra and $\L$-spaces} \label{sec:l-spectra-l}

We mimic the definitions of \cite[\S I]{EKMM}.  Fix a universe
$\universe{U}$ (a countably infinite-dimensional real vector space),
and let $\Lin(1)$ denote the space of linear isometries
$\universe{U} \to \universe{U}$.  As the notation suggests, this is
the first piece of the linear isometries operad.

There is a monad $\L$ on spaces with
\[
   \L X = \Lin (1) \times X,
\]
where the product map comes from the composition on $\Lin(1)$ and the
unit from the inclusion of the identity map.  This is the space-level
analogue of the monad $\L$ on spectra defined as $\L Y = \Lin(1)
\ltimes Y$.  An $\L$-space is precisely a space with an action of
$\Lin(1)$.  In direct analogy with the commutative and associative
product on the category of $\L$-spectra \cite[\S I.5.1]{EKMM} we can
define an operadic product on $\L$-spaces:

\begin{Definition}
Let $X,Y$ be $\L$-spaces.  Define the operadic product $X\timesL Y =
\Lin (2)\times_{\Lin (1)\times \Lin (1)} X \times Y$ to be the
coequalizer in the diagram
\[
\xymatrix{ {\Lin (2)\times (\Lin (1)\times \Lin (1)) \times (X \times
Y)} \ar@<.3ex>[rr]^-{\gamma\times 1} \ar@<-.3ex>[rr]_-{1\times \xi} &
& {\Lin (2)\times X \times Y} \ar[r] & {X\timesL Y.}  }
\]
\end{Definition}
Here $\xi$ denotes the map using the $\L$-algebra structure of $X$ and
$Y,$ and $\gamma$ denotes the operad structure map $\Lin (2)\times
\Lin (1)\times \Lin (1)\to \Lin (2)$.  The left action of $\Lin (1)$
on $\Lin (2)$ induces an action of $\Lin (1)$ on $X\timesL Y.$

With this definition, many of the results and arguments of \cite[\S
I]{EKMM} carry over directly to the case of $\alg{\L}{\spaces}$.  For
instance, a result of the senior author (see \cite[\S I.5.4]{EKMM})
implies that $\timesL$ is associative and commutative:

\begin{Proposition}\label{t-pr-L-spaces-L-spectra}
\hspace{5 pt}
\begin{enumerate}
\item The operation $\timesL$ is associative.  Precisely, for any
$\L$-spaces $X_{1},\dotsc ,X_{k}$ and any way of associating the
product on the left, there is a canonical and natural isomorphism of
$\L$-spaces
\[
     X_{1}\timesL \dotsb \timesL X_{k} \iso \Lin (k)\times_{\Lin
(1)^{k}} X_{1}\times \dotsb \times X_{k}.
\]
\item The operation $\timesL$ is commutative in the sense that there
is a natural isomorphism of $\L$-spaces
\[
     \tau: X\timesL Y \iso Y\timesL X
\]
with the property that $\tau^{2}=1.$
\end{enumerate}
\end{Proposition}

There is a corresponding mapping space $\lmap (X,Y)$ which
satisfies the usual adjunction; in fact, the definition is forced by
the adjunctions.

\begin{Definition}
The mapping space $\lmap (X,Y)$ is the equalizer of the diagram
\[
\xymatrix{
\Map_{\alg{\L}{\spaces}}(\sL(2) \times X, Y) \ar[r]<2pt> \ar[r]<-2pt> &
\Map_{\alg{\L}{\spaces}}(\sL(2) \times \sL(1) \times \sL(1) \times X, Y).\\
}
\]
Here one map is given by the action of $\sL(1) \times \sL(1)$ on
$\sL(2)$ and the other via the adjunction
\[\Map_{\alg{\L}{\spaces}}(\sL(2) \times \sL(1) \times \sL(1) \times X, Y) \cong
\Map_{\alg{\L}{\spaces}}(\sL(2) \times \sL(1) \times X, \Map_{\alg{\L}{\spaces}}(\sL(1),
Y))\]
along with the action $\sL(1) \times X \to X$ and coaction
\[Y \to \Map_{\alg{\L}{\spaces}}(\sL(1),Y).\]
\end{Definition}

A diagram chase verifies the following proposition.

\begin{Proposition}
Let $X$, $Y$, and $Z$ be $\sL(1)$-spaces.  Then there is an adjunction
homeomorphism
\[\Map_{\alg{\L}{\spaces}}(X \timesL Y, Z) \cong \Map_{\alg{\L}{\spaces}}(X,
\lmap(Y,Z)).
\]
\end{Proposition}

The unital properties of $\timesL$ are again precisely analogous to
those spelled out in \cite[\S I.8]{EKMM}.  For a general $\L$-space,
there is always a unit map
\[
   \lambda: \ptspace\timesL X \to X.
\]
The unit map is compatible with $\timesL$.  Specifically, one can
adapt the arguments of \cite[\S I.8.5]{EKMM} to prove the following
proposition.

\begin{Proposition} \label{t-pr-L-space-unit}
\hspace{5 pt}
\begin{enumerate}
\item $\ptspace\timesL\ptspace\to\ptspace$ is an isomorphism.
\item For any $\L$-space $X$, the unit map $* \timesL X \to X$ is a
weak equivalence.
\item The unit map and $\timesL$ specify the structure of a weak
symmetric monoidal category on $\alg{\L}{\spaces}$.
\end{enumerate} 
\end{Proposition}

Recall that a weak symmetric monoidal category is a category with a
product that satisfies all of the axioms of a symmetric monoidal
category with the exception that the unit map is not required to be an
isomorphism \cite[\S II.7.1]{EKMM}.

The product $\timesL$ on $\alg{\L}{\spaces}$ is a version of the
cartesian product; in order to make a precise statement of the
relationship between $\timesL$ and $\times$, we need to discuss model
structures.

\begin{Proposition}
There is a compactly generated topological model structure on
$\alg{\L}{\spaces}$ in which
\begin{enumerate}
\item{The weak equivalences are the maps which are weak equivalences
of spaces,}
\item{The fibrations are the maps which are fibrations of spaces,}
\item{and the cofibrations are determined by the left-lifting
property.}
\end{enumerate}
The generating cofibrations and generating acyclic cofibrations are
the sets $\{\L A \to \L B\}$ for $A \to B$ a generating cofibration in
$\spaces$ and $\{\L C \to \L D\}$ for $C \to D$ a generating acyclic
cofibration in $\spaces$, respectively.
\end{Proposition}

The resulting model category is Quillen equivalent to spaces, since
$\sL(1)$ is contractible:

\begin{Proposition}
The free-forgetful adjunction induces
a Quillen equivalence between the usual model structure on $\spaces$
and the model structure on $\alg{\L}{\spaces}$ given in the preceding
proposition.
\end{Proposition}

Furthermore, we have the following key
comparison result, which says that the derived functor of $\timesL$ is
$\times$.

\begin{Proposition}
Let $X$ and $Y$ be cofibrant $\L$-spaces.
Then the natural map
\[
X \timesL Y \to X \times Y
\]
is a weak equivalence.
\end{Proposition}

The force of the construction of $\timesL$ is the fact that it gives
us control of $A_\infty$-spaces and $E_\infty$-spaces on the
``point-set'' level; just as in the setting of spectra, this makes it
simple to define monoids and modules, and more generally to carry out
definitions from homological algebra.  Specifically, define monads
$\T$ and $\PP$ on $\alg{\L}{\spaces}$ and $\A$ and $\E$ on $\spaces$
by the formulae
\[
\xymatrix{ \T X = \coprod_{n\geq 0} X^{\timesL n} && \relax\A X =
\coprod_{n\geq 0}
\Lin (n)\times X^{n} \\
\PP X = \coprod_{n\geq 0} X^{\timesL n} / \Sigma_n && \relax\E X =
\coprod_{n\geq 0} \Lin (n)\times_{\Sigma_n} X^{n}.  }
\]

Recall that $\T$-algebras in $\alg{\L}{\spaces}$ are monoids (i.e.
$\L$-spaces $X$ equipped with multiplication maps $X \timesL X \to X$
which are coherently associative and unital) and similarly
$\PP$-algebras are commutative monoids.  As discussed in
Section~\ref{sec:einfty-spaces}, $\A$-algebras in $\spaces$ are
$A_\infty$-spaces structured by the (non-$\Sigma$) linear isometries
operad and $\E$-algebras in $\spaces$ are $E_\infty$-spaces structured
by the linear isometries operad.  Just as in \cite[\S I.4.6]{EKMM},
these monads are closely related.

\begin{Proposition}\label{t-pr-assoc-L-spaces-assoc-L-spectra}
There are canonical isomorphisms
\[
\xymatrix{ \A \iso \T \L && \E \iso \PP \L }
\]
of monads on $\spaces$.
\end{Proposition}

Via \cite[\S I.6.1]{EKMM}, this has the following consequence.

\begin{Corollary}
\hspace{5 pt}
\begin{enumerate}
\item The categories of $\A$-algebras in $\spaces$ ($\ainfty$-spaces)
and of $\T$-algebras in $\alg{\L}{\spaces}$ are equivalent.
\item The categories of $\E$-algebras in $\spaces$ ($E_\infty$-spaces)
and $\PP$-algebras in $\alg{\L}{\spaces}$ are equivalent.
\end{enumerate}
\end{Corollary}

Finally, the work of \cite[p. 366]{LMS:esht} reviewed in
\S\ref{sec:einfty-spaces-einfty-1} and \cite{EKMM} implies that the
category of $\L$-spaces has the expected relationship to the category
of $\L$-spectra.  There is a subtle point here, however: for an
$\L$-space $X$, the Lewis-May suspension spectrum $\splus X$
admits two structures as an $\L$-space.  There is a trivial structure
described in \cite[\S I.4.5]{EKMM}, and the structure induced by the
isomorphism $\sL(1) \ltimes \Sigma^\infty (-) \cong \Sigma^\infty
(\sL(1) \times X)$.  In the following discussion, we always use the
latter.

\begin{Proposition}\label{t-pr-mbl-spaces-mbl-spectra}
\hspace{5 pt}
\begin{enumerate}
\item
If $X$ and $Y$ are $\L$-spaces, there is a natural isomorphism of
$\L$-spectra
\[
   \splus (X\timesL Y) \iso \splus X \SmashL \splus Y
\]
which is compatible with the commutativity isomorphism $\tau$.

\item The Quillen pair
\begin{equation}
       \splus: \spaces \leftrightarrows \spectra: \linf
\end{equation}
induces by restriction a continuous Quillen adjunction
\begin{equation}
       \splus: \alg{\L}{\spaces}\leftrightarrows \alg{\L}{\spectra}:
\linf
\end{equation}
between topological model categories.
\item If $X$ is a space then
\[
\xymatrix{
\splus \A X \iso \bigvee_{n} \Lin (n)\ltimes (\splus X)^{\smash n} &&
\splus \E X \iso \bigvee_{n} \Lin (n)\ltimes_{\Sigma_n} (\splus X)^{\smash n},
}
\]
and if $X$ is an $\L$-space then
\[
\xymatrix{
\splus \T X \iso \bigvee_{n} (\splus X)^{\SmashL n} && \splus \PP X
\iso \bigvee_{n} (\splus X)^{\SmashL n}/\Sigma_n.
}
\]
\item If $R$ is an $\ainfty$ spectrum, then $\linf R$ is a monoid in
$\alg{\L}{\spaces}$, and $\GL{R}$ is a group-like monoid in
$\alg{\L}{\spaces}$.  Similarly, if $R$ is an $\einfty$ spectrum, then
  $\linf R$ is a commutative monoid in $\alg{\L}{\spaces}$, and
  $\GL{R}$ is a group-like commutative monoid in $\alg{\L}{\spaces}$.
\end{enumerate}
\end{Proposition}

\subsection{$S$-modules and $*$-modules}

In order to work with modules over an $\ainfty$ space, it is
convenient to work with a symmetric monoidal category.  In this
section, we discuss the analogue of $S$-modules in the context of
$\L$-spaces.  Just as in \cite{EKMM}, one can restrict to the
subcategory of $\L$-spaces which are unital: $\L$-spaces $X$ such that
the unit map $* \timesL X \to X$ is an isomorphism.

\begin{Definition}
The category $\aM_*$ of $*$-modules is the subcategory of
$\sL(1)$-spaces such that the unit map $\lambda\colon * \timesL X \to
X$ is an isomorphism.  For $*$-modules $X$ and $Y$, define $X
\boxtimes Y$ as $X \timesL Y$ and $F_\boxtimes (X,Y)$ as $* \timesL
\lmap(X,Y)$.
\end{Definition}

The category $\aM_*$ is a closed symmetric monoidal category with unit
$*$ and product $\boxtimes$.  The inclusion functor $\aM_* \to
\alg{\L}{\spaces}$ has both a left and a right adjoint: the right
adjoint from $\L$-spaces to $*$-modules is given by $* \timesL (-)$,
and the left adjoint by $\lmap (*, -)$, just as in
the stable setting the inclusion functor from $S$-modules to
$\L$-spectra has both a left and a right adjoint \cite[\S II.2]{EKMM}.
Proposition 2.7 of \cite{EKMM} shows that these adjunctions are
respectively monadic and 
comonadic, and therefore as discussed in the proofs of \cite[\S
  II.1.4]{EKMM} and \cite[\S VII.4.6]{EKMM}, standard arguments
establish the following theorem.

\begin{Theorem}
The category $\aM_*$ admits a cofibrantly generated topological model
structure in which the weak equivalences are detected by the forgetful
functor to $\L$-spaces.  A map $f \colon X \to Y$ of $*$-modules
is a fibration if the induced map $F_\boxtimes(*, X) \to
F_\boxtimes(*,Y)$ is a fibration of spaces.  Colimits are created in
the category of $\L$-spaces, and limits are created by applying $*
\timesL (-)$ to the limit in the category of $\L$-spaces.
\end{Theorem}

The functor $* \timesL (-)$ is part of a Quillen equivalence between
$\alg{\L}{\spaces}$ and $\aM_*$.  This implies in particular that
there is a composite Quillen equivalence between $\spaces$ and
$\aM_*$.  In addition, just as for $\sL$-spaces, for cofibrant
$*$-modules $X$ and $Y$ there is a weak equivalence $X \boxtimes Y \to
X \times Y$.

Next, observe that the monads $\T$ and $\PP$ on $\L$-spaces
restrict to define monads on $\aM_*$.  The algebras over these monads
are monoids and commutative monoids for $\boxtimes$, respectively.  Thus, a
$\boxtimes$-monoid in $\aM_*$ is a $\timesL$-monoid in $\L$ which
is also a $*$-module.  The functor $* \timesL (-)$ gives us a means to
functorially replace $\ainfty$ and $\einfty$ spaces with
$\boxtimes$-monoids and commutative $\boxtimes$-monoids which are
weakly equivalent as $\ainfty$ and $\einfty$ spaces respectively.
Standard lifting techniques provide model structures on $\aM_*[\T]$
and $\aM_*[\PP]$ in which the weak equivalences and fibrations are
determined by the forgetful functor to $\aM_*$.

We are now in a position to define categories of modules.

\begin{Definition}\label{def-l1modules}
If $G$ is a monoid in $\aM_*$, then a $G$-module is a $*$-module $P$
together with a map
\[
    G \boxtimes P \to P
\]
satisfying the usual associativity and unit conditions.  We write
$\EKMM_{G}$ for the category of $G$-modules.
\end{Definition}

Once again, there is a model structure on the category of $G$-modules
in which the weak equivalences and fibrations are determined by the
forgetful functor to $*$-modules.

Let $\linf_{\sL}$ denote the composite functor $* \timesL \linf
F_{\sL}(S,-)$ from $S$-modules to $\aM_*$.  Again, recall that the
inclusion functor $\aM_* \to \alg{\L}{\spaces}$ has both a left and a
right adjoint: the right adjoint from $\L$-spaces to $*$-modules is
given by $* \timesL (-)$, and the left adjoint by $\lmap (*, -)$.  As
a consequence, the right adjoint of the functor $\splus$ from $\aM_*$
to $\aM_S$ turns out to be $\linf_{\sL}$.  In addition, recall that
this implies that to lift right adjoints from $\alg{\L}{\spaces}$ to
$\aM_*$ we forget from $\aM_*$ to $\alg{\L}{\spaces}$, apply the
functor, and then compose with $* \timesL (-)$.

In particular, for an object of $\aM_*[\T]$ or $\aM_*[\PP]$ we
compute $\GLsym$ as the composite $* \timesL \GLsym (-)$ and therefore
for an object of $\aM_S[\T]$ or $\aM_S[\PP]$ we find that $\GLsym$ is
computed as
\begin{equation}\label{stargl}
\GLsym R = * \timesL \GLsym \linf_{\sL} R.
\end{equation}
We have the following elaboration of
Proposition~\ref{t-pr-mbl-spaces-mbl-spectra}, connecting the category
$\aM_*$ to the category of EKMM $S$-modules.

\begin{Proposition}\label{t-pr-star-mod-mbl-s-mod}
\hspace{5 pt}
\begin{enumerate}
\item
If $X$ and $Y$ are $*$-modules, then is a natural isomorphism of
$S$-modules
\[
   \splus (X\boxtimes Y) \iso \splus X \Smash_{S} \splus Y
\]
which is compatible with the commutativity isomorphism $\tau$.  That
is, $\splus$ is a strong symmetric monoidal functor from $*$-modules
to $S$-modules.

\item
There is a continuous Quillen adjunction
\begin{equation}
       \splus: \aM_* \leftrightarrows \aM_S: \linf_{\sL}
\end{equation}
between topological model categories.
\item If $R$ is a associative $S$-algebra, then $\linf_{\sL}$ is
  a monoid in $\aM_*$ and $\GL{R}$ is a group-like monoid in $\aM_*$.
  Similarly, if $R$ is a commutative $S$-algebra, then $\linf_{\sL}$
  is a commutative monoid in $\aM_{*}$ and $\GL{R}$ is a group-like
  commutative monoid in $\aM_*$.

\item Let $G$ be a monoid in $*$-modules, and let $P$ be a $G$-module.
  Then $\splus P$ is an $\splus G$-module, and $\splus$ and
  $\linf_{\sL}$ restrict to give a continuous adjunction
\[
   \splus: \CatOf{$G$-modules} \leftrightarrows \CatOf{$\splus
     G$-modules}: \linf_{\sL}.
\]
\end{enumerate}
\end{Proposition}


\subsection{Principal bundles} \label{sec:principal-bundles}

Suppose that $G$ is a monoid in $*$-modules.  We will be interested in
studying ``classifying space'' constructions on $G$, and for this we
must say something about geometric realization of simplicial
$\L$-spaces and $*$-modules.

\begin{Proposition}
\hspace{5 pt}
\begin{enumerate}
\item If $X_{\bullet}$ is a simplicial $\L$-space, then its geometric
realization (as a simplicial space) $|X_{\bullet}|$ is a $\L$-space.
\item If $X_{\bullet}$ and $Y_{\bullet}$ are simplicial $\L$-spaces, then
there is a natural isomorphism of $\L$-spaces
\[
    |X_{\bullet} \timesL Y_{\bullet}| \iso |X_{\bullet}| \timesL
|Y_{\bullet}|.
\] 
\item If $X_{\bullet}$ is a simplicial $*$-module, then its geometric
realization (as a simplicial space) $|X_{\bullet}|$ is a $*$-module.
\item If $X_{\bullet}$ and $Y_{\bullet}$ are simplicial $*$-modules, then
there is a natural isomorphism of $*$-modules
\[
    |X_{\bullet} \boxtimes Y_{\bullet}| \iso |X_{\bullet}| \boxtimes
|Y_{\bullet}|.
\] 
\end{enumerate}
\end{Proposition}

Thus, we can form the $*$-modules
\[
\xymatrix{ \relax\EL{G} \eqdef |B_\bullet (\ptspace,G,G)| && \relax
\BL{G} \eqdef B_\bullet (\ptspace,G,\ptspace).  }
\]

That is, $\BL{G}$ is the geometric realization of the simplicial
$*$-module which has simplices
\begin{equation} \label{eq:18}
       [n] \mapsto G^{\boxtimes (n-1)}.
\end{equation}

The face maps are induced by the multiplication $G \boxtimes G \to G$,
and the degeneracies from the unit $* \to G$.

We now will establish that $\BL{G}$ is in fact a model of the usual
classifying space.  Associated to the map $\Lin(1) \to *$, there is
the change of monoids functor $Q$ from $\alg{\L}{\spaces}$ to
$\spaces$ given by $QX = * \times_{\Lin(1)} X$.  The behavior of $Q$
is described by the following proposition:

\begin{Proposition} \label{t-pr-Q-U}
\hspace{5 pt}
\begin{enumerate}
\item{$Q$ is a strong symmetric monoidal functor from $\L$-spaces to
  spaces.} \item{Let $U$ be the forgetful functor from $\L$-spaces to spaces.
There is a natural transformation $U \to Q$ which is a weak
equivalence for cofibrant objects of $\alg{\L}{\spaces}$, $\aM_*$,
and $\aM_*[\T]$.}
\end{enumerate}
\end{Proposition}

\begin{proof}
Let $X$ and $Y$ be $\sL(1)$-spaces.  To show that $Q$ is strong
symmetric monoidal, we need to compare $* \times_{\sL(1)} (X \timesL
Y)$ and $(* \times_{\sL(1)} X) \times (* \times_{\sL(1)} Y)$.  Observe
that $\sL(2)$ is homeomorphic to $\sL(1)$ as a left $\sL(1)$-space, by
composing with an isomorphism $U^2 \rightarrow U$.  Therefore we have
isomorphisms
\begin{eqnarray*}
* \times_{\sL(1)} (X \boxtimes_{\sL} Y) & = & * \times_{\sL(1)} \sL(2)
\times_{\sL(1) \times \sL(1)} (X \times Y) \\ & \cong & (*
\times_{\sL(1)} X) \times (* \times_{\sL(1)} Y)
\end{eqnarray*}
One checks that the required coherence diagrams commute.  The result
now follows, as $* \times_{\sL(1)} * \cong *$.

The second part of the proposition follows in each case by observing
that it suffices to work with cell objects, and then inducting over
the cellular decomposition.  The interested reader should consult
\cite{Blumberg-Cohen-Schlichtkrull} for details.
\end{proof}

As a consequence of the first property, $Q$ takes monoids in $\aM_*$
to topological monoids and $G$-modules in $\aM_*$ to $QG$-spaces.  The
second property allows us to retain homotopical control.  Let $G$ be a
cofibrant $\boxtimes$-monoid; equivalently, $G$ is an
$A_\infty$-space.  Denote by $G'$ a weakly equivalent topological
monoid produced by any of the standard rectification techniques.  Then
the Proposition and the fact that $Q$ evidently commutes with
geometric realization implies that
\[
\xymatrix{
Q(\BL{G}) \cong B(QG) \heq BG' & \text{and} & Q(\EL{G}) \cong
E(QG)\heq EG'.
}
\]

\begin{Remark}
Note however that as a consequence of the first property, nothing like
the second property can be true in the setting of commutative monoids
in $\aM_*$.  That is, $Q$ takes commutative $\boxtimes$-monoids to
commutative topological monoids.  Since we know that commutative
topological monoids have the homotopy type of a product of
Eilenberg-Mac Lane spaces, in general this rectification process
cannot be a weak equivalence.
\end{Remark}

Now, $\EL{G}$ is a right $G$-space, and the projection
\[
\pi: \EL{G} \to \BL{G}
\]
is a map of $G$-spaces if $\BL{G}$ is given the ``trivial'' action
\[
    \BL{G} \boxtimes G \to \BL{G} \boxtimes \ptspace \to \BL{G}.
\]

Furthermore, the comparison afforded by $Q$ allows us to deduce that
the projection map $\pi$ is a model for the universal quasifibration.

\begin{Proposition} \label{t-EL-BL-quasi-fib}
Let $G$ be a group-like cofibrant monoid in $\aM_*[\T]$ with a
nondegenerate basepoint.  Then the map $\pi \colon \EL{G} \to \BL{G}$
is a quasifibration of underlying spaces.
\end{Proposition}

\begin{proof}
By the remarks above, $Q\EL{G} \cong E(QG)$ and $Q\BL{G} \cong
B(QG)$.  By naturality, there is a commutative diagram
\[
\xymatrix{
U\EL{G} \ar[r] \ar[d]^{U\pi} & E(QG) \ar[d]^{Q\pi} \\
U\BL{G} \ar[r]^f & B(QG).\\
}
\]
For any $p \in U\BL{G}$, $(U\pi)^{-1}(p) = UG$, $(Q\pi)^{-1}(fp) =
QG$, and the map between them is induced from the natural
transformation $U \to Q$.  Writing $F(U\pi)_p$ for the homotopy fiber
of $U\pi$ at $p$ and $F(Q\pi)_{fp}$ for the homotopy fiber of
$Q\pi$ at $fp$, we have a commutative diagram
\[
\xymatrix{
UG \cong (U\pi)^{-1}(p) \ar[r] \ar[d] & F(U\pi)_p \ar[d] \\
QG \cong (Q\pi)^{-1}(fp) \ar[r] & F(Q\pi)_{fp},\\
}
\]
where the horizontal maps are the natural inclusions of the actual
fiber in the homotopy fiber.  The hypotheses on $G$ ensure that
the vertical maps are weak equivalences: on the left, this follows
directly from Proposition~\ref{t-pr-Q-U}, and on the right, we use the
fact that $U\EL{G} \to Q\EL{G}$ and $U\BL{G} \to Q\BL{G}$ are weak
equivalences since $U$ and $Q$ commute with geometric realization and
all the simplicial spaces involved are proper.  Furthermore, since
$QG$ is a group-like topological monoid with a nondegenerate
basepoint, $Q\pi$ is a quasifibration \cite[7.6]{May:class}, and so
the inclusion of the actual fiber of $U\pi$ in the homotopy fiber of
$U\pi$ is an equivalence.  That is, the bottom horizontal map is an
equivalence.  Thus, we deduce that the top horizontal map is an
equivalence and so that $U\pi$ is a quasifibration.
\end{proof}

Given a map of $*$-modules $f:X\to \BL{G}$, let $P$ be the pull-back
in the category of $G$-modules
\begin{equation} \label{eq:46}
\xymatrix{ {P} \ar[r] \ar[d] & {\EL{G}} \ar[d]^{\pi}
\\
{X} \ar[r]^{f} \ar@{-->}[ur] & {\BL{G}.}  }
\end{equation}
About this situation we have the following.

\begin{Theorem} \label{t-pr-princ-bundle-obstr-thy}
Suppose that $G$ is a cofibrant group-like monoid in $\aM_*$, and
$f$ is a fibration.  Then there is a natural zigzag of weak
equivalences between the derived mapping space
$\map_{\aM_*/\BL{G}}(f,\pi)$ of lifts in the diagram
\eqref{eq:46} and the derived mapping space
$\map_{\aM_G}(P,G)$.
\end{Theorem}

\begin{proof}
We will deduce this result from the corresponding result for
group-like monoids, using the functorial rectification process
provided by the functor $Q$.  Although this theorem in the classical
setting is folklore, only recently have modern proofs appeared in the
literature \cite{MR2380927}.  We discuss the situation in
Appendix~\ref{sec:princ-bundl-lift}, where the result appears as
Corollary~\ref{corliftmon}.

It is straightforward to verify that for cofibrant $G$, $Q$ induces a
Quillen equivalence between $\aM_*/\BL{G}$ and $\spaces/B
(QG)$, and so there is an equivalence of derived mapping spaces
\[
\map_{\aM_* / \BL{G}}(X, \EL{G}) \heq \map_{\spaces /
B(QG)}(QX, E(QG)).
\]
If $G$ is group-like, then $QG$ is a group-like topological monoid
which has the homotopy type of a $CW$-complex and a nondegenerate
basepoint, and so Corollary~\ref{corliftmon} gives us a weak
equivalence of derived mapping spaces
\[
\map_{\spaces / B(QG)}(QX, E(QG))\heq\map_{(QG) \spaces}(P', QG),
\]
where $P'$ is the homotopy pullback in the diagram
\[
\xymatrix{
P' \ar[r] \ar[d] & B(QG) \ar[d]\\
QX \ar[r] & E(QG).
}
\]

It is similarly straightforward to show that for cofibrant $G$, $Q$
induces a Quillen equivalence between $G\aM_*$ and
$QG\spaces,$ and so there is an equivalence of derived mapping spaces
\[
  \map_{G \aM_*}(P, G) \to \map_{(QG) \spaces}(QP, QG).
\]
The proof of the theorem will be complete once we have shown that $QP$
is naturally weakly equivalent to $P'$ as a $QG$-space.  But this
follows because $Q$ preserves homotopy limits up to a zigzag of
natural weak equivalences.
\end{proof}

\subsection{Thom spectra}
\label{sec:thom-spectra-2}

Now suppose that $R$ is a cofibrant EKMM $S$-algebra; by forgetting
the unit homeomorphism, an $\ainfty$ ring spectrum.  For the work of
this section, we need a model of $\GL{R}$ such that $\splus \GL{R}$ is
cofibrant as an $S$-algebra and there is an action of $\splus \GL{R}$
on $R$ which makes $R$ into a left $\splus \GL{R}$-module.

Based on Proposition~\ref{t-pr-mbl-spaces-mbl-spectra}, we observe
that the adjunction described in equation~\ref{eq:ainfadj} passes
through $\alg{\L}{\spaces}$ and takes values in $\ainfty$ ring
spectra, regarded as monoids under $\sma_{\sL}$ which are not
necessarily unital.  Therefore, we can use the adjunction in which
$\GLsym$ participates to obtain the desired action and promote to
$S$-algebras.


Even better, Proposition~\ref{t-pr-star-mod-mbl-s-mod} implies that
we have a version $\GL{R}$ which produces a group-like monoid in
$\aM_*$ such that $R$ obtains the structure of a $\splus
\GL{R}$-module directly from the defining adjunction.  That is, we
have the following structured version of the adjunction \ref{eq:ainfadj}:
\begin{equation} \label{eq:Lainfadj}
\xymatrix{ (\aM_*[\T])^{\times} \ar@<.5ex>[r] & {\aM_*[\T]}
\ar@<.5ex>[r]^-{\splus} \ar@<.5ex>[l]^-{\GLsym} & {\aM[\T]:
\GLsym.}  \ar@<.5ex>[l]^-{\linf_{\sL}} }
\end{equation}
where here recall that $\GLsym \colon \aM_*[\T] \to
(\aM_*[\T])^{\times}$ is given by equation~\ref{stargl}.
The adjoint of the identity $\GLsym R \to \GLsym R$ under this
adjunction is the map $\splus \GLsym R \to R$.  Taking a cofibrant
replacement $(\GL{R})^c$ in the category of monoids in $\aM_*$, we
have a map of $S$-algebras
\[
\splus (\GL{R})^c \to \splus \GL{R} \to R.
\]

For the remainder of the section, we will abusively refer to such a
model as $\GL{R}$.

\begin{Definition} \label{def-thom-ainfty}
The \emph{Thom spectrum} of $f$ is the derived smash product in the
homotopy category of $R$-modules
\[
Mf  \eqdef \splus P \Smash^{L}_{\splus \GL{R}} R.
\]
\end{Definition}

It is often useful to have a specific point-set model for this derived
smash product which involves only operations on $f$.  For this
purpose, we sometimes prefer to work with a model of $Mf$ given by
replacing a given map $f$ with a cofibrant-fibrant replacement in the
model structure on $\GL{R}$-modules over $\BL{\GL{R}}$.  While the
resulting spectrum $\splus P$ will not necessarily be cofibrant as a
$\splus \GL{R}$-module, it can be shown to be an extended cell module
(in the sense of \cite{EKMM}); this is a flatness condition that
ensures that smashing with $\splus P$ computes the correct homotopy
type.  Alternatively, we can replace $R$ as a $\splus \GL{R}$-$R$
bimodule by a cofibrant bimodule object $R'$; it then suffices to work
with a fibration to compute the derived functor of the Thom spectrum.

To describe orientations in this setting, we first observe that, by
construction, $Mf$ is a right $R$-module, and if $T$ is a right
$R$-module then there is a natural weak equivalence of derived mapping
spaces
\[
 \EKMM_{R} (Mf,T)\heq \EKMM_{\splus\GL{R}} (\splus P,T).
\]
By Proposition \ref{t-pr-star-mod-mbl-s-mod} there is a further
adjoint weak equivalence of derived mapping spaces
\[
\EKMM_{R} (Mf,T) \heq \EKMM_{\GL{R}} (P,\linf_{\sL} T).
\]
For example taking $T=R$ we have
\begin{equation} \label{eq:19}
\EKMM_{R} (Mf,R) \heq \EKMM_{\GL{R}} (P,\linf_{\sL} R).
\end{equation}

\begin{Definition} \label{def-orientation}
The space of \emph{orientations} of $Mf$ is the subspace of
components of $\EKMM_{R} (Mf,R)$ which correspond to
\[
\EKMM_{\GL{R}} (P,\GL{R}) \subseteq \EKMM_{\GL{R}}(P,\linf_{\sL} R)
\]
under the adjunction \eqref{eq:19}.  That is, we have a pull-back
diagram
\begin{equation} \label{eq:71}
\xymatrix{ {\CatOf{orientations} (Mf,R)} \ar[r]^-{\heq}
\ar@{>->}[d] & {\EKMM_{\GL{R}} (P,\GL{R})} \ar@{>->}[d]
\\
{\EKMM_{R} (Mf,R)} \ar[r]^-{\heq} & \EKMM_{\GL{R}}
(P, \linf_{\sL}{R}).}
\end{equation}
\end{Definition}

With this definition, Theorem \ref{t-pr-princ-bundle-obstr-thy}
implies the following.

\begin{Theorem} \label{t-th-orientation-ainfty}
The space of orientations of $Mf$ is weakly equivalent to the space
of lifts in the diagram \eqref{eq:46}.  In particular, the spectrum
$Mf$ is orientable if and only if $f:X\to \BL{\GL{R}}$ is null
homotopic. 
\end{Theorem}

To make contact with familiar notions of orientation, we'll be more
explicit about the adjunctions in Definition
\ref{def-orientation}.  For this it it helpful to observe that the
Thom spectrum of a point is equivalent to $R.$

\begin{Lemma} \label{t-le-thom-spectrum-point}
The Thom spectrum of
\[
    \{q \} \to X \to \BL\GL{R}
\]
is weakly equivalent to $R.$
\end{Lemma}

\begin{proof}
Let $\{q \} \to \BL\GL{R}$ be the inclusion of a point.  The
instructions in Definition \ref{def-thom-ainfty} tell us that the
Thom spectrum is $\splus P \Smash^{L}_{\splus \GL{R}} R,$ where $P$ is
the homotopy pull-back in
\begin{equation*}
\xymatrix{
&
{P}
 \ar[r]
 \ar[d]
&
{\EL\GL{R}}
 \ar[d]\\
{\{q \}}
 \ar[r]
&
{Z}
\ar[r]
&
\BL\GL{R},
}
\end{equation*}
and $\{q \}\to Z \to \BL\GL{R}$ is a fibrant replacement of $\{q \}$
in $*$-modules over $\BL\GL{R}.$ Since
by Proposition \ref{t-EL-BL-quasi-fib}, $\EL\GL{R}\to \BL\GL{R}$ is a
quasifibration (with fiber $\GL{R},$) it follows that $\GL{R}\heq P$
as $\GL{R}$-modules.  The result follows easily from this.
\end{proof}

\begin{Corollary}\label{t-co-Th-ELGL}
Since $\EL\GL{R}\heq \ptspace$, we have
\[
   M (\EL\GL{R} \to \BL\GL{R}) \heq R.
\]
\end{Corollary}

Now suppose that $f: X\to \BL\GL{R}$ is a fibration of $*$-modules,
and let $P$ be the pull-back in the diagram
\begin{equation} \label{eq:72}
\xymatrix{ {P} \ar[r] \ar[d] & {\EL{G}} \ar[d]^{\pi}
\\
{X} \ar[r]^{f} \ar@{-->}[ur]^{\tilde{a}} & {\BL{G},}  }
\end{equation}
and let $M = Mf$.  If $\tilde{a}$ is a lift as indicated,
then by passing to Thom spectra along $\tilde{a}$ we get a map of $R$-modules
\[
      a: M \to R
\]
which is the orientation associated to the lift $\tilde{a}.$

Conversely, suppose that $a:M \to R$ is a map of $R$-modules.  More
precisely, fix a cofibrant $\splus\GL{R}$-$R$-bimodule $\Rcf$, so that
the $R$-module  $\splus P \Smash_{\splus \GL{R}} \Rcf$
models $M$.  Each point $p\in P$ determines a $\GL{R}$-map
\[
 \GL{R}\to P
\]
and so a map of $R$-modules
\[
j_{p}:   \Rcf \heq \splus \GL{R}\Smash_{\splus \GL{R}} \Rcf \to M \to R.
\]
As $p$ varies the $j_{p}$ assemble into a map
\[
      P \to \EKMM_{R} (\Rcf,R).
\]
Put another way, we're studying the adjoint of the composite
\begin{align*}
\splus P & \longrightarrow F_{\splus \GL{R}}(\splus \GL{R}, \splus P)
\\ & \longrightarrow F_R(\splus \GL{R} \sma_{\splus \GL{R}} \Rcf, \splus P \sma_{\splus
  \GL{R}} \Rcf) = F_R(\Rcf,M) \xrightarrow{a} F_R(\Rcf,R).
\end{align*}

In Proposition \ref{inf-t-pr-Rwe-and-GL}, we show that
\[
\EKMM_{R} (\Rcf,R)\heq \linf_{\sL} R,
\]
and the resulting map
\[
   j: P \to \linf_{\sL} R
\]
corresponds to $a$ under the equivalence of derived mapping spaces
\[
  \EKMM_{R} (Mf,R) \heq \EKMM_{\GL{R}} (P,\linf_{\sL} R).
\]

Put another way,  for each $q\in X,$ Lemma
\ref{t-le-thom-spectrum-point} implies that
the Thom spectrum $M_{q}$ of
$\{ q \} \to X \to \BL\GL{R}$ is equivalent to $R.$
Passing to Thom spectra gives a map
\[
   i_{q}: M_{q} \to M \xra{a} R.
\]
By Lemma \ref{t-le-thom-spectrum-point}, $M_{q}$ is non-canonically
equivalent to $R$: indeed, a choice of point $p\in P$ lying over $q$
fixes an equivalence $\Rcf \heq M_{q}$ making the diagram
\[
\xymatrix{
{\Rcf}
\ar[dr]_{j_{p}}
\ar[rr]^{\heq}
&
&
{M_{q}}
 \ar[dl]^{i_{q}} \\
&
{R}
}
\]
commute.
Thus we have the following analogue of the standard description of
Thom classes as in for example \cite{MR643101}.

\begin{Proposition} \label{t-pr-characterize-orientations}
Suppose that $a:M \to R$ is a map of $R$-modules.  Then the following
are equivalent.
\begin{enumerate}
\item $a$ is an orientation.
\item For each $q\in X$, the map of
$R$-modules
\[
 i_{q}:  M_{q}\to M \xra{a} R
\]
is a weak equivalence.
\item For each $p\in P$, the map of $R$-modules
\[
j_{p}:     \Rcf \to M \xra{a} R
\]
is a weak equivalence.
\end{enumerate}
\end{Proposition}



We now move on to discuss the Thom isomorphism in this setting.  In
this part of the section we tacitly assume we are working a point-set
model of the Thom spectrum functor throughout, although this is
suppressed from the notation and discussion.  Suppose we are given an
orientation in the form of a $\GL{R}$-map $s:
P\to \GL{R},$ corresponding to an $R$-module map
\[
     a: M \to R,
\]
where $M$ is the Thom spectrum associated to a map of $*$-modules $f
\colon X \to \BL{\GL{R}}$.  The diagonal map $\Delta \colon X \to X
\times X$ induces a diagram
\begin{equation}\label{eq:thomdiag}
\xymatrix{
X \ar[r]^-\Delta \ar[d]^f & X \times X \ar[dl]^{f\pi_2} \\
\BL{\GL{R}} \\
}
\end{equation}
in the category of $*$-modules over $\BL{\GL{R}}$, where here $\pi_1$
is the projection onto the first factor.  Passing to Thom spectra, we
obtain the $R$-module Thom diagonal map
\[
M\xra{\Delta} \splus X \Smash M.
\]
Now, we can form the composite
\begin{equation}\label{eq:thomdiageq}
    M \xra{\Delta} \splus X \Smash M \xra{1 \Smash a} \splus X \Smash
R
\end{equation}
as in \cite{MR609673}.  To see directly that this map is a weak
equivalence, we proceed by analyzing the diagonal on the level of
$*$-modules.  Specifically, passing to pullbacks the diagram
\eqref{eq:thomdiag} gives rise to a map $P \to X \times P$ of
$\GL{R}$-modules, where the action on $X \times P$ is induced from the
actions on $X$ and $P$ via the composite
\[
\xymatrix{
(X \times P) \boxtimes Z \ar[r] & ((X \times *) \boxtimes Z) \times ((*
  \times P) \boxtimes Z) \ar[r] & (X \boxtimes Z) \times (P \boxtimes Z).
}
\]
Applying the map $s$, we obtain the composite map of $\GL{R}$-modules
\[
P \to X \times P \to X \times \GL{R}.
\]
Since $s$ corresponds to a section of the map $P \to X$ induced by the
universal property of the pullback, this composite is a weak
equivalence, and so there is an induced weak equivalence of
$\splus \GL{R}$-modules
\[
\splus P \heq X_+ \sma \splus \GL{R}.
\]
Now, passing to Thom spectra we have an induced weak equivalence of
$R$-modules
\[
M = \splus P \Smash_{\splus \GL{R}} R \heq (X_+
\Smash \splus \GL{R}) \Smash_{\splus \GL{R}} \Smash R \heq \splus X
\Smash R.
\]
It is straightforward to check that this weak equivalence is the same
as the composite \eqref{eq:thomdiageq}.  Summarizing, we have the
following proposition.

\begin{Proposition} \label{t-pr-thom-iso}
If $a: M\to R$ is an orientation, then the map of right $R$-modules
\[
     M \xra{\Delta} \splus X \Smash M \xra{1 \Smash a} \splus X \Smash
R
\]
is a weak equivalence. 
\end{Proposition}

\section{The space of units is the derived space of homotopy
automorphisms}

\label{sec:space-units-derived}

In the preceding sections, we associated to an $\ainfty$ or $\einfty$
ring spectrum $R$ the group-like $\ainfty$ or $\einfty$ space $\GL{R}$
which is the pull-back in the diagram
\[
\xymatrix{
\GL{R} \ar[r] \ar[d] & \linf R \ar[d] \\
(\pi_{0}R)^{\times} \ar[r] & \pi_{0}R.
}
\]
In this section, we connect this definition to a more conceptual
definition of $\GL{R}$ as a derived space of automorphisms of $R$.
Specifically, we show that this pullback description of $\GL{R}$ is
weakly equivalent to the derived space of homotopy automorphism of
$R$, considered as a module over itself.  One of the appealing facets
of the $\i$-categorical approach to these matters is that, as we shall
see in \S\ref{sec:units-via-infty}, it is this conceptual definition of
$\GL{R}$ which naturally presents itself.

Our analysis is closely related to the evident question of how to
define $\GL{R}$ in other modern categories of spectra (e.g. diagram
spectra).  We do not make any particular claim to novelty in this
section; in particular, May and Sigurdsson provide an excellent
discussion of the situation in \cite[\S 22.2]{MR2271789} (although
note that our use of $\End$ and $\Aut$ is slightly different than
theirs), and the conceptual description we describe is of course
implicit in the original definition in \cite{MQRT:ersers}.
Nonetheless, there are subtleties associated to the
interaction between cofibrant and fibrant replacements in categories
of commutative algebras and the underlying module categories that are
worth exposing.

To begin, let us be clear about what we mean by the space of homotopy
automorphisms of $R$ as an $R$-module.  Suppose that $\monspectra$ is
a symmetric monoidal simplicial or topological category of spectra,
and let $R$ be a monoid in $\monspectra,$ that is, an $S$-algebra.  By
the space of homotopy automorphisms of $R$, we mean the subspace of
$\Rmod (R,R)$ consisting of weak equivalences.  In order to make this
notion homotopically well-defined, we need to ``derive'' it in two
ways.  First, we should replace $R$ with a weakly equivalent
cofibrant-fibrant \emph{algebra} $R'$.  Then, we should find a
cofibrant-fibrant replacement $\Rcf$ of $R'$ as a module over itself.

\begin{Definition}\label{def-9}
If $R'$ is a cofibrant-fibrant algebra, and $M$ is a
cofibrant-fibrant $R'$-module, then the \emph{space of endomorphisms}
of $M$ is
\[
   \End (M) \eqdef \monspectra_{R'} (M,M).
\]
This is a monoid, and by definition the \emph{space of homotopy
automorphisms} of $M$ is the subspace of group-like components: that
is, $\Aut (M) = \GL{\End (M)}$ is the pull-back in the diagram
\[
\begin{CD}
    \Aut (M) @>>>  \End (M)  \\
@VVV @VVV \\
(\pi_{0} (\End (M))^{\times}) @>>> \pi_{0} \End (M).
\end{CD}
\]
Since $M$ is cofibrant and fibrant, we can equivalently define $\Aut
(M)$ to be the subspace of $\End (M)$ consisting of weak
equivalences.
If $R$ is an arbitrary algebra, then the \emph{derived space of
endomorphisms} of $R$ is the homotopy type
\[
\End (R) = \End (\Rcf) \eqdef \monspectra_{R'} (\Rcf,\Rcf),
\]
where $R'$ is a cofibrant-fibrant replacement of $R$ as an algebra,
and $\Rcf$ is a cofibrant-fibrant replacement of $R'$ as a module over
itself.  The \emph{derived space of homotopy automorphisms} of $R$ is the
homotopy type of the subspace
\[
\Aut (R) = \Aut (\Rcf) \subset \End (\Rcf) = \monspectra_{R'}(\Rcf,\Rcf).
\]
\end{Definition}

We have elected to use the notation $\Aut (R)$ for the space
of homotopy automorphisms of $\Rcf$, even though it is not a group,
just as we
have written $\GL{R}$ for the space of units, even though it is not a
group.
This is because both are groups in the $\i$-categorical sense, which is to say that they arise as $\i$-groupoids of automorphisms of objects in $\i$-categories;
equivalently, but from the homotopical point of view, they are loop spaces.
This notation is nearly inevitable in the
setting of $\i$-categories: as we shall see in
\S\ref{sec:units-via-infty}, in the $\infty$-category of
(cofibrant-fibrant) $R$-modules, the maximal $\i$-groupoid on the single
object $\Rcf$ is the space $B\Aut (\Rcf)$ which is a delooping of
$\Aut (\Rcf).$

As written, we have presented $\Aut (R)$ as a group-like topological
or simplicial monoid.  In practice, it is easier to access this
homotopy type if we let $R^{c}$ be a cofibrant replacement of $R'$,
and $R^{f}$ a fibrant replacement.   Then we have a homotopy
equivalence of spaces
\[
      \End (R) \heq \monspectra_{R'} (R^{c},R^{f}),
\]
with $\Aut (R)$ equivalent to the subspace of weak equivalences.

We shall compare this notion of $\GL{R}$ to the classical one, in the
setting of the $S$-modules of \cite{EKMM}.  Let $\spectra$ be the
Lewis-May-Steinberger category of spectra,  and let $\EKMM$ be the
associated category of $S$-modules.

\begin{Proposition} \label{inf-t-pr-Rwe-and-GL}
Let $R$ be a cofibrant and fibrant $S$-algebra or commutative
$S$-algebra in $\EKMM$.  Then the inclusion of derived mapping spaces
\[
  \Aut (R) \rightarrow \End (R)
\]
is a model for the inclusion
\[
  \GL{R} \rightarrow \linf R
\]
considered elsewhere in this paper.
\end{Proposition}

\begin{proof}
If $R$ is an associative or commutative $S$-algebra, then the
underlying $R$-module associated to $R$ will be fibrant.  Thus, we can
use $R$ for $R^{f}$.  In the notation of \cite{EKMM},  $S
\Smash_{\Lin} \L \sinf S$ is a cofibrant replacement for $S$ as an
$S$-module, and $R\Smash_{S}
\L\sinf S$  is a cofibrant replacement for $R$ as an $R$-module.  So
the derived mapping space $\EKMM_{R} (R^{c},R^{f})$ is given by
\begin{align*}
  \EKMM_{R}(R\Smash_{S}\L \sinf S^{0},R) & \iso
  \EKMM (S\Smash_{S}\L \sinf S^{0},R) \\
  & \iso
  \alg{\L}{\spectra} (\L\sinf S^{0},F_{\Lin} (S,R)) \\
& \iso     \spectra (\sinf S^{0},F_{\Lin} (S,R)) \\
  & \iso \linf F_{\Lin} (S,R).
\end{align*}

By \cite[\S{I}, Cor 8.7]{EKMM}, the natural map of $\L$-spectra
\[
    R \to F_{\Lin} (S,R)
\]
is a weak equivalence of $\L$-spectra, and so of spectra.
The weak equivalence
\[
   \EKMM_{R} (R\Smash_{S}\L \sinf S^{0},R) \heq \linf R
\]
follows since $\linf$ preserves weak equivalences.  It is then easy to
see that the subspace of $R$-module weak equivalences corresponds to
$\GL{R}$.
\end{proof}

The preceding proposition illustrates how useful it is that in the
Lewis-May-Steinberger and EKMM categories of spectra, an algebra or
commutative algebra $R$ is automatically fibrant as a module over
itself.  In particular, since $\GL{R}$ is identified as a subspace of
$\linf R$, it is straightforward to see how to identify the
multiplicative structure on $\GL{R}$.

In the setting of a category of diagram spectra $\C,$ (e.g. orthogonal
spectra), the
situation is somewhat more complicated.  For an associative
$S$-algebra $R$, one can carry out a similar analysis after passing to
a cofibrant-fibrant replacement of $R$ as an $S$-algebra, and the
pullback description of $\GL{R}$ in fact itself yields a genuine
topological monoid \cite[22.2.3]{MR2271789}.  But the situation for
commutative $S$-algebras in the diagrammatic setting is different.
The model structure on commutative $S$-algebras is lifted from the
positive model structure on (orthogonal) spectra, and in this
model structure the underlying $S$-module of a cofibrant-fibrant
commutative $S$-algebra will not be fibrant; indeed its zero space
will be $S^{0},$ and so
\[
      \C (S,R) = S^{0} \neq \C (S^{0},R^{f})
      \heq h\End (R).
\]
Of course, one can instead replace the given commutative $S$-algebra
by an associative $S$-algebra instead, but in this case it is
impossible to recover the $\einfty$ structure on $\GL{R}$.  To
describe $\GL{R}$ in this setting requires a different construction;
see \cite{MR2057776} or \cite{Lind} for a description.

The problem that arises above is a manifestation of Lewis's theorem
\cite{MR1124786} about the nature of symmetric monoidal categories of
spectra.  If $S = \sinf S^{0}$ is cofibrant (as it is in diagram
categories of spectra), then the zero space of a cofibrant-fibrant
commutative $S$-algebra must not be homotopically meaningful, as
otherwise we could make a cofibrant-fibrant replacement $S'$ of $S,$
and
\[
    \C (S,S') \heq QS^{0}
\]
would realize $QS^{0}$ as a commutative topological monoid.  On the
other hand, if the zero space of a cofibrant-fibrant commutative
$S$-algebra is homotopically meaningful, then $S$ cannot be cofibrant,
and the $(\sinf,\linf)$ adjunction must take a modified form (as it
does in the setting of EKMM spectra).

\section{Parametrized spectra, units and Thom spectra via $\infty$-categories}

\label{sec:units-via-infty}

\subsection{Introduction}

In this section, we show that the theory of $\i$-categories
developed by Joyal and Lurie provides a powerful technical and conceptual framework for the study of Thom
spectra and orientations.

In this setting, an $\ainfty$ ring spectrum $R$ has an associated
$\infty$-category $\Rmod$ of
(right)
$R$-modules.   We define an
\emph{$R$-line} to be a $R$-module $L$ which admits an equivalence
\[
   L \heq R.
\]
We define $\Rwe$ to be the sub-$\i$-category of $\Rmod$ in which the
objects are $R$-lines and in which the morphism space
\[
   \Rwe (L,M) \subset \Rmod (L,M)
\]
is the subspace of  equivalences.    As such $\Rwe$ is an
$\i$-groupoid, i.e. a Kan complex.  A \emph{trivialization} of an
$R$-line $L$ is an  equivalence
\begin{equation} \label{eq:54}
     L  \xra{\heq} R,
\end{equation}
and we define $\Rtriv$ to be the $\infty$-groupoid
\begin{equation} \label{eq:70}
    \Rtriv = \Rwe_{/R}
\end{equation}
of trivialized $R$-lines.  The forgetful map
\[
      \Rtriv \to \Rwe
\]
is a Kan fibration, and is our model for the fibration
\[
     E\GL{R} \to B\GL{R}.
\]
If $X$ is a Kan complex, then a map (equivalently of simplicial sets
or $\i$-categories)
\begin{equation} \label{eq:15}
f:   X \to \Rwe
\end{equation}
is a family of $R$-lines parametrized by $X.$  The Thom spectrum of $f$
is just the ($\i$-categorical) colimit
\[
    Mf \eqdef \colim (X \xra{f} \Rwe  \xra{} \Rmod).
\]
Using this definition and the description \eqref{eq:70} of $\Rtriv$,
one sees that the space of orientations $Mf \to R$ is equivalent to
the space of lifts
\begin{equation}\label{eq:55}
\xymatrix{ & {\Rtriv}
 \ar[d]_\iota \\
{X} \ar@{-->}[ur] \ar[r]_-{f} & {\Rwe.}  }
\end{equation}

It is possible to develop the theory of Thom spectra and orientations
using only these observations, together with some basic facts about
$\i$-categories from \cite{math.CT/0608040} and about symmetric
monoidal model categories of spectra; in fact, this is our approach in
\S\ref{sec:i-cat-ii}.  We recommend that the reader who is unfamiliar
with $\i$-categories begin with that treatment before reading this
section.  Nonetheless, this geodesic approach to the construction
obscures some of the essential clarification provided by the
$\i$-categorical point of view.

For example, in this setting, we can make precise the slogan that
$\Rwe$ is the classifying space for bundles of $R$-lines.
Specifically, there is a universal bundle of $R$-lines $\mathscr{L}$
over $\Rwe,$ and the map of simplicial sets \eqref{eq:15} classifies
the bundle of $R$-lines $f^{*}\mathscr{L}$ over $X.$  Moreover a lift
in \eqref{eq:55} corresponds to an equivalence
\[
    f^{*}\mathscr{L} \to R_{X},
\]
where $R_{X}$ denotes the trivial bundle of $R$-lines over $X.$

In this section, we sketch the theory of bundles of $R$-modules, and
use it to discuss Thom spectra and orientations.  The story we tell
reflects the close connection between abstract homotopy theory and
$\i$-category theory arising from the fact that $\i$-groupoids are a
model for spaces.  However, given an existing theory of
$\i$-categories, it is possible to construct the $\i$-category of
$\i$-groupoids without any mention of spaces whatsoever.  We adopt
this approach here: beginning with the notion of $\i$-category, we
summarize and expand on ideas from
\cite{math.CT/0608040,DAGI,DAGII,DAGIII} to review the construction of
$\i$-categories of $\i$-groupoids (spaces), stable $\i$-groupoids
(spectra), and bundles of stable $\i$-groupoids (bundles of spectra).
With these foundations in place, we discuss Thom spectra and
orientations.

This approach has the advantage of being both self-contained and
concise, but it has the disadvantage of being less concrete and
potentially hiding the relationship to the established theory.  In
particular, as discussed in Remark~\ref{rem:omit}, in the interests of
expositional manageability we have chosen not to discuss the
comparison between the approach to parametrized spectra discussed in
this section and the parametrized spectra of May-Sigurdsson; this
comparison will appear in a future paper.

However, at the level of Thom spectra and orientations, we do take
pains to relate the abstract theory of this section to more concrete
approaches: as mentioned above, in \S\ref{sec:i-cat-ii} we take in
some sense the opposite approach and show how to develop the
$\i$-categorical approach to Thom spectra and orientations, starting
with an existing symmetric monoidal category of spectra such as
$S$-modules or symmetric spectra.  Furthermore, in
\S\ref{sec:comp-thom-spectra}, we provide a comparison between our
various approaches to Thom spectra.

Our story in this section is independent of the material in
\S\ref{sec:constr-may-quinn}---\S\ref{sec:ainfty-thom-spectrum}.

\subsection{$\i$-Categories and $\i$-Groupoids}

\label{sec:i-categories-i}

For the purposes of this paper, an $\i$-category will always
mean a quasicategory in the sense of Joyal \cite{MR1935979}.  This is
the same as a weak Kan complex in the sense of Boardman and Vogt
\cite{MR0420609}; the different terminology reflects the fact that
these objects simultaneously generalize the notions of category and of
topological space.  Specifically, recall that a quasicategory is a
simplicial set which satisfies certain lifting properties.

Given two $\i$-categories $\C$ and $\D$, the $\i$-category of functors
from $\C$ to $\D$ is simply the simplicial set of maps from $\C$ to
$\D$ as simplicial sets.  More generally, for any simplicial set $X$
there is an $\i$-category of functors from $X$ to $\C$, $\Fun(X,\C)$:
By \HTT{Proposition 1.2.7.2, 1.2.7.3}, the simplicial set $\Fun(X,\C)$
is a quasicategory whenever $\C$ is, even for an arbitrary simplicial set
$X$.


An $\i$-groupoid is an $\i$-category with the property that its
homotopy category is a groupoid (cf. \HTT{\S 1.2.5}); equivalently
\cite{MR1935979}, an $\i$-category is a Kan complex, and so the
homotopy category is the fundamental groupoid of the underlying
simplicial set.  Just as ordinary categories are categories enriched
over $0$-groupoids, or sets, an $\i$-category is essentially a category
enriched over $\i$-groupoids, or spaces.  Indeed, if $\C$ is an
$\i$-category, and $a$ and $b$ are a pair of objects (vertices) of
$\C$, then the $\i$-groupoid $\C(a,b)$ of maps from $a$ to $b$ in $\C$
may be modeled as the fiber
$$
\C(a,b)\longrightarrow\Fun(\Delta^1,\C)\longrightarrow\Fun(\partial\Delta^1,\C)
$$
over the object $(a,b)$ of $\C\times\C\cong\Fun(\partial\Delta^1,\C)$
(see \HTT{\S 1.2.2} for more details on this and various other
homotopy equivalent models for $\C(a,b)$).  Since
$\i$-categories are to be thought of as enriched over $\i$-groupoids
in a higher categorical sense (which we will not make precise), we
should regard $\C(a,b)$ as only being defined up to equivalence, and
so we shouldn't expect a composition map
$\C(a,b)\times\C(b,c)\to\C(a,c)$, but rather a contractible space of
possible composites.  Nonetheless, we still write $gf:a\to c$ for a
composite of $f:a\to b$ and $g:b\to c$.

The description of $\Fun(\C,\D)$ described above gives rise to {\em
  categories} of $\i$-categories and $\i$-groupoids, but the real
power of this approach comes from having {\em $\i$-categories}
$\Cat_\i$ and $\Gpd_\i$ of $\i$-categories and $\i$-groupoids,
respectively.  We construct these $\i$-categories by a general
technique for converting a simplicial category to an $\i$-category:
there is a simplicial nerve functor $\N$ from simplicial
categories to $\i$-categories which is the right Quillen functor of a Quillen equivalence \HTT{\S 1.1.5.5, 1.1.5.12, 1.1.5.13}
$$
\mathfrak{C}:\Set_\Delta\rightleftarrows\Cat_\Delta:\N.
$$
Note that this process also gives rise to a standard passage from a simplicial
model category to an $\i$-category which retains the homotopical information encoded by the simplicial model structure.  Specifically, given a
simplicial model category $\mathcal{M}$, one restricts to the simplicial category on the
cofibrant-fibrant objects, $\mathcal{M}^\circ$.  Then applying the simplicial
nerve yields an $\i$-category $\N\mathcal{M}^\circ$.

Given $\i$-categories $\C$ and $\D$, we have the $\i$-category of
functors $\Fun(\D,\C)$ from $\D$ to $\C$; from this we obtain an
$\i$-groupoid $\map(\D,\C)$ by forgetting the non-invertible natural
transformations.  Then $\Cat_\i$ is the simplicial nerve of the
simplicial category of $\i$-categories, in which the mapping spaces
are made fibrant by restricting to maximal Kan subcomplexes, and
$\Gpd_\i$ is the full $\i$-subcategory of $\Cat_\i$ on the
$\i$-groupoids, or equivalently the simplicial nerve of the simplicial
category of $\i$-groupoids (here the mapping spaces are automatically
fibrant since natural transformations of $\i$-groupoids are always
invertible).

\subsection{Bundles of $\i$-groupoids}

\label{sec:bundles-i-groupoids}

Because $\i$-categories are enriched over $\i$-groupoids,
$\i$-groupoids play a role in $\i$-category theory analogous to that
of sets in ordinary category theory.  In particular, the $\i$-category
of $\i$-groupoids inherits one of the most important exactness
properties enjoyed by the category of sets: colimits commute with
base-change.  That is, if $X_\alpha$ is a diagram (possibly even
indexed by an $\i$-category $\C$) of $\i$-groupoids, with colimit
$$
\colim_\C X_\alpha\heq X,
$$
then, for any map $f:X'\to X$,
$$
X'\heq f^*X\heq\colim_\C f^*X_\alpha,
$$
where $X_\alpha$ is regarded as sitting over $X$ via the inclusion
into the colimit
\[X_\alpha\to\colim_\C X_\alpha\heq X.\]

To see this, it will be convenient to show a slightly stronger statement.

\begin{Proposition}
The base-change functor
$$
f^*:\Gpd_{\i/X}\longrightarrow\Gpd_{\i/X'}
$$
admits a right adjoint (in the $\i$-categorical sense, as in \HTT{\S 5.2.2}).
In particular, $f^*$ commutes with colimits.
\end{Proposition}

\begin{proof}
This can be verified directly on the level of simplicial sets.
To do so, first note that, replacing $X'$ by an equivalent $\i$-groupoid if necessary, we may assume without loss of generality that $f:X'\to X$ is a fibration of (fibrant) simplicial sets.
Since $\Set_\Delta$ is an ordinary topos,
$$
f^*:\Set_{\Delta/X}\longrightarrow\Set_{\Delta/X'}
$$
admits a right adjoint
$$
f_*:\Set_{\Delta/X'}\longrightarrow\Set_{\Delta/X}
$$
which together comprise a Quillen pair, as $f^*$ preserves cofibrations and weak equivalences (here we are using right properness).
Moreover, this Quillen pair is compatible with simplicial model structures, so it extends to a simplicial adjunction of simplicial categories.
It therefore follows from \HTT{Proposition 5.2.2.12} that restricting to (co)fibrant objects and forming simplicial nerves yields the desired adjunction on the level of $\i$-categories.
\end{proof}

Given a Kan complex $X$, let $\mathrm{Fib}(X)$ denote the (large)
simplicial category of fibrations $Y\to X$ with target $X$.
In other words, $\mathrm{Fib}(X)$ is the full simplicial subcategory of simplicial sets over $X$ consisting of the fibrations.
This is a contravariant simplicial functor in Kan complexes $X$, so applying the simplicial nerve functor
$$
\N:\Cat_\Delta\longrightarrow\Set_\Delta
$$
yields a presheaf $\mathscr{F}$ of (large)
$\i$-categories on $\Gpd_\i$.

\begin{Proposition}
The fibrant simplicial category $\mathrm{Fib}(X)$ is a model for the slice $\i$-category
$\Gpd_{\i/X}$; that is, for each Kan complex $X$, there is an equivalence of $\i$-categories
$$
\mathscr{F}(X)=\N\mathrm{Fib}(X)\heq\Gpd_{\i/X}.
$$
\end{Proposition}

\begin{proof}
The projection
$
\mathscr{F}(X)\rightarrow\Gpd_\i
$
induces a map
$
\mathscr{F}(X)_{/1_X}\rightarrow\Gpd_{\i/X},
$
where $1_X$ is the identity fibration $X\to X$.
The desired map is the composite
$$
\mathscr{F}(X)\longrightarrow\mathscr{F}(X)_{/1_X}\longrightarrow\Gpd_{\i/X},
$$
where the first map is a homotopy inverse of the projection $\mathscr{F}(X)_{/1_X}\to\mathscr{F}(X)$, which by \HTT{Proposition 1.2.12.4} is a trivial fibration as $1_X$ is a final object of $\mathscr{F}(X)$.
It is essentially surjective because any $Z\to X$ admits a factorization $Z\to Y\to X$ with $Z\to Y$ a homotopy equivalence and $Y\to X$ a fibration, so it only remains to check that it is fully faithful.
Indeed, if $Y\to X$ and $Z\to X$ are fibrations, then $\map_{/X}(Z/X,Y/X)$ is the homotopy fiber of the map
$$
\map(Z,Y)\longrightarrow\map(Z,X)
$$
given by composing with $Y\to X$ over the vertex defined by $Z\to X$; since $Y\to X$ is a fibration, this is equivalent to the ordinary fiber, which is the mapping space in $\mathrm{Fib}(X)$ and hence also in $\mathscr{F}(X)$.
\end{proof}

\begin{Proposition}\label{X=colim1}
Let $X$ be an $\i$-groupoid.
Then $X$ is a colimit (in the $\i$-category of $\i$-groupoids) of the constant functor $1:X\to\Gpd_\i$ with value the terminal $\i$-groupoid $1$.
\end{Proposition}

\begin{proof}
Let $\C$ denote the full $\i$-subcategory of $\Gpd_\i$ spanned by the $\i$-groupoids $X$ with the property that $\colim_X 1\simeq X$; note that $\C$ is nonempty, as it contains the terminal $\i$-groupoid $1$.
Now clearly $\C$ is closed under (possibly infinite) coproducts and, according to \HTT{Proposition 4.4.2.2}, $\C$ is closed under pushouts.
Hence, by \HTT{Proposition 4.4.2.6}, $\C$ is closed under arbitrary colimits, and any $\i$-groupoid $X$ may be built out of colimits from the terminal $\i$-groupoid: indeed, $S^0=1\coprod 1$, so inductively $S^n$ arises as the pushout of $1\leftarrow S^{n-1}\rightarrow 1$, and an arbitrary $X$ is a colimit $X=\colim_n\mathrm{sk}_n X$, where $\mathrm{sk}_0 X$ is discrete and $\mathrm{sk}_{n+1} X$ is obtained from $\mathrm{sk}_n X$ as a pushout $\coprod 1\leftarrow\coprod S^n\rightarrow\mathrm{sk}_n X$.
Hence $\C=\Gpd_\i$.
\end{proof}

\begin{Proposition} \label{t-pr-G-space-space-over-bg}
Let $X$ be an $\i$-groupoid.
Then the colimit functor
$$
\Fun(X,\Gpd_\i)\longrightarrow\Gpd_\i
$$
factors through $\Gpd_{\i/X}$, and the induced map
$$
\Fun(X,\Gpd_\i)\longrightarrow\Gpd_{\i/X}
$$
is an equivalence of $\i$-categories.  In particular if $X=BG$ is an
$\infty$-groupoid with a single object $*$ and $G$ is the
group-like monoidal $\infty$-groupoid $G=\map_{X}(*,*)$, then there we have
an equivalence
\[
    \Fun (BG,\Gpd_{\i}) \heq \Gpd_{\i/BG}
\]
between the $\i$-category of $\i$-groupoids with an action of $G$ and
the $\i$-category of $\i$-groupoids over $BG.$
\end{Proposition}

\begin{Remark} \label{rem-4}
Since $X$ is an $\i$-groupoid, we have
$X\heq X^{op}$, and so
\[
  \Fun (X^{op},\Gpd_{\i}) \heq \Gpd_{\i/X}
\]
as well.  This is a natural model for $\Gpd_{\i/X}$ when
we think of an $\i$-groupoid over $X$ as a presheaf of
$\i$-groupoids on $X$, as we do for example in Proposition
\ref{t-pr-yoneda-taut}.
\end{Remark}

\begin{proof}
First note that, while $\colim$ is not ordinarily a functor, we may model it as the derived functor of a (functorial choice of) colimit
$$
\colim:\Fun(\mathfrak{C}[X],\Set_\Delta)\longrightarrow\Set_\Delta
$$
on the level of model categories, where the model category on the left is equipped with the projective model structure.
Also note that the first claim is a special case of the following fact:  given $\i$-categories $\C$ and $\D$ such that $\C$ has a terminal object $1$ and given a functor $f:\C\to\D$, then $f$ factors through the projection $\D_{/f(1)}\to\D$, since $f$ determines a functor $\C\heq\C_{/1}\to\D_{/f(1)}$.
By Proposition \ref{X=colim1}, $X$ is a colimit for the terminal functor $1:X\to\Gpd_\i$, giving an equivalence $X\heq f(1)$ and thus a factorization of the colimit through $\Gpd_{\i/X}$.

It remains to show that this resulting map is an equivalence of
$\i$-categories, which is to say that it is fully faithful and
essentially surjective.  For this we may assume that $X$ is connected,
since otherwise $X\heq\coprod_{\pi_0 X}X_\alpha$ with each $X_\alpha$
connected and, given the result for connected $\i$-groupoids, we
deduce that
$$
\Fun(X,\Gpd_\i)\heq\prod_\alpha\Fun(X_\alpha,\Gpd_\i)\heq\prod_\alpha\Gpd_{\i/X_\alpha}\heq\Gpd_{\i/X}.
$$
But then $X\heq BG$, where $G$ is the group-like simplicial monoid of endomorphisms of any object of $\mathfrak{C}[X]$, and the colimit (in this case, the quotient by the action of $G$) determines a Quillen equivalence of simplicial model categories between $G$-simplicial sets over $EG$, equipped with the projective model structure, and simplicial sets over $BG$ (see Appendix~\ref{sec:princ-bundl-lift} for further details on this).
Passing to $\i$-categories yields an equivalence
$$
\N(\Set_\Delta^G)^\circ\heq\Fun(BG,\Gpd_\i)\heq\Gpd_{\i/BG}
$$
and hence also the desired equivalence
$$
\Fun(X,\Gpd_\i)\heq\Gpd_{\i/X}.
$$
\end{proof}

%
%
Said differently, $\mathscr{F}$ is represented (in the $\i$-category of large $\i$-categories) by the
$\i$-category of $\i$-groupoids.
Moreover, we may work with $\mathscr{F}(X)$ in place of the equivalent but much larger $\i$-groupoid $\Gpd_{\i/X}$ by replacing the colimit with a map
\begin{equation}\label{fun=fib}
\Fun(X,\Gpd_\i)\longrightarrow\mathscr{F}(X)
\end{equation}
described as follows.

There is a universal bundle of $\i$-groupoids $\mathscr{U}\to\Gpd_\i$
over $\Gpd_\i$, characterized by the fact that the fiber over the
$\i$-groupoid $T$ (a vertex of the quasicategory $\Gpd_\i$) is $T$
itself (see \HTT{\S 3.3.2} for a treatment of universal fibrations in more general contexts).  Given a functor $f:X\to\Gpd_\i$, the restriction
$f^*\mathscr{U}\to X$ of $\mathscr{U}\to\Gpd_\i$ along $f$ is an
$\i$-groupoid over $X$ such that the fiber of $f^*\mathscr{U}\to X$
over an object $x$ of $X$ is the $\i$-groupoid $f(x)$.  We think of
$Y=f^*\mathscr{U}\to X$ as the bundle of $\i$-groupoids classified by
the map $f:X\to\Gpd_\i$.  There is also an inverse procedure which
associates to an arbitrary bundle of $\i$-groupoids $Y\to X$ over $X$
a functor $f:X\to\Gpd_\i$, whose value on the object $x$ is equivalent
to the fiber $Y_x$ of $Y$ over $x$, but this process is less explicit (cf. \HTT{\S 2.2.1}).

We may also want to consider more specific types of fibrations over
$X$; in particular, we will be interested in those fibrations $Y\to X$
such that, for each object $x$ of $X$, the fiber $Y_x$ is equivalent
to a fixed $\i$-groupoid $F$.  These are precisely those fibrations
equivalent to ones of the form $f^*\mathscr{U}$ for some
$f:X\to\Gpd_\i$ which factors through the inclusion
$B\End(F)\to\Gpd_\i$, where $B\End(F)$ denotes the full
$\i$-subcategory of $\Gpd_\i$ on the object $F$ (an $\i$-category with
one object and endomorphisms $\End(F)$).  In other words, the functor
of fibrations with fiber $F$ is represented by the $\i$-category
$B\End(F)$; generally, we're more interested in its maximal
$\i$-subgroupoid, which is represented by the $\i$-groupoid
$B\Aut(F)$.  This description shows that it is enough to consider
principal $\Aut(F)$-bundles instead, since the two functors are
equivalent.

An $\i$-category $\C$ equipped with a mapping space {\em functor}
$$
\C^{op}\times\C\longrightarrow\Gpd_\i
$$
induces a Yoneda embedding
$$
\C\longrightarrow\Fun(\C^{op},\Gpd_\i).
$$
Of course, in $\i$-category theory, mapping spaces are not
uniquely defined, much less functorial.
In particular, there is not really a
canonical choice of a Yoneda embedding; instead, as with other
$\i$-categorical constructions, it's only defined up to a contractible
space of choices.

\begin{Remark}\label{rem-yoneda-rigid}
This ``problem'' goes away if we work instead with the more rigid model of simplicial categories.
Here we have a simplicial mapping space functor
$$
\mathfrak{C}[\C^{op}\times\C]\heq
\mathfrak{C}[\C]^{op}\times\mathfrak{C}[\C]\longrightarrow\Set_{\Delta}\overset{\mathrm{fib}}{\longrightarrow}\mathrm{Kan}
$$
which we may suppose (after taking a fibrant replacement) lands in the simplicial category of Kan complexes.  Passing
back via the simplicial nerve yields a functorial assignment of mapping spaces
$$
\C^{op}\times\C\heq\N\mathfrak{C}[\C^{op}\times\C]\longrightarrow\N\mathrm{Kan}=\Gpd_\i
$$
and hence a particular choice of Yoneda embedding
$\C\to\Fun(\C^{op},\Gpd_\i)=\Pre(\C)$.
In general, we define the Yoneda embedding $\C\to\Pre(\C)$ to be any functor equivalent to the one constructed in this way.
\end{Remark}

\begin{Proposition}\label{colimX_y=X}
Let $X\to Y$ be a map of $\i$-groupoids and let $Y\to\Gpd_{\i/Y}$ be the Yoneda embedding.
Then, as an $\i$-groupoid over $Y$, $X$ is a colimit of the composite $X\to Y\to\Gpd_{\i/Y}$.
\end{Proposition}

\begin{proof}
Since $Y\simeq\colim_Y 1$ and colimits commute with base-change, we obtain a decomposition $X\simeq\colim_Y X_y$, where $X_y$ denotes the fiber (calculated in the $\i$-category of $\i$-groupoids) of $X\to Y$ over the object $y$; that is, $X_y:Y\to\Gpd_\i$ is the functor obtained from $Y\to X$ via the equivalence $\Gpd_{\i/Y}\simeq\Fun(Y,\Gpd_\i)$.
As in the proof of Proposition \ref{X=colim1}, we easily reduce to the case in which $Y$ is the terminal $\i$-groupoid $1$, in which case $X$ is evidently a colimit of the functor $1\to\Gpd_\i$ with value $X$.
\end{proof}

We'll be most interested in the case where $Y=BG$ is the one-object $\i$-groupoid associated to a group-like monoidal $\i$-groupoid $G$; i.e. $G\simeq\map_{BG}(*,*)$.
Then
$$\Gpd_{\i/BG}\simeq\Fun(BG,\Gpd_\i)$$
is the $\i$-category of $\i$-groupoids equipped with an action of $G$, and the Yoneda embedding
$$BG\to\Gpd_{\i/BG}\simeq\Fun(BG^{op},\Gpd_\i)$$
sends $X\to BG$ (viewed as a generalized element of the $\i$-groupoid $BG$) to $X\to BG$ (viewed as an $\i$-groupoid over $BG$), or, equivalently, to the fiber of $X\to BG$ over $*\to BG$ equipped with its natural $G\simeq\map_{/BG}(*,*)$-action (viewed as a functor $BG\to\Gpd_\i$).

\begin{Corollary} \label{t-pr-P-colim-X-to-BG}
Let $X\to BG$ be an $\i$-groupoid over $BG$, and write $P(X)=X\times_{BG} *$ for the associated $G$-$\i$-groupoid.
Then $P(X)$ is a colimit of the composite $X\to BG$ followed by the Yoneda embedding $BG\to\Fun(BG^{op},\Gpd_\i)$.\hfill$\Box$
\end{Corollary}

\subsection{Bundles of stable $\i$-groupoids}
\label{sec:infinity-category-r}

A pointed $\i$-category $\C$ with finite limits is said to be {\em stable} if the endofunctor $$\Omega:\C\longrightarrow\C,$$ defined by sending $X$ to the limit of the diagram $\ast\to X\leftarrow\ast$, is an equivalence \cite[Corollary 10.12]{DAGI}.  Strictly
speaking, $\Omega$, as defined, is not actually a functor on the level
of $\i$-categories (though it can be strictified to form one), but
it's enough to check that $\Omega$ induces an endo-equivalence of the
homotopy category of $\C$.  A morphism of stable $\i$-categories is an
exact functor, meaning a functor which preserves finite limits and
colimits \cite[\S 5]{DAGI}.

More generally, given any $\i$-category $\C$ with finite limits,
the {\em stabilization} of $\C$ is the limit (in the
$\i$-category of $\i$-categories) of the tower
$$
\xymatrix{\cdots \ar[r]^{\Omega} & \C_*\ar[r]^{\Omega} & \C_*},
$$
where $\C_*$ denotes the pointed $\i$-category associated to $\C$ (the
full $\i$-subcategory of $\Fun(\Delta^1,\C)$ on those arrows whose
source is a final object $*$ of $\C$).  See \cite[\S 10]{DAGI} for more on stabilization.  Provided $\C$ is presentable,
$\Stab(\C)$ comes equipped with a stabilization functor $\Sigma^\i_+:\C\to\Stab(\C)$ functor from $\C$ \cite[Proposition 17.4]{DAGI},
formally analogous to the suspension spectrum functor, and left
adjoint to the zero-space functor $\Omega^\i_-:\Stab(\C)\to\C$ (the
subscript indicates that we forget the basepoint).  Heuristically,
$\text{Stab}$ is left adjoint to the inclusion into the $\i$-category
of presentable $\i$-categories of the full $\i$-subcategory of
presentable {\em stable} $\i$-categories (note that morphisms of
presentable $\i$-categories are colimit-preserving functors).  In
other words, a morphism of presentable $\i$-categories $\C\to\D$ such
that $\D$ is stable factors (uniquely up to a contractible space of
choices) through the stabilization $\Sigma^\i_+:\C\to\text{Stab}(\C)$
of $\C$ (cf. \cite[Corollary 17.5]{DAGI}).

According to \cite[Corollary 6.24]{DAGIII},
$\Stab (\Gpd_{\i})$ is a symmetric monoidal
$\i$-category under the smash product, in such a way that
$$
\Sigma^\i_+:\Gpd_\i\longrightarrow\Stab(\Gpd_\i)
$$
is a strong symmetric monoidal functor (with respect to the cartesian
monoidal structure on $\Gpd_\i$).  In particular, $S=\Sigma^\i_+ 1$ is
a unit for a symmetric monoidal structure $\land$ on $\Stab(\Gpd_\i)$.
Here $1$ denotes the terminal $\i$-groupoid, the unit for the
cartesian monoidal structure on $\Gpd_\i$.

Since $S$ is the unit of the symmetric monoidal structure on
$\Stab(\Gpd_\i)$, an $S$-algebra is just a monoid for the
$\land$-product.  Given an $S$-algebra $R$, we write $\Rmod$ for the
$\i$-category of right $R$-modules.  If $R$ is commutative then
$\Rmod$ inherits a symmetric monoidal structure $\land_R$ for which
$R$ is the unit.


Now, for any $\i$-groupoid $X$, the equivalence
$$
\Gpd_{\i/X}\heq\Fun(X,\Gpd_\i)
$$
induces an equivalence
of stabilizations
$$
\Stab(\Gpd_{\i/X})\heq\Stab(\text{Fun}(X,\Gpd_\i)).
$$
Since limits in functor categories are computed pointwise, one easily
checks that
$$
\Stab(\Gpd_{\i/X})\heq\Stab(\Fun(X,\Gpd_\i))\heq\Fun(X,\Stab(\Gpd_\i));
$$
that is, the $\i$-category of bundles of stable $\i$-groupoids over
$X$ --- the stabilization of the $\i$-category of bundles of
$\i$-groupoids over $X$ --- is equivalent to the $\i$-category of
functors from $X$ to stable $\i$-groupoids.

The resulting $\infty$-category $\Stab(\Gpd_{\i/X})$ is closed
symmetric monoidal under the {\em fiberwise} smash product $\land_X$ \cite{MR2271789}, as well as enriched, tensored and cotensored over $\Gpd_{\i/X}$.  We
write $S_X$ for the sphere over $X$, the unit of the symmetric
monoidal product $\land_X$.  Note that $\Stab(\Gpd_{\i/X})$ is
naturally equivalent to the $\infty$-category $\Mod{S_X}$ of
$S_X$-modules.

All of this works equally well for any $S$-algebra $R$.  Writing
$p:X\to 1$ for the projection to the terminal Kan complex, then
$$
R_X\eqdef p^*R
$$
is an associative $S_X$-algebra, and we may form the $\i$-category
$\Mod{R_X}$ of right $R_X$-modules.
We also have an
equivalence
\begin{equation}\label{Rfun=Rfib}
\Fun(X,\Mod{R})\overset{\!\!\!\heq}{\longrightarrow}\Mod{R_X},
\end{equation}
the $R$-module analogue of the equivalence
\eqref{fun=fib}, which we may think of as arising from pulling back a
universal bundle of $R$-modules $\mathscr{E}
$ over
$\Mod{R}$.

\begin{Remark}\label{rem:omit}
Since they are not necessary to obtain the main results of this paper,
we have chosen to omit detailed discussion of certain foundational
results necessary for a full theory of parametrized spectra in the
$\infty$-categorical context.  For one thing, for consistency one
should compare the $\infty$-category $\Fun(X,\spectra)$ to the
May-Sigurdsson model category of spectra parametrized by $X$.
Although this comparison is conceptually straightforward, the
technical details are not insubstantial and so we leave the comparison
for another paper.  Note however that we do implicitly perform a part
of such a comparison in our discussion of the equivalence of various
definitions of the Thom spectrum functor.

Furthermore, we also have chosen to defer the explicit construction of
the universal bundle of $R$-modules over $\Mod{R}$, as that requires
the significantly more complicated theory of left fibrations of
(stable) $\i$-groupoids over $\i$-categories; see \HTT{\S 2} for
details of the unstable theory.  Fortunately, this construction is
again not actually necessary for our work in this paper, as any bundle
of $R$-modules over an $\i$-category $\C$ is classified (up to
equivalence) by a functor $\C\to\Mod{R}$.  In other words, the reader
is free to take $\Fun(\C,\Mod{R})$ as the {\em definition} of
$\Mod{R_\C}$, in which case the equivalence
$\Fun(\C,\Mod{R})\heq\Mod{R_\C}$ is actually equality.  From this
point of view, the universal bundle of $R$-modules $\mathscr{E}
$ over $\Mod{R}$ corresponds to the identity map $\Mod{R}\to\Mod{R}$,
and pulling back the universal bundle along a functor $f:\C\to\Mod{R}$
corresponds to precomposing the identity with $f$.
\end{Remark}

A map of spaces $f:X\to Y$ gives rise to a restriction functor
$$
f^*:\Mod{R_Y}\to\Mod{R_X}
$$
which admits a right adjoint $f_*$ as well as a left adjoint $f_!$.
This means that, given an $R_X$-module $L$ and an $R_Y$-module $M$,
there are natural equivalences of $\i$-groupoids
$$
\Mod{R_Y}(f_!L,M)\heq\Mod{R_X}(L,f^*M)\qquad\text{and}\qquad\Mod{R_X}(f^*M,L)\heq\Mod{R_Y}(M,f_*L).
$$
If $M\heq\psi^*\mathscr{E}$ is the bundle of $R$-modules classified
by the functor $\psi:Y\to\Mod{R}$, then $f^*M\heq
f^*\psi^*\mathscr{E}$; similarly, if $L\heq\varphi^*\mathscr{E}$ for
some $\varphi:X\to\Mod{R}$, then $f_!M$ and $f_*M$ are classified by
the left and right Kan extensions, respectively, of $\varphi$ along
$f$.

The {\em projection formula} asserts that
the adjoint
$$
f_!(L\land_X f^*M)\longrightarrow f_!L\land_Y M
$$
of the composite
$$
L\land_X f^*M\longrightarrow f^*f_!L\land_X f^*M\heq f^*(f_!L\land_Y
M)
$$
is an equivalence (here we are using the fact that $f^*$ is strong
monoidal).  To see this, we merely need to examine the fiber over an
object $y$ of $Y$:
\begin{eqnarray*}
f_!(L\land_X f^*M)_y\heq\colim_{x\in X_y}(L_x\land
f^*M_x)\heq\colim_{x\in X_y}(L_x\land M_y) \\ \heq(\colim_{x\in
X_y}L_x)\land M_y\heq f_!L_y\land M_y\heq (f_!L\land_Y M)_y.
\end{eqnarray*}

This has some notable consequences.  Given a second $R_Y$-module $N$,
the equivalences
\begin{eqnarray*}
\map(L,f^*\Map_X(M,N))\heq
\map(f_!L\land_X M,N)\heq
\map(L\land_Y f^*M,f^*N)\heq \map(L,\Map_Y(f^*M,f^*N))
\end{eqnarray*}
and
\begin{eqnarray*}
\map(M,\Map_Y(f_!L,N))\heq \map(f_!L\land_Y M,N)\heq
\map(L\land_X f^*M,f^*N)\heq
(M,f_*\Map_X(L,f^*N))
\end{eqnarray*}
imply that the natural maps
\begin{eqnarray*}
f^*\Map_Y(M,N)\longrightarrow\Map_X(f^*M,f^*N) \qquad\text{and}\qquad
\Map_Y(f_!L,N)\longrightarrow f_*\Map_X(L,f^*N)
\end{eqnarray*}
are equivalences of $R_X$-modules and $R_Y$-modules, respectively.  It
also follows that the tensor of the space $f:X\to Y$ over $Y$ with the
$R_Y$-module $M$ is given by $f_!f^*M$: indeed, $f\heq f_!\id_X$, so
\begin{equation}\label{pull-push=tensor}
\map(f_!f^*M,N)\heq\map(f^*M,f^*N)\heq\map(\id_X,f^*\map_Y(M,N))\heq\map(f,\map_Y(M,N)).
\end{equation}
In other words,
\begin{equation}\label{f_!f^*}
\Sigma^\infty_Y f_+\land_{S_Y} M\heq f_!f^*M,
\end{equation}
and in particular $\Sigma^\infty_{Y} f_+\land_{S_Y}R_Y\heq f_!f^*R_Y$ is the free $R_Y$-module on $f$, where $f_+:X+Y\to Y$ is $X\to Y$ plus a disjoint basepoint $Y\to Y$.

\subsection{Bundles of $R$-lines} \label{sec:glr-symm-mono}


\begin{Definition} \label{inf-def-1}
An $R$-line is an $R$-module $M$ which admits an $R$-module
equivalence $M\heq R$.
\end{Definition}

Let $\Rwe$ denote the full $\i$-subgroupoid of $\Rmod$ spanned by the
$R$-lines.  This is not the same as the full $\i$-subcategory of
$\Rmod$ on the $R$-lines, as a map of $R$-lines is by definition an
equivalence.  We regard $\Rwe$ as a pointed $\i$-groupoid via the
distinguished object $R$.

\begin{Remark}
$\Rwe$ is a subcategory of the $\infty$-category $\Rmod$ in the sense
of \HTT{1.2.11}.  Indeed, our definition really specifies $\Ho\Rwe$ as
a subgroupoid of the homotopy category $\Ho\Rmod$ of $\Rmod$ (the full
subgroupoid consisting of those $R$-modules which are {\em isomorphic}
to $R$ in $\Ho\Rmod$).  Then $\Rwe$ is obtained as the pullback of
simplicial sets
\[
\begin{CD}
\Rwe @>>> \Rmod \\
@VVV @VVV \\
\N\Ho\Rwe @>>> \N\Ho \Rmod.
\end{CD}
\]
Since limits of weak Kan complexes are again weak Kan complexes,
$\Rwe$ is an $\i$-category; moreover, its homotopy category is the
groupoid $\Ho\Rwe$, basically by construction, so in fact $\Rwe$ is an
$\i$-groupoid.
\end{Remark}

Fix an $\i$-groupoid $X$ and write $p:X\to 1$ for the projection to
the terminal $\i$-groupoid.  Recall that $R_X=p^*R$ is the unit of the
symmetric monoidal $\i$-category $\Mod{R_X}$ of bundles of $R$-modules
over $X$.

Given an $\i$-category $\C$, we write $\mathrm{Iso}(\C)$ for the
maximal $\i$-subgroupoid of $\C$, obtained by throwing away all
noninvertible arrows of $\C$.  Thus, if $a$ and $b$ are objects of
$\C$, $\mathrm{Iso}(\C)(a,b)$ is the subcomplex of $(\C)(a,b)$
consisting of the {\em equivalences}.

\begin{Definition}
The $\i$-groupoid $\Line{R_X}$ is the full $\i$-subgroupoid of
$\mathrm{Iso}(\Mod{R_X})$ on those $R_X$-modules $M$ such that, for
all objects $x$ of $X$, the fiber $M_x$ of $M$ over $x$ is equivalent
to $R$.
\end{Definition}

\begin{Proposition}\label{RXline=isoRXmod}
There is an equivalence
$$
\Line{R_X}(L,M)\heq\mathrm{Iso}(\Mod{R_X})(L,M),
$$
natural in $R_X$-lines $L$ and $M$.
\end{Proposition}

\begin{proof}
$\Line{R_X}$ a full $\i$-subgroupoid of $\mathrm{Iso}(\Mod{R_X})$.
\end{proof}

The restriction of the universal $R$-module bundle $\mathscr{E}$ over
$\Mod{R}$ along the inclusion $j:\Line{R}\to\Mod{R}$ is the universal
$R$-line bundle $\mathscr{L}=j^*\mathscr{E}$ over $\Line{R}$.  This is
the analogue of the tautological line bundle in geometry: the fiber
over the line $l\in\mathbb{P}^\infty$ is the line itself.

\begin{Proposition}\label{RXline=map(X,Rline)}
There is a commutative diagram of $\i$-categories
$$
\xymatrix{
\Fun(X,\Line{R})\ar[d]\ar[r] & \Line{R_X}\ar[d]\\
\Fun(X,\Mod{R})\ar[r] & \Mod{R_X}}
$$
in which the vertical maps are inclusions of full $\i$-subgroupoids
and the horizontal maps are equivalences.
\end{Proposition}

\begin{proof}
The fibers of an $R_X$-module $L$ are equivalent to $R$ if and only if
$L\heq\varphi^*\mathscr{E}$ for some $\varphi:X\to\Mod{R}$ such that
$\varphi$ factors through the inclusion $\Line{R}\to\Mod{R}$.  Hence
$\Line{R_X}$ is the full $\i$-subgroupoid of $\Mod{R_X}$ on those
objects in the essential image of $\Fun(X,\Line{R})$ via the map
$\Fun(X,\Mod{R})\to\Mod{R_X}$ of \ref{Rfun=Rfib}; this map is an
equivalence, so $\Fun(X,\Line{R})\heq\Line{R_X}$.
\end{proof}

\begin{Corollary}
The fiber over the $\i$-groupoid $X$ of the projection
$\Gpd_{\i/\Line{R}}\to\Gpd_\i$ is equivalent to the $\i$-groupoid
$\Line{R_X}$.
\end{Corollary}

\begin{proof}
$\Line{R_X}\heq\Fun(X,\Line{R})\heq\Gpd_\i(X,\Line{R})$, and, in
general, the $\i$-groupoid $\C(a,b)$ of maps from $a$ to $b$ in the
$\i$-category $\C$ may be calculated as the fiber over $a$ of the
projection $\C_{/b}\to\C$.
\end{proof}

\begin{Definition}
A {\em trivialization} of an $R_X$-module $L$ is an $R_X$-module
equivalence $L\to R_X$.  The $\i$-category $\Triv{R_X}$ of trivialized
$R$-lines is the slice category
\[
\Triv{R_X}\eqdef\Line{R_X}_{/R_X}.
\]
\end{Definition}

The objects of $\Triv{R_X}$ are {\em trivialized} $R_X$-lines, which
is to say $R_X$-lines $L$ with a trivialization $L\to R_X$; more
generally, an $n$-simplex $\Delta^n\to\Triv{R_X}$ of $\Triv{R_X}$ is a
map $\Delta^n\star\Delta^0\to\Line{R_X}$ of $\Line{R_X}$ which sends
$\Delta^0$ to $R_X$.  There is a canonical projection
$$
\iota_X:\Triv{R_X}\longrightarrow\Line{R_X}
$$
which sends the $n$-simplex $\Delta^n\star\Delta^0\to\Line{R_X}$ to
the $n$-simplex $\Delta^n\to\Delta^n\star\Delta^0\to\Line{R_X}$;
according to (the dual of) \HTT{Corollary 2.1.2.4}, this is a right fibration, and hence a
fibration as $\Line{R_X}$ is an $\i$-groupoid \HTT{Lemma 2.1.3.2}.

When $X$ is the terminal Kan complex, we write $\Triv{R}$ in place of
$\Triv{R_X}$ and $\iota:\Triv{R}\longrightarrow\Line{R}$ in place of
$\iota_X$.  Given a map $f:X\to\Line{R}$, we will refer to a
factorization
$$
\xymatrix{ & {\Rtriv}
 \ar[d]_\iota \\
{X} \ar@{-->}[ur] \ar[r]_-{f} & {\Rwe.}  }
$$
of $f$ through $\iota$ as a {\em trivialization} of $f$;
i.e. $\mathrm{Triv}(f)$ is the fiber over $f$ of
$\iota_X:\Fun(X,\Triv{R})\to\Fun(X,\Line{R})$.

\begin{Proposition}\label{inf-t-pr-egl-bgl-fib}
Let $X$ be an $\i$-groupoid.  Then $\Triv{R_X}$ is a contractible
$\i$-groupoid, and the fiber of the fibration
$$
\iota_X:\Triv{R_X}\longrightarrow\Line{R_X}
$$
over an $R_X$-line $L$ is the (possibly empty) $\i$-groupoid
$\mathrm{Triv}(L)$ of trivializations of $L$.
\end{Proposition}

\begin{proof}
Once again, use the description of $\C(a,b)$ as the fiber over $a$ of
the projection $\C_{/b}\to\C$, together with the fact that if $\C$ is
an $\i$-groupoid then $\C_{/b}$, an $\i$-groupoid with a final object,
is contractible.
\end{proof}

\begin{Corollary}\label{triv(f)=triv(f^*L)}
For any $f:X\to\Line{R}$ there is a commutative diagram of
$\i$-groupoids
$$
\xymatrix{
\mathrm{Triv}(f)\ar[r]\ar[d]          & \Fun(X,\Triv{R})\ar[r]\ar[d] & \Fun(X,\Line{R})\ar[d]\\
\mathrm{Triv}(f^*\mathscr{L})\ar[r] & \Triv{R_X}\ar[r] & \Line{R_X}}
$$
in which each of the squares is cartesian and each of the vertical
maps is an equivalence.
\end{Corollary}

\begin{proof}
Indeed,
$$
\Fun(X,\Line{R}_{/R})\heq\Fun(X,\Line{R})_{/p^*R}\heq\Line{R_X}_{/R_X},
$$
so the right-hand square is cartesian.  The left-hand square is
obtained by taking the fiber over $f$.
\end{proof}

%

\subsection{Thom $R$-modules and orientations}
\label{sec:thom-r-modules}



\begin{Definition}
The Thom $R$-module is the functor
$$
M:\Gpd_{\i/\Line{R}}\longrightarrow\Rmod
$$
which sends $f:X\to\Line{R}$ to
$$
Mf\eqdef p_!f^*\mathscr{L},
$$
the push-forward of the restriction along $f$ of the universal
$R$-line bundle $\mathscr{L}$.
\end{Definition}

\begin{Remark} \label{rem-2}
In terms of the equivalences
\[
    \Mod{R_{X}} \heq \Fun (X,\Mod{R})
\]
and
\[
     \Line{R_{X}}\heq \Fun(X,\Line{R}),
\]
the $R_{X}$-line $f^{*}\mathscr{L}$ corresponds to the map
\[
     X \xra{f} \Rwe \xra{} \Rmod,
\]
and then $Mf$ is the colimit
\[
     Mf = \colim (X \xra{f} \Rwe \xra{} \Rmod).
\]
As we shall see, this is a very useful formula for the Thom spectrum.
It is the only formula available in the development in
\S\ref{sec:i-cat-ii}.
\end{Remark}


We also obtain another useful characterization in the following
proposition.

\begin{Proposition}\label{quotient}
Let $G$ be a group-like monoidal $\i$-groupoid with delooping $BG$ ---
that is, an $\i$-groupoid with one object $*$ and an equivalence
$G\heq\Omega BG\heq\Aut_{BG}(*)$ --- and suppose that $X\heq
BG$.  Then
$$
Mf\heq R/G,
$$
where $G$ acts on $R$ by $R$-module maps via $\Omega f:G\heq\Omega
BG\to\Omega\Line{R}\heq\Aut_R(R)$.
\end{Proposition}

\begin{proof}
Evidently, $R/G$ is the colimit of the composite functor
$$
BG\to B\Aut_R(R)\heq\Line{R}\to\Mod{R}
$$
classifying the $R_{BG}$-module $f^*\mathscr{L}$.  But a colimit is a
special case of a left Kan extension; namely, the left Kan extension
along the projection $p:BG\to\ast$ to the point.
\end{proof}



With these in place, one can analyze the space of orientations in a
straightforward manner, as follows.  First of all observe that, by
definition, we have a weak equivalence
\[
   \Mod{R} (Mf,R) \heq \Mod{R_{X}} (f^{*}\mathscr{L},p^{*}R).
\]

%

\begin{Definition} \label{inf-def-2}
The $\i$-groupoid of \emph{orientations} of $Mf$ is the
``pull-back''
\begin{equation}\label{inf-eq:2}
\begin{CD}
\Rorient (Mf,R) @>>> \Rmod (Mf,R) \\
@V\heq VV @VV \heq V \\
\Line{R_X}(f^*\mathscr{L},p^*R) @>>> \Mod{R_X}(f^*\mathscr{L},p^*R).
\end{CD}
\end{equation}
\end{Definition}

The $\i$-groupoid $\Rorient (Mf,R)$ enjoys an obstruction theory
analogous to that of the space of orientations described in Definition
\ref{def-orientation}.  The following theorem is the analogue in this
context of Theorem \ref{t-th-orientation-ainfty}.

\begin{Theorem} \label{t-th-or-thy-infty-lifting}
Let $f:X\to\Rwe$ be a map, with associated Thom $R$-module $Mf$.
Then the $\i$-groupoid of orientations $Mf\to R$ is equivalent to
the $\i$-groupoid of lifts in the diagram
\begin{equation} \label{eq:26}
\xymatrix{ & {\Rtriv}
 \ar[d]_\iota \\
{X} \ar@{-->}[ur] \ar[r]_-{f} & {\Rwe.}  }
\end{equation}
\end{Theorem}

\begin{proof}
According to Corollary \ref{triv(f)=triv(f^*L)}, the $\i$-groupoid
$\mathrm{Triv}(f)$ of factorizations of $f$ through $\iota$ is
equivalent to the $\i$-groupoid $\mathrm{Triv}(f^*\mathscr{L})$ of
$R_X$-module equivalences from $f^*\mathscr{L}$ to $p^*R$.
\end{proof}

\begin{Corollary} \label{t-thom-iso-infty-cat-I}
An orientation of $Mf$ determines an equivalence of $R$-modules
\[
Mf \heq \splus X \Smash R.
\]
\end{Corollary}

\begin{proof}
The desired map is the composite
$$
\Rorient(Mf,R)\to\mathrm{Iso}(\Mod{R_X})(f^*\mathscr{L},p^*R)\to\mathrm{Iso}(\Mod{R})(p_!f^*\mathscr{L},p_!p^*R)\to\mathrm{Iso}(\Mod{R})
(p_!f^*\mathscr{L},\Sigma^\i_+ X\land R),
$$
where the first map is the equivalence of \ref{inf-eq:2}, the second
map applies $p_!$, and the last map composes with the canonical
equivalence $p_!p^*R\to\Sigma^\i_+ X\land R$ of
\eqref{pull-push=tensor}.
\end{proof}

\subsection{$\i$-category of $R$-modules from model categories of
spectra}
\label{sec:i-category-r}

Our $\i$-categorical treatment of Thom spectra has been based on the
the symmetric monoidal $\i$-category of spectra developed in
\cite{DAGI,DAGII,DAGIII}.   We associate to an algebra $R$ in $\Stab
(\Gpd_{\i})$ the $\i$-category $\Rwe$, and then a Thom spectrum
functor
\[
   (\Gpd_{\i})_{/\Rwe} \xra{} \Rmod.
\]
In \S\ref{sec:ainfty-thom-spectrum} 
we
associate to an EKMM $S$-algebra $R$ the space $B\GL{R},$ and our Thom
spectrum construction may
be viewed as a functor
\[
   \spaces_{/ B\GL{R}}  \to \EKMM_{R}.
\]
In \S\ref{sec:comp-thom-spectra}, we will compare these two approaches
to Thom spectra to each other and to Thom spectra in the literature.
To prepare for this comparison, we show here how to set
up our $\i$-categorical analysis, starting from a (symmetric) monoidal
simplicial model category of spectra, such as the
$S$-modules of \cite{EKMM} or the symmetric spectra of
\cite{MR1695653}.

We first associate to an algebra $R$ in a simplicial
monoidal $\i$-category of spectra an $\i$-category $\Rwe$ of
$R$-lines, as above, and we show that $\Rwe$ has the homotopy type of
$B\GL{R}.$  It is straightforward then to construct a Thom
spectrum functor
\[
   (\Gpd_{\infty})_{/\Rwe} \to \Rmod,
\]
as above.

We also show that there is an algebra $R'$ in $\Stab(\Gpd_{\i})$ such
that
\[
        \Rmod \heq \Mod{R'},
\]
and so
\[
    (\Gpd_{\infty})_{/\Rwe} \heq (\Gpd_{\infty})_{/\Line{R'}},
\]
in such a way that the two $\i$-categorical Thom
spectrum functors
\[
    (\Gpd_{\infty})_{/\Rwe} \to \Rmod
\]
and
\[
    (\Gpd_{\infty})_{/\Line{R'}} \to \Mod{R'}
\]
are obviously equivalent.

Let $\monspectra$ be a symmetric monoidal simplicial model category of
spectra such as the $S$-modules of \cite{EKMM}, or the symmetric
spectra of \cite{MR1695653}, and let $R$
be a cofibrant and fibrant associative algebra in
$\monspectra.$   We require $R$ to be
cofibrant and fibrant so that the derived spaces of self-equivalences of
$R$ has the homotopy type of $\GL{R}$; see Proposition
\ref{inf-t-pr-Rwe-and-GL}.

Let $\RmodT$ be the simplicial model category of
$R$-modules.   From $\RmodT$ we form the $\infty$-category
$\Rmod$ in the usual way, as we now recall.  If $C$ is a simplicial
model category, let $\MtoI{C}$ denote the full subcategory of $C$
consisting of cofibrant-fibrant objects.  Now take the simplicial
nerve \HTT{1.1.5.5} to obtain the $\infty$-category
\[
      \Rmod \eqdef \N \MtoI{\RmodT}.
\]
For cofibrant-fibrant $L$ and $M$, the mapping spaces $\RmodT (L,M)$
are Kan complexes, and it follows \HTT{1.1.5.9} that $\Rmod$ is an
$\infty$-category.

Once we have $\Rmod$, we can as in Definition \ref{inf-def-1} define
an $R$-line to be an $R$-module $L$ which admits a weak equivalence
\[
    L \to R,
\]
and let $\Rwe$ be the maximal $\infty$-subgroupoid of $\Rmod$ whose
objects are the cofibrant-fibrant $R$-lines.   Equivalently, we can
define $\RweT$ to be the subcategory of $\RmodT$ in
which the objects are $R$-lines $M,$ and in which the space of
morphisms from $L$ to $M$ is the subspace of $\RmodT (L,M)$
consisting of weak equivalences
\[
    L \xra{\heq} M.
\]
We can then set
\[
\MtoI{\Rwe}_{T} \eqdef \MtoI{\Rmod}_{T} \cap \Rwe_{T},
\]
and
\begin{align*}
      \Rwe  \iso \N \MtoI{\Rwe}_{T}.
\end{align*}
Explicitly, the set of $n$-simplices of
$\N\MtoI{\Rwe}_{T}$ is the set of simplicial functors
\[
    \mathfrak{C}[\Delta^n]\to \MtoI{\Rwe}_{T}.
\]
In particular,
\begin{enumerate}
\item  $\Rwe_{0}$ consists of fibrant-cofibrant $R$-modules which are $R$-lines;
\item  $\Rwe_{1}$ consists of the $R$-module weak equivalences
\[
     L \to M;
\]
\item $\Rwe_{2}$ consists of diagrams (not necessarily commutative)
\[
\xymatrix{
{L}
 \ar[r]^{f}
 \ar[dr]_{h}
&
{M}
 \ar[d]^{g}
\\
&
{N,}
}
\]
together with a path in $\RweT (L,N)$ from $gf$ to $h.$
\end{enumerate}
This is precisely the sort of data we proposed in \eqref{eq:37} as the
transitions functions for a bundle of free rank-one $R$-modules.

Recall from \HTT{1.2.2}  and \S\ref{sec:i-categories-i} that if
$Y$ is any  $\infty$-category and $a,b\in Y$, then  the space of maps
$Y (a,b) = \map_{Y} (a,b)$ from  $a$ to $b$ is a homotopy type; it can be
modeled as the $\infty$-groupoid $\hom_{Y} (a,b)$ which is
the fiber in the diagram of
simplicial sets
\begin{equation}\label{eq:51}
\begin{CD}
\hom_{Y} (a,b) @>>> Y^{\Delta^{1}} \\
@VVV @VVV \\
(a,b) @>>> Y^{\partial \Delta^{1}}.
\end{CD}
\end{equation}
If $Y=\N C$ is the nerve of a simplicial category $C$, then
\HTT{Theorem 1.1.5.12} implies that
\begin{equation}\label{eq:47}
    \hom_{Y} (a,b) \heq \hom_{C} (a,b).
\end{equation}

Taking $Y=\Rmod = \N\MtoI{\RmodT}$ and recalling that objects of $\Rmod$ are
cofibrant-fibrant $R$-modules, we see that $\Rmod
(L, M)$ has the homotopy type of the \emph{derived} space of $R$-module
maps from $M$ to $N.$  In view of the equivalence \eqref{eq:47}, we
shall adopt the convention that for arbitrary $R$-modules $M$ and $N$
(not necessarily cofibrant-fibrant), $\Rmod (M,N)$ will mean the
associated derived mapping space.

We note that $R$ itself may not be a
cofibrant-fibrant $R$-module, and we fix a
cofibrant-fibrant replacement $\Rcf.$  By construction,
$\Rwe$ is a connected $\infty$-groupoid, and so equivalent to
the full $\infty$-subgroupoid $B\Aut (\Rcf)$ with one object $\Rcf$
and morphisms
\[
 \Rwe (\Rcf,\Rcf) \heq \text{ (by convention) } \Rwe (R,R),
\]
the $\infty$-groupoid of derived self-equivalences of $\Rcf.$
This is the mapping space we studied in
\S\ref{sec:space-units-derived}, and in this setting Proposition
\ref{inf-t-pr-Rwe-and-GL} becomes the following.

\begin{Lemma} \label{inf-t-le-Rwe-and-GL-late}
Suppose that $R$ is a cofibrant and fibrant algebra in $\monspectra.$
Let $L$ and $M$ be objects of $\Rwe$.  The inclusion of
derived mapping spaces
\[
  \Rwe (R,R) = \Rwe (\Rcf,\Rcf) \rightarrow \Rmod (\Rcf, \Rcf) = \Rmod (R,R)
\]
is a model for the inclusion
\[
  \GL{R} \rightarrow \linf R.
\] 
\end{Lemma}

With the $\i$-category $\Rwe$ in place, we can develop the theory of
Thom spectra as in \S\ref{sec:glr-symm-mono} and
\ref{sec:thom-r-modules}; we do a little of this in
\S\ref{sec:i-cat-ii}.  For now,  we define a
\emph{trivialization} of an $R$-line $L$ to be an
equivalence of $R$-modules
\[
    L \rightarrow \Rcf,
\]
and we define the $\infty$-category $\Rtriv$ of trivialized $R$-lines
to be the slice category
\[
    \Rtriv  = \Rwe_{/\Rcf}.
\]

\begin{Proposition}\label{inf-ii-t-pr-egl-bgl-fib}
The forgetful map
\[
    \Rtriv \to \Rwe
\]
is a Kan fibration.  Indeed this fibration 
is a  model for
\[
   \ptspace \heq E\GL{R} \rightarrow B\GL{R}.
\]
\end{Proposition}

\begin{proof}
$\Rtriv$ is the simplicial set
of paths  ending at $\Rcf$ in the $\infty$-groupoid (Kan complex) $\Rwe.$
In particular  $\Rtriv$ is a contractible Kan complex, and the map
\[
   \Rtriv \to \Rwe
\]
is a Kan fibration;  by the discussion of \eqref{eq:51} and Lemma
\ref{inf-t-le-Rwe-and-GL-late}, the fiber is $\Aut (\Rcf) \heq \Rwe
(R,R) \heq \GL{R}.$
\end{proof}

To see that the resulting theory of Thom spectra agrees with the
theory developed in sections \ref{sec:glr-symm-mono} and
\ref{sec:thom-r-modules} amounts to comparing notions of $\Rmod.$  More
precisely, we'd like to know that the
$\i$-category of algebras in $\monspectra$ is equivalent
to $\Alg (\Stab (\Gpd_{\i})),$ and that if $R$ is an algebra in
$\monspectra$ corresponding to an algebra $R'$ in $\Stab (\Gpd_{\i}),$ then the
$\i$-categories of $R$-modules and $R'$-modules are equivalent.

The work of \cite{MR1806878,MR1922205,MR1819881,EKMM} establishes the
necessary Quillen equivalences for the model categories of
spectra of Lewis-May-Steinberger, $S$-modules of \cite{EKMM}, and
symmetric and orthogonal spectra \cite{MR1806878,MR1695653}.  In
\cite{DAGI,DAGII},  Lurie relates the $\i$-category
$\Stab (\Gpd_{\infty})$ to symmetric spectra.

Let $\symmspectra$ denote the category of symmetric spectra.
Example 8.21  of \cite{DAGIII} establishes an equivalence of symmetric
monoidal $\i$-categories
\begin{equation}\label{eq:16}
     \Stab (\Gpd_{\infty}) \heq \N\MtoI{\symmspectra},
\end{equation}
which by \cite[Example 1.6.14]{DAGII} induces an equivalence
\[
\theta: \Alg (\Stab (\Gpd_{\infty})) \heq
      \N \MtoI{\Alg (\symmspectra)}
\]
of $\i$-categories of algebras.

Let $R$ be an algebra in $\Alg (\Stab (\Gpd_{\infty})),$ so $\theta
(R)$ is a cofibrant-fibrant algebra in $\symmspectra.$  Then we have
$\Rmod$ as in \S\ref{sec:infinity-category-r}, and may construct
\[
    \Mod{\theta (R)} \eqdef \N \MtoI{\CatOf{$\theta (R)$--modules in
    $\symmspectra$}},
\]
as above.  According to
Theorem 2.5.3 of \cite{DAGII} we have an equivalence of
$\i$-categories
\begin{equation} \label{eq:61}
    \Rmod \heq \Mod{\theta (R)},
\end{equation}
as desired.

We summarize the discussion as we shall use it in
\S\ref{sec:comp-thom-spectra}.  Let $\monspectra$ be any simplicial
symmetric monoidal model category
of spectra, connected by a zig-zag of Quillen equivalences to
$\symmspectra.$  Suppose that
$\monspectra$ has a theory of algebras and modules similarly equivalent to that of
$\symmspectra.$  If follows that there are equivalences of $\i$-categories
\begin{align*}
     \N\MtoI{\monspectra}& \xra{\heq} \Stab (\Gpd_{\infty}) \\
     \N\MtoI{\Alg (\monspectra)} & \xra{\heq} \Alg (\Stab (\Gpd_{\infty})).
\end{align*}
Moreover, let $R'$ be a cofibrant and fibrant algebra in $\monspectra,$
and let $R$ a corresponding algebra in $\Stab (\Gpd_{\infty}).$  If
$X$ is a space, let $\Sing{X}$ denote its singular complex (an
$\infty$-groupoid).   Then according to Proposition
\ref{inf-ii-t-pr-egl-bgl-fib} and \eqref{eq:61},
we have equivalences
\[
        \Sing{(B\GL{R'})} \heq \Line{R'} \heq \Rwe,
\]
and so we may regard a map of spaces $f: X\to B\GL{R'}$ equivalently
as a map of $\i$-groupoids
\[
   \Sing{f}: \Sing{X} \to \Sing{(B\GL{R'})} \heq \Rwe.
\]

\section{Morita theory and Thom spectra}\label{sec:comp-thom-spectra}

In this section we interpret the construction of the Thom spectrum
from the perspective of Morita theory.  This viewpoint is implicit in
the ``algebraic'' definition of the Thom spectrum in
Section~\ref{sec:ainfty-thom-spectrum} as the derived smash product
\[Mf\eqdef \splus P \sma^{L}_{\splus \GL{R}} R,\]
where $R$ is the $\splus \GL{R}$-$R$ bimodule specified by the
canonical action of $\splus \GL{R}$ on $R$.  Recalling that the target
category of $R$-modules is stable, we can regard the Thom spectrum as
essentially given by a functor from (right) $\splus \GL{R}$-modules to
$R$-modules.

Now, roughly speaking, Morita theory (more precisely, the
Eilenberg-Watts theorem) implies that any continuous functor from
(right) $\splus \GL{R}$-modules to (right) $R$-modules which preserves
homotopy colimits and takes $\GL{R}$ to $R$ can be realized as
tensoring with an appropriate $(\splus \GL{R})$-$R$ bimodule.  In
particular, this tells us that the Thom spectrum functor is
characterized amongst such functors by the additional data of the
action of $\GL{R}$ on $R$.

Put another way, a continuous homotopy-colimit preserving functor
\[T \colon \Mod{\splus \GL{R}} \to \Mod{R}\]
can be determined to be the Thom spectrum functor by inspection of the
induced bimodule structure on $R \heq T(\splus \GL{R})$ obtained from
the map
\[
\xymatrix{
\splus \GL{R} \ar[r] & F_{\splus \GL{R}}(\splus \GL{R},\splus
\GL{R}) \ar[r]^T &  F_R(T(\splus
\GL{R}), T(\splus \GL{R})).
}
\]
By the adjunction characterizing $\GL{R}$, this
data is equivalent to specifying a multiplicative map $\GL{R} \to
\GL{R}$ (or equivalently a map $B\GL{R} \to B\GL{R}$).

Beyond its conceptual appeal, this viewpoint on the Thom spectrum
functor provides the basic framework for comparing the different
constructions which we have discussed in this paper (both with one
another, and with the ``neo-classical'' construction of Lewis and May
and the parametrized construction of May and Sigurdsson).  On the
face of it, the definition of the Thom spectrum in
Section~\ref{sec:units-via-infty} as the functor which sends
$f:X\to\Line{R}$ to
\[Mf\eqdef p_!f^*\mathscr{L},\]
where $\mathscr{L}$ is the universal parametrized spectrum over
$\Line{R}$, has a rather different (``geometric'') character.  In fact,
however, these definitions are essentially the same, under the
equivalence between parametrized spectra and modules.

After discussing the analogue of the classical Eilenberg-Watts theorem
in the context of ring spectra in Subsection~\ref{sec:morita}, we work
out a version of this classification in the setting of parametrized
spectra in Subsection~\ref{sec:colim-pres-funct}, using the techniques
made available by the $\i$-categorical perspective.  Here it turns out
to be possible to use the connections between modules, parametrized
spectra, and presheaves to obtain a classification of
colimit-preserving functors between $\i$-categories that specializes
to imply that a colimit-preserving functor from the infinity category
of spaces over $B\GL{R}$ to $R$-modules is determined by a map
$B\GL{R} \to B\GL{R}$.

This $\i$-category machinery then provides the foundation of the
comparison of the different definitions of Thom spectra in the paper.
As a consistency check, we verify that the classification obtained in
Subsection~\ref{sec:colim-pres-funct} agrees with that of the previous
subsection; the data of a map $B\GL{R} \to B\GL{R}$ determines the
same functor in each framework.  We then employ this axiomatic
perspective to show that the ``algebraic'' Thom spectrum functor
induces a functor on $\i$-categories which agrees with the
``geometric'' definition.  We also give a direct argument for the
comparison between the two definitions of the Thom spectrum functor;
although this argument appears to skirt the underlying Morita theory,
we believe it provides a useful concrete depiction of the situation.

Finally, the close relationship between our $\i$-categorical
construction of the Thom spectrum and the definition of May and
Sigurdsson \cite[23.7.1,23.7.4]{MR2271789} allows us to compare that
construction (and by extension the ``neo-classical'' Lewis-May
construction) to the ones discussed herein.

\subsection{The Eilenberg-Watts theorem for categories of module spectra}
\label{sec:morita}

The key underpinning of classical Morita theory is the Eilenberg-Watts
theorem, which for rings $A$ and $B$ establishes an equivalence
between the category of colimit-preserving functors $\Mod{A} \to
\Mod{B}$ and the category of $(A,B)$-bimodules.  The proof of the theorem
proceeds by observing that any functor $T\colon \Mod{A} \to \Mod{B}$
specifies a bimodule structure on $TA$ with the $A$-action given by
the composite
\[A \to F_A (A,A) \to F_B(TA, TA).\]
It is then straightforward to check that the functor $- \otimes_A TA$
is isomorphic to the functor $T$, using the fact that both of these
functors preserve colimits.

In this section, we discuss the generalization of this result to the
setting of categories of module spectra.  The situation here is more
complicated than in the setting of rings; for instance, it is
well-known that there are equivalences between categories of module
spectra which are not given by tensoring with bimodules, and there are
similar difficulties with the most general possible formulation of the
Eilenberg-Watts theorem.  However, much of the subtlety here comes
from the fact that unlike in the classical situation, compatibility
with the enrichment in spectra is not automatic (see for example the
excellent recent paper of Johnson \cite{Niles} for a comprehensive
discussion of the situation).  By assuming our functors are enriched,
we can recover a close analogue of the classical result.

Let $A$ and $B$ be (cofibrant) $S$-algebras, and let $T$ be an
enriched functor $T\colon \Mod{A} \to \Mod{B}$.  Specifically, we
assume that $T$ induces a map of function spectra $F_A(X,Y) \to
F_B(TX,TY)$. Furthermore, assume that $T$ preserves tensors (and in
particular is homotopy-preserving) and preserves homotopy colimits.
For instance, these conditions are satisfied if $T$ is a Quillen
left-adjoint.  The assumption that $T$ is homotopy-preserving in
particular means that $T$ preserves weak equivalences between
cofibrant objects and so admits a total left-derived functor $T^L
\colon \Ho{\Mod{A}} \to \Ho{\Mod{B}}$.  Furthermore, $T(A)$ is an
$A$-$B$ bimodule with the bimodule structure induced just as above.

Using the work of \cite{Niles} and an elaboration of
the arguments of \cite[4.1.2]{MR1928647} (see also
\cite[4.20]{MR2122154}) we now can prove the following Eilenberg-Watts
theorem in this setting.  We will work in the EKMM categories of
modules, so we can assume that all objects are fibrant.

\begin{Proposition}\label{prop:watts}
Given the hypotheses of the preceding discussion, there is a natural
isomorphism in $\Ho{\Mod{B}}$ between the total left-derived functor
$T^L(-)$ and the derived smash product $(-) \sma^L T(A)$, regarding
$T(A)$ as a bimodule as above.
\end{Proposition}

\begin{proof}
By continuity, there is a natural map of $B$-modules
\[(-) \sma_A T(A) \to T(-).\]
Let $T'$ denote a cofibrant replacement of $T(A)$ as an $A$-$B$
bimodule.  Since the functor $(-) \sma_A T'$ preserves weak
equivalences between cofibrant $A$-modules, there is a total
left-derived functor $(-) \sma_A^L T'$ which models $(-) \sma_A^L
T(A)$.  Thus, the composite
\[
(-) \sma_A T' \to (-) \sma_A T(A) \to T(-).
\]
descends to the homotopy category to produce a natural map
\[
(-) \sma_A^L T(A) \to T^L(-).
\]
The map is clearly an equivalence for the free $A$-module of rank
one; i.e. $A$.  Since both sides commute with homotopy colimits, we
can inductively deduce that the first map is an equivalence for all
cofibrant $A$-modules, and this implies that the map of derived
functors is an isomorphism.
\end{proof}

We now specialize to the case of Thom spectra.  Recall that for an
$S$-algebra $R$, $\GL{R}$ is an $\ainfty$ space and hence a monoid in
the category of $*$-modules. As such, there is a well-defined
notion of a $\GL{R}$-module (Definition~\ref{def-l1modules}) and
moreover we defined the bar constructions $\BL \GL{R}$ and $\EL\GL{R}$
such that the map $\EL\GL{R} \to \BL\GL{R}$ models the universal
quasifibration (Section~\ref{sec:principal-bundles}).

Given a fibration of $*$-modules $f \colon X \to \BL\GL{R}$, we first
took the pullback of the diagram,
\[
\xymatrix{
X \ar[r] & \BL\GL{R} & \ar[l] \EL\GL{R} \\
}
\]
to obtain a $\GL{R}$-module $P$.  This procedure defines a functor
from $*$-modules over $\BL\GL{R}$ to $\GL{R}$-modules; since we are
assuming $f$ is a fibration, we are computing the derived functor.
Next, since the category $\Mod{R}$ is stable, a continuous functor
from $\GL{R}$-modules to $\Mod{R}$ factors through the stabilization
$\Mod{\splus \GL{R}}$.  That is, up to equivalence a continuous
functor from $\GL{R}$-modules to $R$-modules is determined by a
continuous functor from $\splus \GL{R}$-modules to $R$-modules,
which is to say a $(\splus\GL{R},R)$-bimodule.

One might like to deduce a characterization of the
Thom spectrum functor as a functor from $\spaces / B\GL{R}$ to
$R$-modules from Proposition~\ref{prop:watts}.  However, it turns out
to be technically involved to state such a theorem precisely, because
the derived Thom spectrum functor as we have constructed it
``algebraically'' is presented as a composition of a right derived
functor (which is an equivalence) and a left derived functor.  We
remark that much of the technical difficulty in the neo-classical
theory of the Thom spectrum functor arises from the difficulties
involved in dealing with point-set models of such composites.

Fortunately, this is the kind of formal situation that the
$\i$-category framework handles well.  The functor $f\mapsto P$
induces an equivalence of homotopy categories; by Theorem
\ref{t-pr-princ-bundle-obstr-thy}, it is an equivalence of enriched
categories.  In particular, it induces an equivalence of
$\i$-categories
\[
    \N\MtoI{\EKMM_{*/\BL{\GL{R}}}} \heq \N \MtoI{\EKMM_{\GL{R}}};
\]
in the $\i$-category setting, this is an instance of
Proposition \ref{t-pr-G-space-space-over-bg}.
In light of the above discussion, the abstract characterization of the
Thom spectrum functor is immediate.

\begin{Proposition}\label{abstract-alg-thom}
Let $T:\Mod{\GL{R}}\to\Mod{R}$ be a continuous, colimit-preserving functor which sends $\GL{R}$ to an $R$-module $R'$ homotopy equivalent to $R$ in such a way that
$$
\GL{R}\simeq\mathrm{hEnd}_{\Mod{\GL{R}}}(\GL{R})\longrightarrow\mathrm{hEnd}_{\Mod{R}}(R')\simeq\mathrm{hEnd}_{\Mod{R}}(R)
$$
is homotopy equivalent to the inclusion $\GL{R}\simeq\mathrm{hAut}(R)\to\mathrm{hEnd}(R)$.
Then $T^L$, the left-derived functor of $T$, is homotopy equivalent to $\Sigma^\i_+(-)\land^L_{\Sigma^\i_+\GL{R}}R:\Mod{\GL{R}}\to\Mod{R}$.
\end{Proposition}

\begin{proof}
The stability of $\Mod{R}$ and Proposition \ref{prop:watts} together imply that $T^L$ is homotopy equivalent to $\Sigma^\i_+(-)\land^L_{\Sigma^\i_+\GL{R}} B$ for some $(\Sigma^\i_+\GL{R},R)$-bimodule $B$.
Since $T(\Sigma^\i_+\GL{R})\simeq R$, we must have $B\simeq R$; since the left action of $\GL{R}$ on itself induces (via the equivalence $R'\simeq R$) the canonical action of $\GL{R}$ on $R$, we conclude that $B\simeq R$ as $(\Sigma^\i_+\GL{R},R)$-bimodules.
\end{proof}

\subsection{Colimit-preserving functors from the $\i$-categorical perspective}

\label{sec:colim-pres-funct}

We now switch to the context of $\i$-categories.  In this section we
develop a general description of functors between $\i$-categories
which preserve colimits.  Specializing to module categories, we will
recover the version of the Eilenberg-Watts theorem which forms the
basis of the comparison of the various Thom spectrum functors
discussed in this paper.  The setting of $\i$-categories turns out to
be technically felicitous for performing the general comparisons we
need; our descriptions will arise from elementary considerations of
cocomplete $\i$-categories, the relationship between modules,
parametrized objects, and presheaves, and stabilization.

We begin by considering cocomplete $\i$-categories.  Let $\C$ be a
small $\i$-category, and consider the $\i$-topos
$$
\Pre(\C):=\text{Fun}(\C^{\op},\Gpd_\i)
$$
of presheaves of $\i$-groupoids on $\C$.  It comes equipped with a fully
faithful functor
\begin{equation} \label{eq:63}
\C\longrightarrow\Pre(\C),
\end{equation}
the Yoneda embedding, which (as we discussed in Remark
\ref{rem-yoneda-rigid}) is only really defined up to a contractible
space of choices.

The fact that $\Pre(\C)$ is equivalent to the nerve of the simplicial
model category of simplicial presheaves on $\mathfrak{C}[\C]$
(equipped, say, with the projective model structure) implies that
$\Pre(\C)$ is cocomplete as an $\i$-category.  Furthermore, just like
in ordinary category theory, the Yoneda embedding $\C\to\Pre(\C)$ is
in some sense initial among cocomplete $\i$-categories under $\C$.

\begin{Lemma}[\HTT{5.1.5.6}] For any cocomplete $\i$-category $\D$,
precomposition with the Yoneda embedding induces an equivalence of
$\i$-categories
\begin{equation}\label{eq:74}
\Fun^\mathrm{colim}(\Pre(\C),\D)\longrightarrow\Fun(\C,\D),
\end{equation}
where $\Fun^\mathrm{colim}(-,-)$ denotes the full $\i$-subcategory of
$\Fun(-,-)$ on the colimit-preserving functors.
\end{Lemma}

We shall be particularly interested in the case that $\C$ is an
$\i$-groupoid, so that
\[
    \Fun (\C^{op},\Gpd_{\i}) \heq \Gpd_{\i/\C},
\]
as in Remark \ref{rem-4}.   Suppose given a functor $g: \C\to \D.$  To
an $\i$-groupoid over $\C$
\[
    f: X\to \C,
\]
we can associate the colimit
\[
   \colim g\circ f: X \to \C \to \D.
\]
Insofar as we have a functorial model for this colimit (by making use, say, of functorial homotopy colimits in simplicial categories), then this procedure determines a colimit-preserving functor
\[
      \Gpd_{\i/\C} \heq \Fun (\C^{op},\Gpd_{\i}) \heq \Pre (\C)
      \to \D,
\]
with the property that its restriction along the Yoneda embedding
\[
    \C \to \Fun (\C^{op},\Gpd_{\i})
\]
is equivalent to $g$.
This is just a restatement of Proposition \ref{colimX_y=X}.

\begin{Corollary}  \label{t-co-yoneda-inverse}
Any inverse $\Fun(\C,\D)\to\Fun^{\colim}(\Pre(\C),\D)$ of the equivalence \eqref{eq:74}, given by restriction along the Yoneda
embedding $\C\to\Gpd_{\i/\C}$, sends $g$ to a colimit-preserving functor whose value on $f:X\to\C$ is equivalent to $\colim g\circ f$.\hfill$\Box$
\end{Corollary}

Lastly, we specialize to the context of stable categories.

\begin{Lemma}[\cite{DAGI}, Corollary 17.5] Let $\C$ and $\D$ be
presentable $\i$-categories such that $\D$ is stable.  Then
$$
\Omega^\i_-:\Stab(\C)\longrightarrow\C
$$
admits a left adjoint
$$
\Sigma^\i_+:\C\longrightarrow\Stab(\C),
$$
and precomposition with the $\Sigma^\i_+$ induces an equivalence of
$\i$-categories
$$
\Fun^{\colim}(\Stab(\C),\D)\longrightarrow\Fun^{\colim}(\C,\D).
$$
\end{Lemma}

Combining the universal properties of stabilization and the Yoneda
embedding, we obtain the following equivalence of $\i$-categories.

\begin{Corollary}
Let $\C$ and $\D$ be $\i$-categories such that $\D$ is stable and
cocomplete.  Then there are equivalences of $\i$-categories
$$
\Fun^{\colim}(\Stab(\Pre(\C)),\D)\heq\Fun^{\colim}(\Pre(\C),\D)\heq\Fun(\C,\D).
$$
\end{Corollary}

\begin{proof}
This follows from the last two lemmas.
\end{proof}

Now suppose that $\C$ and $\D$ have distinguished objects, given by
maps $*\to\C$ and $*\to\D$ from the trivial $\i$-category $*$.  Then
$\Pre(\C)$ and $\Stab(\Pre(\C))$ inherit distinguished objects via the
maps
$$
*\longrightarrow\C\overset{i}{\longrightarrow}\Pre(\C)\overset{\Sigma^\i}{\longrightarrow}\Stab(\Pre(\C)),
$$
where $i$ denotes the Yoneda embedding.  Note that the fiber sequence
$$
\Fun_{*/}(\C,\D)\longrightarrow\Fun(\C,\D)\longrightarrow\Fun(*,\D)\heq\D
$$
shows that the $\i$-category of pointed functors is equivalent to the
fiber of the evaluation map $\Fun(\C,\D)\to\D$ over the distinguished
object of $\D$.

\begin{Proposition}
Let $\C$ and $\D$ be $\i$-categories with distinguished objects such
that $\D$ is stable and cocomplete.  Then there are equivalences of
$\i$-categories
$$
\Fun_{*/}^{\colim}(\Stab(\Pre(\C)),\D)\heq\Fun_{*/}^{\colim}(\Pre(\C),\D)\heq\Fun_{*/}(\C,\D).
$$
\end{Proposition}

\begin{proof}
Take the fiber of
$\Fun^{\colim}(\Stab(\Pre(\C)),\D)\heq\Fun^{\colim}(\Pre(\C),\D)\heq\Fun(\C,\D)$
over $*\to\D$.
\end{proof}

\begin{Corollary}
Let $G$ be a group-like monoidal $\i$-groupoid $G$, let $BG$ be a
one-object $\i$-groupoid with $G\heq\Aut_{BG}(*)$, and let $\D$ be a
stable and cocomplete $\i$-category with a distinguished object $*$.
Then
$$
\Fun_{*/}^{\colim}(\Stab(\Pre(BG)),\D)\heq\Fun_{*/}^{\colim}(\Pre(BG),\D)\heq\Fun_{*/}(BG,\D)\heq\Fun(BG,B\Aut_\D(*));
$$
that is, specifying an action of $G$ on the distinguished object $*$
of $\D$ is equivalent to specifying a pointed colimit-preserving
functor from $\Pre(BG)$ (or its stabilization) to $\D$.
\end{Corollary}

\begin{proof}
A base-point preserving functor $BG\to\D$ necessarily factors through
$B\End_\D(*)$, and hence also through $B\Aut_\D(*)$ since $BG$ is an
$\i$-groupoid.
\end{proof}

Note that the $\i$-category $\Fun(BG,B\Aut_\D(*))$ is actually an
$\i$-groupoid, as $B\Aut_\D(*)$ is an $\i$-groupoid.

Putting this all together, consider the case in which the target
$\i$-category $\D$ is the $\i$-category of right $R$-modules for an
associative $S$-algebra $R$, pointed by the free rank one $R$-module
$R$.  Then $\Aut_\D(*)\heq GL_1 R$, so the space of pointed
colimit-preserving maps from the $\i$-category of spaces over $BG$ to the $\i$-category of $R$-modules is
equivalent to the space of monoidal maps from $G$ to $\GL{R}$.


\subsection{$\i$-categorical Thom spectra, revisited}
\label{sec:i-categorical-thom}

In this section, we return to the definition of Thom spectra from
\S\ref{sec:units-via-infty} and interpret that construction in light
of the work of the previous subsections.
Let $R$ be an algebra in $\Stab (\Gpd_{\i})$,
and form  the $\i$-categories $\Rmod$ and $\Rwe.$    Given a map of
$\i$-groupoids
\[
    f: X \to \Rwe,
\]
the Thom spectrum is the push-forward of the restriction to $X$ of the
tautological $R$-line bundle $\mathscr{L}$; that is,
$$
Mf=p_!f^*\mathscr{L}.
$$
Let $q$ denote the projection to the point from $\Line{R}$, so that
$p=q\circ f$ and we may therefore rewrite $Mf$ as
$$
Mf\heq q_!f_{!}f^*\mathscr{L}\heq q_!(\Sigma^\i_{\Line{R}} f_+\land_{\Line{R}}\mathscr{L}),
$$
where the second equivalence follows from formula \eqref{f_!f^*}.
This exhibits $M$ as a composite of left adjoints:
$\Sigma^\i_{\Line{R}}(-)_+$ followed by
$(-)\land_{\Line{R}}\mathscr{L}$ followed by $q_!$.
In particular, (the $\i$-categorical) $M$ itself is a left adjoint, so
it preserves ($\i$-categorical) colimits.

\begin{Proposition} \label{t-pr-yoneda-taut}
The restriction of $M: \Gpd_{\i/\Line{R}} \to \Rmod$ along the Yoneda
embedding
$$
\Line{R}\longrightarrow\Fun(\Line{R}^{op},\Gpd_\i)\heq\Gpd_{\i/\Line{R}}
$$
is equivalent to the tautological inclusion
\begin{equation}\label{canonicalinclusion}
\Line{R}\longrightarrow\Mod{R}.
\end{equation}
\end{Proposition}

\begin{proof}
Consider the colimit-preserving functor $\Gpd_{\i/\Line{R}}\to\Mod{R}$ induced by the canonical inclusion $\Line{R}\to\Mod{R}$.
As we explain in Corollary \ref{t-co-yoneda-inverse}, it sends
$X\to\Line{R}$ to the colimit of the composite
$X\to\Line{R}\to\Mod{R}$.  As we explain in Remark \ref{rem-2}, this
is equivalent to the Thom spectrum functor $M.$
\end{proof}

The following corollary is now an immediate consequence of the
analysis of the previous subsection.

\begin{Corollary} \label{t-co-axiom-geom-thom-functor}
A functor $\Gpd_{\i/\Line{R}}\to\Mod{R}$ is equivalent to $M$ if and only if it
preserves colimits and its restriction along the Yoneda embedding
$\Line{R}\to\Fun(\Line{R}^{op},\Gpd_\i)\heq\Gpd_{\i/\Line{R}}$ is
equivalent to the tautological inclusion \eqref{canonicalinclusion} of
$\Line{R}$ into $\Mod{R}$.
\end{Corollary}

%

\subsection{Comparing notions of Thom spectrum}
\label{sec:comparison}


In this section, we show that, on underlying $\i$-categories, the
``algebraic'' Thom $R$-module functor is equivalent to the
``geometric'' Thom spectrum functor via the characterization of
Corollary~\ref{t-co-axiom-geom-thom-functor}.  The proof of this
result also implies that the two versions of the Eilenberg-Watts
theorem we have discussed in this section agree up to natural weak
equivalence.

Let $\EKMM_{S}$ be the category of EKMM $S$-modules.  According to the
discussion in \S\ref{sec:i-category-r}, there is an equivalence of
$\i$-categories
\begin{equation} \label{eq:65}
     \N \MtoI{\EKMM_{S}} \heq \Stab (\Gpd_{\i})
\end{equation}
which induces an equivalence of $\i$-categories of algebras
\begin{equation} \label{eq:64}
     \N\MtoI{\Alg (\EKMM_{S})} \heq \Alg (\Stab (\Gpd_{\i})).
\end{equation}
Let $R$ be a cofibrant-fibrant EKMM $S$-algebra, and let $R'$ be
the corresponding algebra in $\Alg (\Stab (\Gpd_{\i})).$   The
equivalence \eqref{eq:65} induces an equivalence of $\i$-groupoids
\begin{equation} \label{eq:66}
   \N\MtoI{\Mod{R}} \heq \Mod{R'}.
\end{equation}
Proposition \ref{inf-ii-t-pr-egl-bgl-fib} gives an equivalence of
$\i$-groupoids
\begin{equation}\label{eq:67}
   \Sing{B\GL{R}} \heq     \N (\MtoI{(\Line{R})})
\end{equation}
and so putting \eqref{eq:66} and \eqref{eq:67} together we have an
equivalence of $\i$-categories
\[
     \N (\MtoI{(\TT/B\GL{R})}) \heq \Gpd_{\i/\Line{R'}}.
\]

\begin{Corollary} \label{t-co-thom-spectra-equivalent}
The functor
\[
     \Gpd_{\i/\Mod{R'}}\heq \N\MtoI{(\TT/B\GL{R})} \xra{\N M}
                             \N\MtoI{(\Mod{R})} \heq \Mod{R'},
\]
obtained by passing the Thom $R$-module functor of
\S\ref{sec:ainfty-thom-spectrum} though the indicated equivalences, is
equivalent to the Thom $R'$-module functor of
\S\ref{sec:units-via-infty}.
\end{Corollary}

\begin{proof}
Let $\C$ denote the topological category with a single object $*$ and
$\map_\C(*,*)=\GL{R}=\Aut_R(R^\circ)\simeq\Aut_{R'}(R')$.
Note that $\C$ is naturally a topological subcategory of
$\Mod{\GL{R}}$ (the full topological subcategory of $\GL{R}$) and by
definition a topological subcategory of $\Mod{R}$.  Note also that
\[
   \N\C \simeq B\Aut (R')\simeq \Line{R'}.
\]
According to Proposition \ref{abstract-alg-thom}, the continuous functor
$$
T^L:\Mod{\GL{R}}\longrightarrow\Mod{R}
$$
has the property that its restriction to $\C$ is equivalent to the
inclusion of the topological subcategory $\C\to\Mod{R}$.  Taking
simplicial nerves, and recalling that
\[
     \N \Mod{\GL{R}}^{\circ} \heq  \N \MtoI{(\TT/B\GL{R})} \heq \Fun
     (\N\C^{op},\Gpd_{\i})\Mod{\GL{R}},
\]
we see that
$$
\N T^L:\Fun(\N\C^{op},\Gpd_\i)\simeq\N\Mod{\GL{R}}^\circ\longrightarrow\N\Mod{R}^\circ\simeq\Mod{R'}
$$
is a colimit-preserving functor whose restriction along the Yoneda
embedding $\N\C\to\Fun(\N\C^{op},\Gpd_\i)\simeq\Gpd_{\i/\Line{R'}}$ is
equivalent to the inclusion of the $\i$-subcategory
$\N\C\simeq\Line{R'}\to\Mod{R'}$. It therefore follows from Proposition \ref{t-co-axiom-geom-thom-functor} that $\N T^L$ is equivalent to the ``geometric'' Thom spectrum functor of \S\ref{sec:units-via-infty}.
\end{proof}


\begin{Remark}
The argument also implies the following apparently more general result.
Recall from \ref{sec:colim-pres-funct} that {\em any} map $k: B\GL{R} \to B\GL{R}$ defines a functor from the $\i$-category of spaces over $B\GL{R}$ to the $\i$-category of $R$-modules, defined by sending $f: X \to B\GL{R}$ to the colimit of the composite
\begin{equation}\label{eq:co-param-mor-eq}
X \xra{f} B\GL{R} \xra{k} B\GL{R} \to \Rmod.
\end{equation}
On the other hand, according to
Proposition~\ref{prop:algebraic-colimit} below, we can describe the
derived smash product from section \ref{sec:morita} associated to $k$
as the colimit of the composite
\[
X \xra{f} B\GL{R} \xra{k} B\GL{R} \xra{\splus} \Mod{\splus \GL{R}}
\xra{(-) \sma_{\splus \GL{R}} R} \Rmod.
\]
Since both functors are given by the formula $M(k\circ f)$, the Thom
$R$-module of $f$ composed with $k$, we conclude that these two
procedures are equivalent for any $k$, not just the identity.
\end{Remark}

\subsection{The algebraic Thom spectrum functor as a colimit}
\label{sec:algebr-thom-spectr}

We sketch another approach to the comparison of the ``geometric'' and
``algebraic'' Thom spectrum functors.     This approach has the
advantage of giving a direct comparison of the two functors.  It has
the disadvantage that it does
not characterize the Thom spectrum functor among functors
\[
   \TT/B\GL{R} \to \Mod{R},
\]
and it does not exhibit the conceptual role played by Morita
theory.  Instead, it identifies both functors as colimits.  In this
sense it is a direct
generalization of the argument we gave in \eqref{eq:56}, for the case
of a discrete ring $R$ and a discrete space $X.$

Suppose that $R$ is a
monoid in $S$-modules.   Let $\Rmod$ be the associated $\i$-category of
$R$-modules, let $\Rwe$ be the the sub-$\i$-groupoid of
$R$-lines, and let $j: \Rwe \to \Rmod$ denote the inclusion.

Let $X$ be a space.  The ``geometric'' Thom spectrum functor sends a
map $f: \Sing{X}\to \Line{R}$ to
\[
    \colim (\Sing{X} \xra{\Sing{f}}\Line{R}  \xra{j}
    \Rmod).
\]

As in \S\ref{sec:thom-spectra-2}, let $G = (\GL{R})^{c}.$
In \S\ref{sec:i-category-r} we  showed that $\Sing{\BL{G}}  \heq \Rwe$.
But observe that we  also have a natural equivalence
\[
       \Sing{\BL{G}} \heq \Line{G}.
\]
That is, let $\Mod{G} = \N \MtoI{\EKMM_{G}}$  be the
$\i$-category of $G$-modules, and let $\Line{G}$ be the maximal
$\i$-groupoid generated by the $G$-lines, that is, cofibrant and
fibrant $G$-modules which admit a weak equivalence to $G.$
By construction, $\Line{G}$ is connected, and so equivalent to $B\Aut
(G)\heq \BL{G}.$

The construction of the ``algebraic Thom
spectrum'' begins by associating  to a fibration of $*$-modules $f: X\to
\BL{G}$ the $G$-module $P$
which is the pull-back in
\begin{equation} \label{eq:76}
\xymatrix{
{P}
 \ar[r]
 \ar[d]
&
{\EL{G}}
 \ar[d]\\
{X}
 \ar[r]
&
{\BL{G}.}
}
\end{equation}
The association $f\mapsto P$ defines a functor
\[
    \EKMM_{*/\BL{G}} \to \EKMM_{G}
\]
which induces an equivalence of homotopy categories; by Theorem
\ref{t-pr-princ-bundle-obstr-thy} it is an equivalence of enriched
homotopy categories.  As discussed in Proposition
\ref{t-pr-G-space-space-over-bg}, this corresponds to an equivalence
of $\i$-categories
\begin{equation} \label{eq:75}
      \Gpd_{\i/\Line{G}} \heq \Mod{G}.
\end{equation}
The key observation is the following.  Let $k: \Line{G}\to \Mod{G}$
denote the tautological inclusion.  To a map of $\i$-groupoids
\[
   f: X\to \Line{G},
\]
we can associate the $G$-module
\[
      P_{f} =  \colim (X \xra{f} \Line{G}\xra{k} \Mod{G} ).
\]

\begin{Lemma}
The functor $f \mapsto P_{f}$ gives the equivalence \eqref{eq:75}.
\end{Lemma}

In other words, if $f: X\to \BL{G}$ is a fibration of $*$-modules,
then we can form $P$ as in \eqref{eq:76}.  Alternatively, we can form
\[
   \Sing{f}: \Sing{X}\to \Sing{\BL{G}} \heq \Mod{G},
\]
and then form $P_{\Sing{f}} = \colim (k \Sing{f}),$ and then we have
an equivalence of $G$-modules
\[
      P_{\Sing{f}} \heq P.
\]
The proof of the Lemma follows the same lines as Corollary
\ref{t-pr-P-colim-X-to-BG}, which treats the case that $G$ is a
group-like monoid in $\i$-groupoids.

\begin{Proposition}\label{prop:algebraic-colimit}
Let $f: X\to \BL{G}$ be a fibration of $*$-modules.  The ``algebraic''
Thom spectrum functor sends $f$ to
\[
   \colim (\Sing{X} \xra{\Sing{f}} \Sing{\BL{G}} \heq \Line{G}
   \xra{k} \Mod{G} \xra{\splus} \Mod{\splus G} \xra{\Smash_{\splus G}
   R} \Rmod ).
\]
\end{Proposition}

\begin{proof}
By the Lemma, we have
\begin{equation}\label{eq:1}
   P \heq \colim (\Sing{X} \xra{\Sing{f}} \Sing{\BL{G}} \heq \Line{G}
   \xra{k} \Mod{G}),
\end{equation}
and so
\begin{align*}
    Mf = \splus P \Smash_{\splus G} R & \heq
 \splus \colim (k\Sing{f}) \Smash_{\splus G}R \\
& \heq \colim (\splus k \Sing{f}) \Smash_{\splus G} R \\
& \heq \colim (\splus k \Sing{f} \Smash_{\splus G} R).
\end{align*}
This last is the colimit in the statement of the Lemma.
\end{proof}

From this point of view, the coincidence of the two Thom spectrum
functors amounts to the fact that diagram
\[
\xymatrix{
{\BL{G}} \ar[r]
 \ar[dr]
&
\Line{G}
 \ar[d]_{\heq}^{\splus (\slot) \Smash_{\splus G} R}
 \ar[rr]^{k}
& &
{\Mod{G}}
 \ar[d]^{\splus (\slot) \Smash_{\splus G} R}
\\
& {\Rwe}
 \ar[rr]_{j}
& &
{\Rmod}
}
\]
evidently commutes.

\subsection{The ``neo-classical'' Thom spectrum functor}
\label{sec:neo-classical-thom}

In this final section we compare the Lewis-May operadic Thom
spectrum functor to the Thom spectrum functors discussed in this
paper.  Since the May-Sigurdsson construction of the Thom spectrum in
terms of a parametrized universal spectrum over $B\GL{S}$
\cite{MR2271789}[23.7.4] is easily seen to be equivalent to the
space-level Lewis-May description, this will imply that all of the
known descriptions of the Thom spectrum functor agree up to homotopy.
Our comparison proceeds by relating the Lewis-May model to the
quotient description of Proposition~\ref{quotient}.

We begin by briefly reviewing the Lewis-May construction of the Thom
spectrum functor; the interested reader is referred to Lewis' thesis,
published as Chapter IX of \cite{LMS:esht}, and the excellent
discussion in Chapter 22 of \cite{MR2271789} for more details and
proofs of the foundational results below.  Nonetheless, we have tried
to make our discussion relatively self-contained.

The starting point for the Lewis-May construction is an explicit
construction of $\GL{S}$ in terms of a diagrammatic model of infinite
loop spaces.  Let $\sI_c$ be the symmetric monoidal category of finite
or countably infinite dimensional real inner product spaces and linear
isometries.  Define an $\sI_c$-space to be a continuous functor from
$\sI_c$ to spaces.  The usual left Kan extension construction gives
the diagram category of $\sI_c$-spaces a symmetric monoidal
structure.  It turns out that monoids and commutative monoids for this
category model $\ainfty$ and $\einfty$ spaces; for technical felicity,
we focus attention on the commutative monoids which satisfy two
additional properties:
\begin{enumerate}
\item The map $T(V) \to T(W)$ associated to a linear isometry $V \to
  W$ is a homeomorphism onto a closed subspace.
\item Each $T(W)$ is the colimit of the $T(V)$, where $V$ runs over
  the finite dimensional subspaces of $W$ and the maps in the colimit
  system are restricted to the inclusions.
\end{enumerate}
Denote such a functor as an $\sI_c$-FCP (functor with cartesian
product) \cite[23.6.1]{MR2271789}; the requirement that $T$ be a
diagrammatic commutative monoid implies the existence of a ``Whitney
sum'' natural transformation $T(U) \times T(V) \to T(U \oplus V)$.
This terminology is of course deliberately evocative of the notion of
$FSP$ (functor with smash product), which is essentially an orthogonal
ring spectrum \cite{MR1806878}.

An $\sI_c$-FCP gives rise to an $\einfty$ space structured by the
linear isometries operad; specifically, $T(\R_\infty) = \colim_V T(V)$
is an $\sL$-space with the operad maps induced by the Whitney sum
\cite[1.9]{MQRT:ersers}, \cite[23.6.3]{MR2271789}.  In fact, as
alluded to above one can set up a Quillen equivalence between the
category of $\sI_c$-FCP's and the category of $\einfty$ spaces,
although we do not discuss this matter herein (see \cite{Lind} for a
nice treatment of this comparison).

Moving on, we now focus attention on the $\sI_c$-FCP specified by
taking $V \subset \R^\infty$ to the space of based homotopy
self-equivalences of $S^V$; this is classically denoted by $F(V)$.
Passing to the colimit over inclusions, $F(\R^\infty) = \colim_V F(V)$
becomes a $\L$-space which models $\GL{S}$ --- this is essentially one
of the original descriptions from \cite{MQRT:ersers}.  Furthermore,
since each $F(V)$ is a monoid, applying the two-side bar construction
levelwise yields an FCP specified by $V \mapsto BF(V)$; here $BF(V)$
denotes the bar construction $B(*,F(V),*)$, and the Whitney sum
transformation is defined using the homeomorphism $B(*, F(V), *)
\times B(*, F(W), *) \cong B(*, F(V) \times F(W), *)$.  The colimit
$BF(\R^\infty)$ provides a model for $B\GL{S}$.

Now, since $F(V)$ acts on $S^V$, we can also form the two-sided bar
construction $B(*,F(V),S^V)$, abbreviated $EF(V)$, and there is a universal quasifibration
\[
\pi_V \colon EF(V) = B(*,F(V),S^V) \longrightarrow B(*,F(V),*) = BF(V)
\]
which classifies
spherical fibrations with fiber $S^V$.  Given a map $X \to
BF(\R^\infty)$, by pulling back subspaces $BF(V) \subset
BF(\R^\infty)$ we get an induced filtration on $X$; denote the space
corresponding to pulling back along the inclusion of $V \in \R^\infty$
by $X(V)$ \cite[IX.3.1]{LMS:esht}.

Denote by $Z(V)$ the pullback
\[
\xymatrix{
X(V) \ar[r] & BF(V) & \ar[l] EF(V).
}
\]
The $V$th space of the Thom prespectrum is then obtained by taking the
Thom space of $Z(V) \to X(V)$, that is by collapsing out the section
induced from the base point inclusion $* \to S^V$; denote the resulting
prespectrum by $TF$ \cite[IX.3.2]{LMS:esht} (note that some work is
involved in checking that these spaces in fact assemble into a
prespectrum).

Next, we will verify that the prespectrum $TF$ associated to the
identity map on $BF(\R^\infty)$ is stably equivalent to the homotopy
quotient $S/\GL{S} \heq S/F(\R^\infty)$.  For a point-set description of
this homotopy quotient, notice that it follows from
Proposition~\ref{t-pr-star-mod-mbl-s-mod} that the category of EKMM
(commutative) $S$-algebras is tensored over (commutative) monoids in
$*$-modules: the tensor of a monoid in $*$-modules $M$ and an
$S$-algebra $A$ is $\splus M \sma A$, with multiplication
\[
(\splus M \sma A) \sma (\splus M \sma A) \cong (\splus M \sma \splus
M) \sma (A \sma A) \cong (\splus (M \boxtimes M)) \sma (A \sma A) \to
(\splus M) \sma A.
\]

Thus, we can model the homotopy quotient as a bar construction in the
category of (commutative) $S$-algebras.  However, we can also describe
the homotopy quotient as $\colim_V S/F(V)$, where here we use the
structure of $F(V)$ as a monoid acting on $S^V$.  It is this
``space-level'' description we will employ in the comparison below.

We find it most convenient to reinterpret the Lewis-May construction
in this situation, as follows:  The Thom space in this case is by
definition the cofiber $(EF(V),BF(V))$ of the inclusion $BF(V) \to
EF(V)$ induced from the base point inclusion $* \to S^V$.
Now,
$$
BF(V)\heq */F(V)$$
and similarly
$$
EF(V)\heq S^V/F(V).
$$
Hence the Thom space is likewise the cofiber $(S^V,*)/F(V)$ of the
inclusion $*\to S^V$, viewed as a {\em pointed} space.

More generally, we can regard the prespectrum $\{MF(V)\}$ as
equivalently described as
\[
MF(V)\eqdef S^V/F(V),
\]
the homotopy quotient of the {\em pointed} space $S^V$ by $F(V)$ via
the canonical action, with structure maps induced from the quotient
maps $S^V\to S^V/F(V)$ together with the pairings
$$
MF(V)\land MF(W)\heq S^V/F(V)\land S^W/F(W)\longrightarrow
S^{V\oplus W}/F(V)\times F(W)\longrightarrow S^{V\oplus W}/F(V \oplus
W),
$$
where $F(V)\times F(W)\to F(V \oplus W)$ is the Whitney sum map of
$F$.  It is straightforward to check that the structure maps in terms
of the bar construction described in \cite[IX.3.2]{LMS:esht} realize
these structure maps.

The associated spectrum $MF$ can then be identified as $\colim_V
S/F(V) \heq S/F(\R^\infty)$.  A key point is that the Thom spectrum
functor can be described as the colimit over shifts of the Thom spaces
\cite[IX.3.7,IX.4.4]{LMS:esht}:
\[
MF = \colim_V \Sigma^{-V} \Sigma^\i MF(V).
\]
Furthermore, using the bar construction we can see that the spectrum
quotient $(\Sigma^V S) / F(V)$ is equivalent to $\Sigma^\i S^V /
F(V)$.  Putting these facts together, we have the following chain of
equivalences:
\begin{eqnarray*}
MF & = & \colim_V\Sigma^{-V}\Sigma^\i
MF(V) = \colim_V\Sigma^{-V}\Sigma^\i S^V /F(V) \\
& \heq & \colim_V
\Sigma^{-V} (\Sigma^V S)/F(V) \heq \colim_V (\Sigma^{-V} \Sigma^V
S)/F(V) \htp S/F(\R^\infty).
\end{eqnarray*}

More generally, a slight elaboration of this argument implies the
following proposition.


\begin{Proposition}\label{L-Mquotient}
The Lewis-May Thom spectrum $MG$ associated to a group-like $A_\i$ map
$\varphi:G\to GL_1 S$ modeled by the map of $\sI_c$-FCPs $G \to F$ is
equivalent to the spectrum $S/G$, the homotopy quotient of the sphere
by the action of $\varphi$.
\end{Proposition}

Note that any $A_\i$ map $X \to F(\R^{\infty})$ can be rectified to a map of
$\sI_c$-FCPs $X' \to F$ \cite{Lind}.

\begin{Corollary}\label{t-co-LM-thom-comp}
Given a map of spaces $f:X\to BGL_1 S$, write $\Pi_\i Mf$ for the
stable $\i$-groupoid associated to the Lewis-May Thom spectrum $Mf$.
Then $\Pi_\i Mf\heq M\Pi_\i f$, where $$\Pi_\i f:\Pi_\i X\to\Pi_\i
BGL_1 S\heq 
\Line{S}$$ is the associated map of
$\i$-groupoids.
\end{Corollary}

\begin{proof}
A basic property of this (and any) Thom spectrum functor is that it
preserves colimits \cite[IX.4.3]{LMS:esht}.  Thus, we can assume that
$X$ is connected.  In this case, $X\heq BG$ for some group-like
$A_\i$ space $G$, and $f:BG\to BGL_1 S$ is the delooping of an $A_\i$
map $G\to GL_1 S$, so $Mf\heq S/G$ by Proposition \ref{L-Mquotient},
so $\Pi_\i Mf\heq M\Pi_\i f$ by Proposition \ref{quotient}.
\end{proof}

\appendix

\section{Spaces over $BG$ and $G$-spaces}

\subsection{The simplicial case}

\label{sec:princ-bundl-lift-new}

In this subsection we give a quick proof of the Quillen equivalence
between simplicial sets over $BG$ and $G$-simplicial sets (with
respect to the projective model structure) for a simplicial monoid
$G$, where $BG$ denotes the diagonal of the bisimplicial set obtained
by regarding $G$ as a simplicial category with one object.  This
result is folklore, and in the case of simplicial groups is an old
result of Dwyer and Kan \cite{MR789792}.  Nonetheless, since no proof
of the specific result we need appears in the literature, we include
one here.  Note that, throughout this subsection, all model categories
are simplicial model categories, and all morphisms of model categories
are simplicial Quillen adjunctions; furthermore, we will sometimes
refer to simplicial sets as spaces.

Write $\Set_{\Delta/BG}$ and $\Set^G_{\Delta}$ for the simplicial model categories of spaces over $BG$, equipped with the over-category model structure, and $G$-spaces, equipped with the projective model structure, respectively.
Interpolating between these is the simplicial model category $\Set^G_{\Delta/EG}$ of $G$-spaces over the $G$-space $EG$.
The projection $p:EG\to *$ to the terminal object induces an adjunction of simplicial model categories
$$
p_!:\Set^G_{\Delta/EG}\rightleftarrows\Set^G_{\Delta}:p^*
$$
(we always write the left adjoint on the left) which is in fact a Quillen equivalence, as $p$ is a weak equivalence of $G$-spaces.
Note that that pullback functor $p^*$ is also the left adjoint of a Quillen equivalence
$$
p^*:\Set^G_{\Delta}\rightleftarrows\Set^G_{\Delta/EG}:p_*,
$$
where the right adjoint $p_*$ sends $P\to EG$ to the $G$-space of sections $\map_{/EG}(EG,P)$.

Now, if we regard $BG$ as a trivial $G$-space and $q:EG\to BG$ as a $G$-map, then we obtain a similar base-change Quillen adjunction
$$
q_!:\Set^G_{\Delta/EG}\rightleftarrows\Set^G_{\Delta/BG}:q^*
$$
which clearly is not in general a Quillen equivalence; nevertheless, we claim that restriction along the unique simplicial monoid morphism $r:G\to *$ induces yet another Quillen adjunction
$$
r_!:\Set^G_{\Delta/BG}:\rightleftarrows\Set_{\Delta/BG}:r^*
$$
such that the composite Quillen adjunction
$$
r_!q_!:\Set^G_{\Delta/EG}\rightleftarrows\Set^G_{\Delta/BG}\rightleftarrows\Set_{\Delta/BG}:q^*r^*
$$
is a Quillen equivalence.
To see this, it's enough to show that the derived unit and counit transformations are equivalences.
But actually, more is true: $r_!q_!$ and $q^*r^*$ are inverse isomorphisms.
Indeed,
$$r_!q_!(P\to EG)=(P/G\to EG/G=BG)$$
and
$$q^*r^*(X\to BG)=(X\times_{BG}EG\to BG\times_{BG}EG=BG),$$
so the fact that colimits are compatible with base-change implies that
$$
q^*r^*r_!q_!(P\to EG)=(P/G\times_{BG}EG\to BG\times_{BG}EG)\cong(P\to EG)
$$
and
$$
r_!q_!q^*r^*(X\to BG)=(X\times_{BG}EG/G\to BG\times_{BG}EG/G)\cong(X\to BG).
$$
Combining this with the Quillen pair $(p^*,p_*)$ above we obtain the following result.

\begin{Proposition}
Let $G$ be a
simplicial monoid.
Then the composite adjunction
$$
r_!q_!p^*:\Set^G_\Delta\rightleftarrows\Set^G_{\Delta/EG}\rightleftarrows\Set^G_{\Delta/BG}\rightleftarrows\Set_{\Delta/BG}:p_*q^*r^*
$$
is a Quillen equivalence.
\end{Proposition}


%

\subsection{Principal bundles and lifting for group-like topological monoids}
\label{sec:princ-bundl-lift}

In this section we give a proof of the analogue of Theorem
\ref{t-pr-princ-bundle-obstr-thy} in the context of topological
monoids.  Our proof depends on a comparison between the homotopy
theory of free $G$-spaces and spaces over $BG$; it is an immediate
consequence of the recent work of Shulman \cite{MR2380927}, who uses
the technical foundations laid in the very careful development of
May-Sigurdsson \cite{MR2271789}.

We can assume without loss of generality that $G$ is a group-like
topological monoid with the homotopy type of a $CW$-complex and a
nondegenerate base point (i.e. the inclusion map $* \to G$ is a
Hurewicz cofibration); we refer to such a $G$ as satisfying the
standard hypotheses.  Associated to $G$ is the free right $G$-space
$EG$ defined as the two-sided bar construction $B(*,G,G)$.  There is a
projection map $\pi \colon B(*,G,G) \to B(*,G,*) = BG$ which is a
quasifibration (and a fibration if $G$ is in fact a topological
group) \cite[7.6]{May:class}.

Mimicking the notation of the previous subsection, let $G \aU$ denote
the category of $G$-spaces, let $\aU / BG$ denote the category of
spaces over $BG$, and let $G \aU / EG$ denote the category of
$G$-spaces over $EG$.  Again we have an adjoint
pair $(p_!, p^*)$ of functors
\[
p_! \colon G \aU / EG \rightleftarrows G \aU \colon p^*,
\]
where $p^*$ is the pullback and $p_!$ is simply given by composition.
There is also an adjoint pair $(p^*, p_*)$, where $p_*$ takes $X \to
EG$ to the $G$-space of sections $\map_{\aU / EG}(EG,X)$.

Regarding $BG$ as a trivial $G$-space, the projection $\pi \colon EG \to BG$ is a $G$-map and so there is a similar pair of adjoint functors $(q_!, q_*)$
\[
q_! \colon G\aU / EG \rightleftarrows G \aU / BG \colon q_*.
\]
Finally, associated to the map of monoids $r \colon G \to *$ there is an
adjoint pair $(r_!, r^*)$ of ``change of monoids'' functors
\[
r_! \colon G\aU / BG \rightleftarrows \aU / BG \colon r^*,
\]
where $r^*$ is the pullback (which assigns the trivial action) and
$r_!$ takes $X$ to $X \times_G *$.

Putting this together, we obtain a composite adjunction
\[
r_!q_! \colon G\aU / EG \rightleftarrows G\aU / BG \rightleftarrows \aU / BG
\colon q^*r^*,
\]
and combining this with the adjunction $(p^*,p_*)$ above we
obtain the composite adjunction
\[
r_!q_!p^* \colon G\aU \rightleftarrows G\aU / EG \rightleftarrows G\aU
/ BG \rightleftarrows \aU / BG \colon p_*q^*r^*.
\]

The homotopy theory of free $G$-spaces can be encoded in the model
structure on $G\aU$ in which the weak equivalences and fibrations are
detected via the forgetful functor to $\aU$.  For $\aU / BG$, we use
the $m$-model structure considered in \cite{MR2271789} based on work of
Cole on mixed model structures \cite{MR2203016}.  Here the weak
equivalences are the maps which induce a weak equivalence after
forgetting to $\aU$ and the fibrations are the Hurewicz fibration.
Cofibrant objects have the homotopy type of $CW$-complexes.

In this setting, and under the standard hypotheses on $G$, Shulman
proves that $(r_! q_! p^*, p_* q^* r^*)$ is a Quillen equivalence
\cite[8.5]{MR2380927}.  An immediate consequence of this is the desired
comparison of mapping spaces.

\begin{Corollary}
There are equivalences of derived mapping spaces
\[
\map_{\aU /BG}(X,Y)\heq\map_{G\aU}(p_*q^*r^*X,p_*q^*r^*Y)
\]
and
\[
\map_{G\aU}(P,Q)\heq\map_{\aU/BG}(r_!q_!p^*P,r_!q_!p^*Q)
\]
which are natural in spaces $X$ and $Y$ over $BG$ and $G$-spaces $P$
and $Q$.
\end{Corollary}

\begin{Corollary}\label{corliftmon}
There are equivalences of derived mapping spaces
\[
\map_{\aU
  /BG}(X,EG)\heq\map_{G\aU} (X\times_{BG}EG,EG\times_{BG}EG)\heq
\map_{G\aU} (X\times_{BG}EG,G),
\]
natural in spaces $X$ over $BG$.
\end{Corollary}

\begin{proof}
We simply need to know that $X\times_{BG}EG$ and $G$ are weakly
equivalent to $p_*q^*r^*X$ and $p_*q^*r^*EG$, respectively.
But $$X\times_{BG}EG\heq\map_{/EG}(EG,X\times_{BG}EG)=p_*q^*r^*X,$$
and the equivalence for $G$ follows from the equivalence $G\heq
EG\times G\cong EG\times_{BG}EG$.
\end{proof}

\section{$\infty$-categories and symmetric monoidal model categories of spectra}

\label{sec:i-cat-ii}

In order to give a self-contained treatment of Thom spectra and
orientations in the setting of $\i$-categories, in
\S\ref{sec:units-via-infty} we used the symmetric monoidal
$\i$-category of spectra developed in \cite{DAGI,DAGII,DAGIII}.
In this section, we show that one can give an account of the obstruction
theory for orientations using only the techniques developed in
\cite{math.CT/0608040} and general facts about associative ring
spectra and their categories of modules.    We hope that our treatment of
Thom spectra and orientations from this point of view will serve as a
useful invitation to $\i$-categories, for those familiar with
symmetric monoidal model categories of spectra.   With this in mind,
we have made this section somewhat more self-contained than strictly
necessary, at the price of some redundancy with \S\ref{sec:units-via-infty}.

\subsection{Symmetric monoidal categories of $R$-modules}
\label{sec:infty-ii-category-r}

Let $\monspectra$ be a symmetric monoidal simplicial model category of
spectra such as the $S$-modules of \cite{EKMM}, or the symmetric
spectra of \cite{MR1695653}, and let $R$ be a cofibrant and fibrant
associative algebra in $\monspectra.$    We build the $\i$-category
$\Rwe$ as in \S\ref{sec:i-category-r}.

Namely, let $\RmodT$ be the
simplicial model category of $R$-modules, and let $\MtoI{\RmodT}$  be the
full subcategory of $\RmodT$ consisting of cofibrant-fibrant
$R$-modules.  Now take the simplicial nerve \HTT{1.1.5.5}
to obtain the $\infty$-category
\[
      \Rmod \eqdef \N \MtoI{\RmodT}.
\]
For cofibrant-fibrant $L$ and $M$, the mapping spaces $\RmodT (L,M)$
are Kan complexes, and it follows \HTT{1.1.5.9} that $\Rmod$ is an
$\infty$-category.

As in Definition \ref{inf-def-1}, define
an $R$-line to be an $R$-module $L$ which admits a weak equivalence
\[
    L \to R,
\]
and let $\Rwe$ be the full $\infty$-subgroupoid of $\Rmod$ whose
objects are the $R$-lines.
As discussed in \S\ref{sec:i-category-r},
\begin{enumerate}
\item  $\Rwe_{0}$ consists of fibrant-cofibrant $R$-modules which are $R$-lines;
\item  $\Rwe_{1}$ consists of the $R$-module weak equivalences
\[
     L \to M;
\]
\item $\Rwe_{2}$ consists of diagrams (not necessarily commutative)
\[
\xymatrix{
{L}
 \ar[r]^{f}
 \ar[dr]_{h}
&
{M}
 \ar[d]^{g}
\\
&
{N,}
}
\]
together with a path in $\RmodT (L,N)$ from $gf$ to $h,$
\end{enumerate}
and so forth.  This is precisely the sort of data we proposed in \eqref{eq:37} as the
transitions functions for a bundle of free rank-one $R$-modules.

Fix a cofibrant-fibrant $R$-line $\Rcf.$ Let
\[
   \Aut (\Rcf) = \Rwe (\Rcf,\Rcf) \subset \Rmod (\Rcf,\Rcf),
\]
and let $B\Aut (\Rcf)$ be the full sub-$\i$-category of $\Rwe$ on the
single object $\Rcf.$  Since by construction $\Rwe$ is a connected
$\i$-groupoid, the inclusion
\[
     B\Aut (\Rcf) \to \Rwe
\]
is an equivalence.  Define
a \emph{trivialization} of an $R$-line $L$ to be an
equivalence of $R$-modules
\[
    L \rightarrow \Rcf,
\]
and let the $\infty$-category $\Rtriv$ of trivialized $R$-lines
be the slice category
\[
    \Rtriv  = \Rwe_{/\Rcf}.
\]
As discussed in Proposition \ref{inf-ii-t-pr-egl-bgl-fib}, the
forgetful map
\[
   \Rtriv  \to \Rwe
\]
is a Kan fibration, which is a model in this setting for $E\GL{R} \to
B\GL{R}.$

\subsection{Parametrized $R$-lines and Thom spectra}
\label{sec:parametrized-r-lines}

Now let $X$ be a Kan complex.
Suppose that we are given a map of
$\infty$-categories (i.e. of simplicial sets)
\[
    f: X \to \Rwe.
\]
Thus $f$ assigns
\begin{enumerate}
\item to each $0$-simplex $p\in X$ an $R$-line $f (p)$;
\item to each path $\gamma$ from $p$ to $q$ a weak equivalence of
$R$-modules
\[
     f (\gamma): f (p) \heq f (q);
\]
\item to each $2$-simplex $\sigma: \Delta^{2}\to X,$ say
\[
\xymatrix{
{p}
 \ar[d]_{\sigma_{01}}
 \ar[dr]^{\sigma_{02}}
\\
{q}
 \ar[r]_{\sigma_{12}}&
{r,}
}
\]
a path $f (\sigma)$ in $\Rwe (f (p), f (r))$ from $f (\sigma_{12})f
(\sigma_{01})$ to $f (\sigma_{02})$;
\end{enumerate}
and so forth.

This illustrates nicely the idea, developed in
\S\ref{sec:units-via-infty}, that the $\infty$-category
$\Rwe^{X}=\Fun (X,\Rwe)$ is a model for the $\infty$-category of $R$-lines
parametrized by $X$, and that $f$ corresponds to a bundle
\[
     \mathscr{L} \to X
\]
of $R$-lines over $X.$

For now, consider the diagram
\[
\xymatrix{
{X}
  \ar[r]^-{f}
  \ar[d]_-{p}
&
{\Rwe}
 \ar[r]^-{j}
&
{\Rmod.}
\\
{\ptspace}
 \ar@{-->}[urr]_-{L_{p} (fj)}
}
\]

\begin{Lemma} \label{inf-t-le-colimit-exists}
The map of $\infty$ categories $jf$ admits a colimit; equivalently,
there is a left Kan extension of $jf$ along $p.$  (See
\HTT{1.2.13,4.3}).
\end{Lemma}

\begin{Remark}  We can't take the colimit in $\Rwe,$ since $\Rwe$ does
not have colimits (it doesn't even have sums).
\end{Remark}

\begin{proof}
Again we take $\Rmod  = \N\MtoI{\Rmod}_{T}$ as our $\i$-category of
$R$-modules.   According to \HTT{4.2.3.14}, there is an ordinary
category $I$ and a cofinal map
\[
   k: \N (I) \to X.
\]
Consider
\[
     \N (I) \xra{k} X \xra{f} \Rwe \xra{j} \Rmod.
\]
According to \HTT{4.1.1.8}, a colimit of $jfk$ is the same thing as a
colimit of $jf.$
According to \HTT{4.2.4.1}, a colimit of $jfk$ is the same thing as a
homotopy colimit of
\[
     I \to \RmodT.
\]
This homotopy colimit exists because $\RmodT$ is a simplicial model
category.
\end{proof}

\begin{Definition}  \label{def-3}
The \emph{Thom spectrum} associated to $f$ is the
$R$-module spectrum
\[
  Mf = \colim jf = \Ho L_{p}jf.
\]
\end{Definition}

For example, let $MR$ be the Thom spectrum associated to identity map
of $\Rwe,$ or equivalently and more suggestively, to the inclusion
\[
   B\Aut (\Rcf) \xra{\heq} \Rwe.
\]

\begin{Proposition}\label{t-pr-univquotient}
$$
MR \heq \Rcf/\Aut (\Rcf).
$$
\end{Proposition}

\begin{proof}
By definition, $\Rcf/\Aut (\Rcf)$ is the colimit of the composite functor
$$
B\Aut (\Rcf) \heq \Rwe \to \Rmod.
$$
But this is just the standard construction of the (homotopy) quotient.
\end{proof}

Before discussing orientations, we describe explicitly the mapping
property of this colimit.  Note that a functor
$\ptspace\to \Rwe$ is just a choice of $R$-line, and we write $\iota$
for the composition
\[
   \iota: X \xra{p} \ptspace \xra{\Rcf} \Rwe.
\]

\begin{Lemma}  \label{inf-t-le-homotopy-adjunction}
There is an adjunction
\begin{equation} \label{inf-eq:1}
    \Rmod (Mf,\Rcf) \heq
    \Rmod^{X} (jf,j\iota)
\end{equation}
\end{Lemma}

\begin{Remark}
The  ``adjunction'' \eqref{inf-eq:1} is the
$\infty$-category analog of the usual property of the colimit.
It is a homotopy equivalence between the indicated mapping spaces.
\end{Remark}


\begin{proof}
According to \HTT{4.3.3.7}, the left Kan extension
\[
p_{!}: \Rmod^{X} \rightarrow \Rmod
\]
is a left adjoint of the constant functor $p^{*}: \Rmod \rightarrow
\Rmod^{X}.$  By \HTT{5.2.2.7}, this means in particular that for any
object $g$ of
$\Rmod^{X}$ and $Y$ of $\Rmod,$ there is a natural homotopy
equivalence
\[
    \Rmod (p_{!}g,Y) \heq \Rmod^{X} (g,p^{*}Y).
\]
Taking $g=jf$ and $Y=\Rcf$ so $p_{!}g=Mf$ and $p^{*}Y =j\iota$
gives the result.
\end{proof}

\subsection{The space of orientations}

With these in place, one can analyze the space of orientations in a
straightforward manner, just as in \S\ref{sec:units-via-infty}.
Notice that we have a natural map
\[
    \Rwe^{X} (f,\iota) \to \Rmod^{X} (jf,j\iota);
\]
by definition, this is just the inclusion of a
set of path components.  The following Definition is equivalent to
Definition \ref{inf-def-2}, except that we need the material in
\S\ref{sec:units-via-infty} to justify the notation
\[
     \Rwe^{X} (f,\iota) \heq \Line{R_X}(f^*\mathscr{L},p^{*}\Rcf),
\]
and so forth.

\begin{Definition} \label{inf-ii-def-2}
The space of \emph{orientations} of $Mf$ is the pull-back
\begin{equation}\label{inf-ii-eq:2}
\begin{CD}
\Rorient (Mf,R) @>>> \Rmod (Mf,R) \\
@V\heq VV @VV \heq V \\
\Rwe^{X} (f,\iota) @>>> \Rmod^{X} (jf,j\iota),
\end{CD}
\end{equation}
where the right vertical equivalence is the adjunction of Lemma
\ref{inf-t-le-homotopy-adjunction}.
\end{Definition}

To describe the obstruction theory associated to $\infty$-groupoid
$\Rorient (Mf,R),$ let $\map_{f}
(X,\Rtriv)$ be the simplicial set which is the pull-back in the diagram
\[
\begin{CD}
\map_{f} (X,\Rtriv) @>>> \map (X,\Rtriv) \\
@VVV @VVV \\
\{f \} @>>> \map (X,\Rwe).
\end{CD}
\]
That is, $\map_{f} (X,\Rtriv)$ is the mapping simplicial set of
lifts in the diagram
\begin{equation} \label{inf-ii-eq:26}
\xymatrix{
&
{\Rtriv}
 \ar[d] \\
{X}
 \ar@{-->}[ur]
 \ar[r]_-{f}
&
{\Rwe.}
}
\end{equation}
We recapitulate in the current setting the statement and proof of
Theorem \ref{t-th-or-thy-infty-lifting}.

\begin{Theorem} \label{t-th-inf-ii-or-thy-infty-lifting}
Let $f: X \to \Rwe$ be a map, and let $Mf$ be the associated
$R$-module Thom
spectrum.   Then there is an equivalence
\[
   \map_{f} (X,\Rtriv) \heq
   \Rwe^{X} (f,\iota).
\]
\end{Theorem}

\begin{proof}
Let $Z=\Rwe^{X}.$  We have a fibration
\[
\begin{CD}
  Z (f,\iota) @>>>  Z^{\Delta^{1}} \\
   @VVV  @VVV \\
\{(f,\iota)\}@>>>  Z^{\partial \Delta^{1}}.
\end{CD}
\]
Pulling back along the inclusion of $Z\times \iota$ gives the
fibration
\[
\begin{CD}
  Z (f,\iota) @>>>  (Z,\iota)^{(\Delta^{1},1)}\\
   @VVV  @VVV \\
\{f\}   @>>>  Z.
\end{CD}
\]
The adjunction (of mapping simplicial sets)
\[
     \map (\Delta^{1},\map (X,\Rwe )) \iso \map (X,\map (\Delta^{1},\Rwe ))
\]
identifies
\[
 (Z,\iota)^{(\Delta^{1},1)} \iso \map (X,\Rtriv),
\]
and so this fibration becomes
\[
\begin{CD}
  Z (f,\iota) @>>>  \map (X,\Rtriv)\\
   @VVV  @VVV \\
\{f\}   @>>>  Z,
\end{CD}
\]
as required.
\end{proof}

\def\cprime{$'$}

\end{document}